\documentclass[oneside,a4paper]{amsart}

\usepackage[OT2,T1]{fontenc}
\usepackage[utf8]{inputenc}
\usepackage{amsmath, amsthm, amsfonts, amssymb}
\newcommand\textcyr[1]{{\fontencoding{OT2}\selectfont #1}}

\usepackage[dvipsnames]{xcolor} % colourfultext
\usepackage{url}
\usepackage{nicefrac}
\usepackage[all]{xy}
\usepackage{placeins} % barriers for floats
\usepackage{colonequals}

\usepackage{euscript} % Lurie-type Euler script
\usepackage{bbm} % mathbb numbers
\usepackage{enumitem} % fancy enumerate

\usepackage{hyperref}
\hypersetup{
    colorlinks, linkcolor=NavyBlue,
    citecolor=OliveGreen, urlcolor=RedOrange
}
\usepackage[nameinlink,capitalise,noabbrev]{cleveref} % \Cref{prop:1.4} prints ``Proposition 1.4''
\usepackage{todonotes}
\setuptodonotes{fancyline, color=green!30}
\usepackage{faktor}

\usepackage{tikz-cd}
\usepackage{tikz}
\usetikzlibrary{arrows,backgrounds,
    decorations.pathreplacing,
    decorations.pathmorphing
}
\usetikzlibrary{matrix,arrows}
\usepackage{colonequals}
\theoremstyle{plain}

\numberwithin{equation}{subsection}

\theoremstyle{definition} % I've moved this up here since I have trouble reading large blocks of italic text. 
\newtheorem{theorem}[equation]{Theorem}
\newtheorem{lemma}[equation]{Lemma}
\newtheorem{proposition}[equation]{Proposition}
\newtheorem{prop}[equation]{Proposition}
\newtheorem{corollary}[equation]{Corollary}
\newtheorem{cor}[equation]{Corollary}

\newtheorem{question}[equation]{Question}

\newtheorem{construction}[equation]{Construction}

\newtheorem{definition}[equation]{Definition}
\newtheorem{defn}[equation]{Definition}
\newtheorem{example}[equation]{Example}
\newtheorem{ex}[equation]{Example}
\newtheorem{notation}[equation]{Notation}

\newtheorem{recollection}[equation]{Recollection}
\newtheorem{remark}[equation]{Remark}
\newtheorem{rmk}[equation]{Remark}
\newtheorem{warning}[equation]{Warning}

\newtheorem*{remark*}{Remark}
\newtheorem*{terminology*}{Terminology}
\newtheorem*{interpretation*}{Interpretation}
\newtheorem*{definition*}{Definition}
\newtheorem*{conjecture*}{Conjecture}
\newtheorem*{notation*}{Notation}
\newtheorem*{convention*}{Convention}

\newcommand{\euscr}[1]{\EuScript{#1}} % Euler script
\newcommand{\acat}{\euscr{A}} % category A in Euler script
\newcommand{\bcat}{\euscr{B}} % category A in Euler script
\newcommand{\ccat}{\euscr{C}} % category C in Euler script 
\newcommand{\dcat}{\euscr{D}} % category D in Euler script
\newcommand{\ecat}{\euscr{E}} % category E in Euler script 
\newcommand{\pcat}{\euscr{P}} % category P in Euler script
\newcommand{\qcat}{\euscr{Q}} % category Q in Euler script
\newcommand{\rcat}{\euscr{R}}
\newcommand{\jcat}{\euscr{J}}
\newcommand{\tcat}{\euscr{T}}
\newcommand{\mcat}{\euscr{M}}

\newcommand{\xcat}{\euscr{X}}
\newcommand{\ycat}{\euscr{Y}}
\newcommand{\lcat}{\euscr{L}}
\newcommand{\gcat}{\euscr{G}}

\newcommand{\Xss}{X_\bullet}
\newcommand{\Yss}{Y_\bullet}
\newcommand{\Zss}{Z_\bullet}
\newcommand{\Pss}{P_\bullet}
\newcommand{\Qss}{Q_\bullet}

\newcommand{\Kcs}{K^\bullet}

\newcommand{\calO}{\euscr{O}}

\newcommand{\calL}{\euscr{L}}

\newcommand{\bbF}{\mathbb{F}}
\newcommand{\bbC}{\mathbb{C}}

\newcommand{\bbT}{\mathbb{T}}

\newcommand{\bfE}{\mathbf{E}}
\newcommand{\bfF}{\mathbf{F}}

\newcommand{\frakm}{\mathfrak{m}}

\DeclareMathOperator{\MU}{MU}
\DeclareMathOperator{\Mod}{Mod}
\DeclareMathOperator{\map}{map}
\DeclareMathOperator{\fun}{fun}
\DeclareMathOperator{\Endo}{end}
\DeclareMathOperator{\Map}{Map}
\DeclareMathOperator{\Fun}{Fun}

\DeclareMathOperator{\Alg}{Alg}
\DeclareMathOperator{\CAlg}{CAlg}
\DeclareMathOperator{\Lie}{Lie}

\DeclareMathOperator{\End}{End}

\DeclareMathOperator{\Tor}{Tor}
\DeclareMathOperator{\Hom}{Hom}

\DeclareMathOperator{\im}{im}

\DeclareMathOperator*{\colim}{colim}
\DeclareMathOperator{\Tot}{Tot}

\newcommand{\spaces}{\euscr{S}\mathrm{pc}} % the category of spaces
\newcommand{\pretheories}{\widehat{\euscr{T}\mathrm{hry}}}
\newcommand{\spectra}{\euscr{S}\mathrm{p}} % the category of spectra
\newcommand{\sets}{\euscr{S}\mathrm{et}} % the category of sets
\newcommand{\simplicialspaces}{\Fun(\Delta^\op,\spaces)}

\newcommand{\malcevtheories}{\euscr{M}\mathrm{alc}} % Category of Malcev theories and homomorphisms
\newcommand{\synspectra}{\euscr{S}\mathrm{yn}} % synthetic spectra
\newcommand{\largecatinfty}{\widehat{\euscr{C}\mathrm{at}}_{\infty}} % category of infty categories
\newcommand{\PrL}{\mathrm{Pr}^{L}} % Category of presentable categories and left adjoints
\newcommand{\lmodoperad}{\mathcal{LM}\mathrm{od}}
\newcommand{\cfrees}{\euscr{L}_0}
\newcommand{\frees}{\euscr{F}}
\newcommand{\lfrees}{\euscr{L}}
\newcommand{\spheres}{\euscr{G}}

\newcommand{\presheaves}{\euscr{P}\mathrm{sh}}

\newcommand{\smallpresheaves}{\euscr{P}\mathrm{sh}}

\newcommand{\largemodels}{\Model^{\nb}}

\newcommand{\LMod}{\mathrm{L}\euscr{M}\mathrm{od}}

\newcommand{\Model}{\euscr{M}\mathrm{odel}}

\newcommand{\Syn}{\synspectra}
\newcommand{\CoAlg}{\mathrm{Co}\euscr{A}\mathrm{lg}}

\newcommand{\Sph}{\euscr{S}\mathrm{ph}}

\newcommand{\Sp}{{\euscr{S}\mathrm{p}}}

\newcommand{\Free}{\operatorname{Free}}

\usepackage{bbm}
\newcommand{\cterminal}{\mathbbm{1}_\ccat}

\newcommand{\integers}{\mathbb{Z}}
\newcommand{\thesphere}{\mathbf{S}}
\newcommand{\pt}{\mathrm{pt}}
\newcommand{\id}{\mathrm{id}}

\newcommand{\proj}{\mathrm{proj}}
\newcommand{\h}{\mathrm{h}}

\newcommand{\op}{\mathrm{op}}

\newcommand{\der}{\mathrm{der}}

\newcommand{\fil}{\mathrm{fil}}

\newcommand{\core}{{\text{core}}}

\newcommand{\Cpl}{\mathrm{Cpl}}
\newcommand{\res}{\mathrm{res}}
\newcommand{\ev}{\mathrm{ev}}

\newcommand{\nb}{\mathrm{nb}}

\newcommand{\stablyctame}{\mathrm{sct}}
\newcommand{\cpl}{\mathrm{cpl}}

\newcommand{\derivable}{derivable}

\newcommand{\bs}{{-}}

\newcommand*\noloc{%
        \nobreak
        \mskip6mu plus1mu
        \mathpunct{}%
        \nonscript
        \mkern-\thinmuskip
        {:}%
        \mskip2mu
        \relax
}

\usepackage{todonotes}

\makeatletter
\newcommand{\pto}{}% just for safety
\newcommand{\pgets}{}% just for safety

\DeclareRobustCommand{\pto}{\mathrel{\mathpalette\p@to@gets\to}}
\DeclareRobustCommand{\pgets}{\mathrel{\mathpalette\p@to@gets\gets}}

\newcommand{\p@to@gets}[2]{%
  \ooalign{\hidewidth$\m@th#1\mapstochar\mkern5mu$\hidewidth\cr$\m@th#1\to$\cr}%
}
\makeatother

\usetikzlibrary{decorations.markings}
\tikzset{mid vert/.style={/utils/exec=\tikzset{every node/.append style={outer sep=0.8ex}},
postaction=decorate,decoration={markings,
mark=at position 0.5 with {\draw[-] (0,#1) -- (0,-#1);}}},
mid vert/.default=0.75ex}

\makeatletter
  \def\subsection{\@startsection{subsection}{1}%
  \z@{.7\linespacing\@plus\linespacing}{.5\linespacing}%
  {\normalfont\bfseries\centering}}% NEW
\makeatother

\let\oldtocsection=\tocsection
\let\oldtocsubsection=\tocsubsection
\let\oldtocsubsubsection=\tocsubsubsection
\renewcommand{\tocsection}[2]{\hspace{0em}\oldtocsection{#1}{#2}}
\renewcommand{\tocsubsection}[2]{\hspace{1em}\oldtocsubsection{#1}{#2}}
\renewcommand{\tocsubsubsection}[2]{\hspace{2em}\oldtocsubsubsection{#1}{#2}}

\calclayout

\begin{document}

\title[Malcev theories]{Unstable synthetic deformations I: \\ Malcev theories}
\author{William Balderrama}
\author{Piotr Pstr\k{a}gowski}

\begin{abstract}
This paper is the first in a series of articles devoted to the construction and study of synthetic deformations of $\infty$-categories in the unstable context: that is, deformations of $\infty$-categories that categorify spectral sequence or obstruction-theoretic information. Our approach is based on the techniques of higher universal algebra, with deformations built from the $\infty$-categories of models of $\infty$-categorical variants of algebraic, or Lawvere, theories.

This paper sets up the foundations of our study. We introduce and study various classes of $\infty$-categorical and infinitary algebraic theories. We establish many basic properties of the $\infty$-categories of the models of different classes of theories, such as the existence of free resolutions, image factorizations, monadicity theories, and presentability, as well as recognition theorems identifying the $\infty$-categories that arise this way.

We give an intrinsic definition of a Malcev theory in higher universal algebra. These are characterized as those theories satisfying Quillen's condition that all simplicial models satisfy the Kan condition, and we prove that this is equivalent to a weak grouplike condition which is easily verified in practice. We establish that the $\infty$-category of models of a Malcev theory may be characterized as freely adjoining geometric realizations to the theory. This leads to the notion of a derived functor between $\infty$-categories of models of Malcev theories, and we study some of the behavior of these derived functors with respect to connectivity and limits. 

The key idea in our work is that the $\infty$-category of models of a Malcev theory $\pcat$ can be thought of as a deformation whose special fibre is the $\infty$-category of models of its homotopy category $\h\pcat$. We also recall the notion of a loop theory, a class of Malcev theories whose $\infty$-category of models also admits a generic fibre, given by the full subcategory of loop models. We study in detail the interaction between functors and derived functors of $\infty$-categories of loop models and models, establishing in particular that monads and a large class of comonads on the generic fibre lift canonically to the whole deformation.

In the last part of the paper, we show that by considering the coalgebras for these deformed comonads over $\infty$-categories of models, one can recover various stable deformations considered in the literature, such as filtered models or Postnikov-complete synthetic spectra. We then expand on these results by constructing $\infty$-categories of synthetic spaces and synthetic $\bfE_k$-rings which, as will be further developed in the sequels, categorify the generalized unstable Adams spectral sequence and Goerss--Hopkins spectral sequence respectively.
\end{abstract}

\maketitle 

\tableofcontents

\section{Introduction}

A crucial ingredient in many recent advances in stable homotopy theory is the discovery that many spectral sequences and obstruction theories can be encoded by appropriate \emph{deformations} of $\infty$-categories, in particular the $\infty$-categories of synthetic spectra which categorify stable Adams-type spectral sequences \cite{gheorghe2022c, pstrkagowski2018synthetic}. This paper is the first in a series which develops a comprehensive theory of such deformations in the \emph{unstable} context. In particular, we construct $\infty$-categories of:
\begin{enumerate}
    \item Synthetic spaces, categorifying the unstable Adams spectral sequence \cite{bendersky1978unstable, benderskythompson2000bousfield};
    \item Synthetic $\bfE_k$-rings, categorifying the Goerss--Hopkins spectral sequence \cite{moduli_spaces_of_commutative_ring_spectra, moduli_problems_for_structured_ring_spectra};
    \item And many variants, such as equivariant, motivic, and $K(n)$-local analogues, synthetic $\bfE_k$-spaces, synthetic $v_n$-periodic spaces, synthetic $\calO$-algebras, and more.
\end{enumerate}

We explain how these $\infty$-categories can be thought of as deformations: there are well-behaved notions of a \emph{generic fibre} (which is the $\infty$-category being deformed), \emph{special fibre} (which is purely algebraic, and encodes the $E_{2}$-term of the corresponding spectral sequence, or obstruction groups of the corresponding obstruction theory), and a suitable ``cofibre of $\tau$'' or ``$\tau$-Bockstein'' formalism refining the associated spectral sequence or obstruction theory. In particular, we show that in each case the deformation itself can be constructed as the limit of a tower of square-zero extensions of $\infty$-categories starting from the special fibre.

Our approach is based on the observation of \cite{pstrkagowski2023moduli, balderrama2021deformations} that many deformations can be constructed as the $\infty$-category of models of an $\infty$-categorical \emph{Malcev theory}. These are a certain well-behaved variation of the classical notion of an algebraic, or Lawvere, theory \cite{lawvere1963functorial,borceuxbourn2004malcev}. We vastly expand the deformations that may be constructed using these methods by showing that there is also a natural way to deform (co)monads, and that the resulting $\infty$-category of (co)algebras over a deformed (co)monad itself inherits all the same structure allowing it to be thought of as a deformation in its own right.

For example, the $\infty$-category of $\mathbf{F}_{p}$-synthetic spaces we construct arises from deforming the comonad associated to the adjunction
\[
\mathbf{F}_{p} \otimes \Sigma^{\infty}_{+}({-}) : \spaces \rightleftarrows \Mod_{\mathbf{F}_{p}} : \Omega^\infty.
\]
This construction has the advantage that its special fibre and relation to the unstable Adams spectral sequence is easily identified directly from the definition. In addition, by replacing $\spaces$ with $\spectra$, one instead obtains the (Postnikov completion of the) $\infty$-category of $\mathbf{F}_p$-synthetic spectra, providing direct compatibility between the stable and unstable deformations.

\subsection{Outline of the project} 
\label{subsection:outline_of_the_project}

The three papers in this series developing a theory of unstable deformations are as follows:

\begin{enumerate}
    \item[(I)] \emph{Malcev theories} - the current work. Develops a general theory of infinitary algebraic theories in the $\infty$-categorical context: the Malcev condition, properties and characterizations of $\infty$-categories of models, derived functors and (co)monads.
    \item[(II)] \emph{Infinitesimal extensions} - \cite{usd2}. Develops a nonabelian deformation theory of Malcev theories and their models. The Postnikov tower of a Malcev theory is constructed as a tower of square-zero extensions, preserved by passage to $\infty$-categories of models, leading to the \emph{spiral tower}, or Goerss--Hopkins tower, of a Malcev theory. This provides an unstable ``cofibre of $\tau$'' formalism including obstruction theories to lifting objects and morphisms.
    \item[(III)] \emph{Naturality of the spiral tower} - \cite{usd3}. Establishes naturality of the spiral tower of a Malcev theory in an $(\infty,2)$-categorical sense. The structure on the spiral tower of a Malcev theory is closed under many $(\infty,2)$-categorical constructions, leading to the spiral tower for coalgebras over deformed comonads, such as for synthetic spaces.
\end{enumerate}

These three papers establish a general abstract theory of unstable deformations. This is motivated by particular examples such as synthetic spaces, but applies more broadly. A benefit of our direct algebraic approach is that it is easy to construct new deformations with prescribed special fibre, categorifying a very flexible class of spectral sequences and obstruction theories.

In separate work, we plan to specialize to study the particular case of synthetic spaces in more detail. In particular, forthcoming work of Bachmann and Hopkins on the motivic Wilson space hypothesis allows us to relate the $\infty$-category of $\bbC$-motivic spaces with our $\infty$-category of even $\MU$-synthetic spaces, extending the stable comparison of \cite{gheorghe2022c, pstrkagowski2018synthetic}.

\subsection{Summary of results}

In this thesis \cite{lawvere1963functorial}, Lawvere introduced a categorical approach to universal algebra through what are now known as \emph{Lawvere theories}. A classical Lawvere theory is an ordinary category $\tcat$ which admits finite coproducts and is generated under these by a single object. The associated classical category of models is then given by the category of presheaves $\tcat^\op\to\sets$ of sets which preserve finite products, i.e.\ send finite coproducts in $\tcat$ to finite products of sets. These encode algebraic structures that may be described as a set together with various $n$-ary operations satisfying universally quantified relations. The current work is largely concerned with the following natural generalization of Lawvere theories: 

\begin{definition}[{\ref{def:theory}}]
\label{definition:introduction_theory}
A (multi-sorted, infinitary algebraic) \emph{theory} $\pcat$ is an $\infty$-category which admits all small coproducts and which is generated under small coproducts and retracts by a small set of objects.\footnote{In the main body of the text, we also consider \emph{pretheories}, which are not necessarily generated by a small set of objects. We stick to theories in the introduction for simplicity.}
The $\infty$-category of \emph{models} of theory $\pcat$ is the full subcategory
\[
\Model_{\pcat} \subseteq \Fun(\pcat^{\op}, \spaces) 
\] 
of presheaves which preserve all small products. 
\end{definition}

Any classical Lawvere theory $\tcat$ may be extended to a theory in the above sense by attaching infinite coproducts; the associated $\infty$-category of models is then equivalent to the \emph{animation} of the classical category of models of $\tcat$ (\cref{thm:animation}). The infinitary context is more flexible, as it allows one to encode algebraic structures which admit operations with infinitely many inputs. The motivating examples of this type are theories of complete modules, with motivating applications being to deformations relevant in $K(n)$-local homotopy theory.

Sitting between Lawvere theories and \cref{definition:introduction_theory}, for each regular cardinal $\kappa$ one may define a (multi-sorted) \emph{$\kappa$-ary theory} to be an $\infty$-category generated under $\kappa$-small coproducts by a small set of objects. Any $\kappa$-ary theory may be extended to a theory in the sense of \cref{definition:introduction_theory} by attaching the remaining small coproducts (\cref{ex:kappabounded}). We prove that this does not affect the associated $\infty$-category of models (\cref{thm:bounded}) and that the theories which arise by such process are exactly those whose $\infty$-categories of models are presentable (\cref{cor:presentable}). Thus \cref{definition:introduction_theory} subsumes essentially all reasonable definitions of an infinitary algebraic theory.

Classically, the category of discrete models for a discrete Lawvere theory $\tcat$ may be identified with its $1$-categorical free cocompletion under filtered colimits and reflexive coequalizers. Its animation, the $\infty$-category of space-valued models of $\tcat$, is then its $\infty$-categorical free cocompletion under filtered colimits and geometric realizations \cite[\S5.5.8]{lurie_higher_topos_theory}. On the other hand, if $\pcat$ is a general infinitary theory, then the inclusion $\Model_\pcat\subset\Fun(\pcat^\op,\spaces)$ need not be closed under geometric realizations, as infinite products need not preserve geometric realizations, and accordingly $\Model_\pcat$ need not be the free geometric realization-cocompletion of $\pcat$.

This subtlety in working with infinitary theories was first observed, implicitly, in Quillen's work on homotopy theories of algebraic structures \cite{quillen1967homotopical}. The classical solution is to observe that infinite products \emph{do} preserve geometric realizations of Kan complexes. Thus one is naturally led to consider those theories whose simplicial models always take values in Kan complexes. These are exactly the \emph{Malcev theories}. We give the definitions most convenient for our purposes here, and refer the reader to \S\ref{subsection:malcevkanhistory} for a discussion of the history of these notions.

\begin{defn}[{\ref{def:malcevop}}]
\label{definition:introduction:malcev_operation}
Let $\ccat$ be an $\infty$-category with finite products. A \emph{Malcev operation} on an object $X \in \ccat$ is a ternary operation
\[
t\colon X \times X \times X \to X,
\]
together with a choice of two homotopies filling the triangles
\begin{equation}
\begin{tikzcd}
X\times X\ar[r,"X\times \Delta"]\ar[dr,"\pi_1"']&X\times X \times X\ar[d,"t"]&X\times X\ar[l,"\Delta\times X"']\ar[dl,"\pi_2"]\\
&X
\end{tikzcd}.
\end{equation}
\end{defn}

For example, any group $G$ admits the Malcev operation $t(g,h,k) = g h^{-1} k$. Malcev operations play a privileged role in the theory of simplicial resolutions as a result of the following.

\begin{theorem}[{\ref{prop:univkancharacterize}}]
\label{theorem:introduction:universally_kan_objects_are_the_same_as_those_admitting_a_malcev_operation}
Let $\ccat$ be an $\infty$-category with finite products. For an object $X \in \ccat$, the following are equivalent:
\begin{enumerate}
\item $X$ admits a Malcev operation in the sense of \cref{definition:introduction:malcev_operation};
\item $X$ is \emph{universally Kan}: for any finite product-preserving functor
\[
F\colon \ccat\to\Fun(\Delta^\op,\spaces)
\]
valued in simplicial spaces, $F(X)$ satisfies the Kan condition (\cref{def:kancondition}).
\end{enumerate}
\end{theorem}

Given its privileged position among all $\infty$-categories, one is immediately led to consider the universally Kan objects in the $\infty$-category $\spaces$ of spaces. By the above theorem, these are exactly those spaces which admit a Malcev operation. This important class of spaces admits the following more intrinsic definition.

\begin{definition}[{\ref{def:supersimple}}]
\label{def:intro:supersimple}
We say that a space $X \in \spaces$ is \emph{supersimple} for all pointed connected spaces $A$ and $B$ and every map $A\vee B\to X$, there exists a dashed arrow making the diagram
\begin{center}\begin{tikzcd}
A\vee B\ar[r]\ar[d]&X\\
A\times B\ar[ur,dashed]
\end{tikzcd}\end{center}
commute.
\end{definition}

The following relates this class of spaces to those characterized by \cref{theorem:introduction:universally_kan_objects_are_the_same_as_those_admitting_a_malcev_operation}: 

\begin{theorem}[{\ref{cor:ukanspaces}}]
\label{thm:intro:ukanspaces}
For a space $X \in \spaces$, the following are equivalent: 
\begin{enumerate}
    \item $X$ is universally Kan;
    \item $X$ is supersimple;
    \item Every path component of $X$ admits a structure of an $H$-space. 
\end{enumerate}
\end{theorem}

By taking $A$ and $B$ to be positive-dimensional spheres in \cref{def:intro:supersimple}, one sees that if $X$ is supersimple then $\pi_\ast(X,x)$ has vanishing Whitehead products at every basepoint $x\in X$. Thus the supersimple spaces are a natural subclass of the spaces with vanishing Whitehead products, themselves a natural subclass of the componentwise simple spaces.  Our characterization in terms of the Kan condition implies that the class of supersimple spaces has good closure properties, such as being closed under taking mapping spaces, see \cref{prop:univkanclosure}, which do not hold for the class of spaces with vanishing Whitehead products \cite[Example 6.6]{luptonsmith2010whitehead}.

Returning to our discussion of theories, \cref{theorem:introduction:universally_kan_objects_are_the_same_as_those_admitting_a_malcev_operation} motivates the following definition.

\begin{defn}[{\ref{def:malcevtheory}}]\label{def:intro:malcev}
A theory $\pcat$ is said to be \emph{Malcev} if either of the following equivalent conditions is satisfied:
\begin{enumerate}
\item Every simplicial model
\[
\Xss\colon \Delta^\op\to \Model_\pcat
\]
satisfies the Kan condition, i.e.\ $\Xss(P)$ satisfies the Kan condition for every $P\in \pcat$.
\item Every object of $\pcat$ admits a co-Malcev operation, i.e.\ a Malcev operation in $\pcat^\op$.
\end{enumerate}
\end{defn}

The general theory of free resolutions for models of theories (\cref{thm:splithypercovering}), together with the fact that infinite products preserve geometric realizations of simplicial spaces satisfying the Kan condition (\cref{prop:kanlimits}), leads to the following characterization of the $\infty$-category of models for a Malcev theory.

\begin{theorem}[{\ref{lem:malcevcocomplete}, \ref{thm:freecocompletion}}]
\label{introduction:theorem_universal_property_of_models_of_a_malcev_theory}
Let $\pcat$ be a Malcev theory. Then the $\infty$-category $\Model_{\pcat}$ of models is complete and cocomplete, and the Yoneda embedding $\nu \colon \pcat \hookrightarrow \Model_{\pcat}$ exhibits $\Model_\pcat$ as the free cocompletion of $\pcat$ under geometric realizations.
\end{theorem}

This result endows $\Model_\pcat$ with a universal property, and confirms our intuition that $\Model_\pcat$ acts as an $\infty$-category of ``formal simplicial resolutions'' by the objects of $\pcat$. Conversely, one might consider the free geometric realization-cocompletion of $\pcat$ as the more fundamental notion, in which case \cref{introduction:theorem_universal_property_of_models_of_a_malcev_theory} provides a concrete description of this cocompletion in terms of product-preserving presheaves.

We emphasize that the Malcev condition is useful not only for providing a universal property of $\infty$-categories of models in the infinitary setting. It is used in a fundamental way throughout our development of unstable synthetic deformations. 

One particularly useful consequence of the definition is that in the $\infty$-category of models of a Malcev theory, geometric realizations commute with pullbacks along effective epimorphisms, see \cref{corollary:for_models_of_malcev_theory_geometric_realization_commutes_with_products_and_pullbacks_along_effective_epi}. As every model can be written as a geometric realization of representables, this enables one to reduce many arguments or constructions to the case of free objects. This incredibly robust technique is used repeatedly in our study of infinitesimal extensions in the sequel \cite{usd2}. Another consequence is that for a Malcev theory, the objects in the image of the Yoneda embedding have the following useful property: 

\begin{definition}[{\ref{def:stronglyprojective}}]\label{def:intro:stronglyprojective}
Let $\ccat$ be an $\infty$-category which admits geometric realizations. We say that an object $P \in \ccat$ is \emph{strongly projective} if for every simplicial object $\Xss$ in $\ccat$, the simplicial space $\map_{\ccat}(P, \Xss)$ is a hypercovering of $\map_{\ccat}(P, |\Xss|)$. 
\end{definition}

This property was, to our knowledge, first considered by Lurie in \cite[Definition 4.2.1]{lurie2011dag8}. Note that since a hypercovering is in particular a colimit diagram, strongly projective objects are projective in the usual sense of \cite[{Definition 5.5.8.18}]{lurie_higher_topos_theory}. The following result gives an intrinsic characterization of the $\infty$-categories of models for Malcev theories in these terms.

\begin{theorem}[{\ref{thm:characterization_of_malcev_theories_in_terms_of_strong_projectivity}}]
\label{theorem:introduction:characterization_of_models_of_malcev_theory_in_temrs_of_projectivity}
Let $\dcat$ be a cocomplete $\infty$-category. Then the following are equivalent: 
\begin{enumerate}
    \item $\dcat$ is \emph{strongly projectively generated}: there exists a small set $\{P_i\}$ of strongly projective objects of $\dcat$ for which $\map_{\dcat}(P_i, -)$ jointly detect equivalences;
    \item There exists a Malcev theory $\pcat$ for which $\dcat \simeq \Model_{\pcat}$. 
\end{enumerate}
In this case, one can take $\pcat\subset\dcat$ to be the full subcategory of strongly projective objects. 
\end{theorem}

We also prove variants for general theories in the sense of \cref{definition:introduction_theory}, where the assumption of projectivity needs to be slightly weakened, and for additive theories, where classical and strong projectivity are equivalent, see \cref{thm:weaklyprojectivecats} and \cref{prop:projprestable}. 

A natural notion of a morphism between theories is that of a \emph{homomorphism} $f \colon \pcat \rightarrow \qcat$, which is a coproduct-preserving functor. Using \cref{introduction:theorem_universal_property_of_models_of_a_malcev_theory}, if $\pcat$ and $\qcat$ are Malcev theories then left Kan extension determines an equivalence between homomorphisms $f\colon \pcat\to\qcat$ and left adjoints $f_! \colon \Model_{\pcat} \rightarrow \Model_{\qcat}$ between $\infty$-categories of models which preserve representables. However, for many purposes both of these conditions are too strict: many useful constructions may preserve neither all colimits nor subcategories of free objects. Example include the free symmetric algebra or free $\bfE_k$-algebra functors on various $\infty$-categories of modules.  This, together with the universal property of models, motivates the following. 

\begin{definition}[{\ref{definition:derived_functors_between_malcev_pretheories}}]
Let $\pcat$ and $\qcat$ be Malcev theories. A \emph{derived functor} $f\colon \pcat\pto\qcat$ is either of the following equivalent pieces of data:
\begin{enumerate}
\item An arbitrary functor $f\colon \pcat\to\Model_\qcat$;
\item A geometric realization-preserving functor $f_!\colon \Model_\pcat\to\Model_\qcat$.
\end{enumerate}
\end{definition}

Derived functors in this sense are a generalization of the \emph{left-derived functors} of classical homological algebra (\cref{ex:classicalderived}). They are always right exact both in the sense that they preserve geometric realizations and in the sense that they preserve connectivity of morphisms (\cref{prop:derivedfunctorpreservesconnectivity}). On the other hand, they are rarely left exact, i.e.\ rarely preserve finite limits and rarely preserve truncativity of morphisms. In \S\ref{ssec:weakleftexact}, we introduce the notion of a \emph{weakly left exact} functor: a derived functor which preserves pullbacks along effective epimorphisms. This condition is satisfied in many useful examples, such as for left adjoint derived functors between additive theories, see \cref{proposition:additive_derived_functors_preserve_pullbacks_along_pi0_epis}.

Associated to any Malcev theory $\pcat$ is a tower of Malcev theories given by its homotopy $n$-categories $\h_n\pcat$. Passing to associated $\infty$-categories of models we obtain what we call the \emph{spiral tower}, or Goerss--Hopkins tower, of $\pcat$:
\begin{equation}\label{eq:intro:spiraltower}
\Model_{\pcat} \rightarrow \cdots \rightarrow \Model_{h_{2} \pcat} \rightarrow \Model_{h \pcat}.
\end{equation}
In the sequel \cite{usd2}, we show that this tower is, in a precise sense, a tower of square-zero extensions of $\infty$-categories, providing a natural nonabelian analogue of the deformation theory of connective modules over connective ring spectra, and providing categorifications of variants of the Blanc--Dwyer--Goerss and Goerss--Hopkins obstruction theories of \cite{realization_space_of_a_pi_algebra, moduli_spaces_of_commutative_ring_spectra, moduli_problems_for_structured_ring_spectra, abstract_gh_theory}. 

We now discuss the connection between Malcev theories and deformations. The above tower allows one to view $\Model_\pcat$ as a deformation of the essentially algebraic $\infty$-category $\Model_{\h\pcat}$, in the same way that a connective ring spectrum $A$ is a deformation of its $0$th homotopy ring $A_{\leq 0} = \pi_0 A$. In the usual language of synthetic homotopy theory, $\Model_{\h\pcat}$ is the \emph{special fibre} of $\Model_\pcat$. 

At this level of generality, there is no generic fibre being deformed. Moreover, the characterization of \cref{theorem:introduction:characterization_of_models_of_malcev_theory_in_temrs_of_projectivity} shows that for an $\infty$-category to be equivalent to $\Model_\pcat$ for a Malcev theory $\pcat$ is a very rigid condition; for example, no nonzero stable $\infty$-category can ever be of this form (\cref{cor:neverstable}). To capture a greater variety of $\infty$-categories, we make use of the following slightly more refined concept.

\begin{definition}[{\ref{def:looptheory}, \ref{def:loop_model}}]
\label{definition:introduction_loop-theory}
A Malcev theory $\pcat$ is a \emph{loop theory} if it admits tensors by $S^{1}$; that is, for all $P \in \pcat$, the constant colimit $S^1\otimes P = \colim_{S^1}P$ exists in $\pcat$. The $\infty$-category of \emph{loop models} of $\pcat$ is the full subcategory
\[
\Model_{\pcat}^{\Omega} \subseteq \Model_{\pcat} 
\]
spanned by those models $X$ for which the canonical comparison map $X(S^{1} \otimes P) \rightarrow X(P)^{S^{1}}$ is an equivalence for all $P \in \pcat$. 
\end{definition}

This definition is a slight generalization of the corresponding notion considered in \cite{balderrama2021deformations} (see \cref{rmk:loopdefns}). The key insights of \cite{pstrkagowski2023moduli, balderrama2021deformations} are that: 
\begin{enumerate}
    \item Many naturally arising $\infty$-categories $\dcat$ (for example: spectra, modules, $\mathbf{E}_{k}$-rings, pointed connected spaces) can be naturally written in the form 
    \[
    \dcat \simeq \Model_{\pcat}^{\Omega},
    \]
    that is, as $\infty$-categories of loop models for a suitable loop theory $\pcat\subset\dcat$.
    \item The $\infty$-category $\Model_{\pcat}$ of all models can be interpreted as an $\infty$-categorical deformation of which $\Model_{\pcat}^{\Omega} \simeq \dcat$ is the generic fibre. In particular, $\Model_\pcat^\Omega$ is a localization of the $\infty$-category $\Model_\pcat$ of all models (under a minor size condition on $\pcat$).
\end{enumerate}

Note that there can be many nonequivalent choices of a loop theory $\pcat\subset\dcat$ for which $\dcat \simeq \Model_{\pcat}^{\Omega}$. Such a choice can be thought of as a choice of free, or detecting, objects for $\dcat$. Different choices can lead to very different deformations of $\dcat$, and consequently to different spectral sequences and obstruction groups converging to the same thing, of which one may be suitable for a given purpose than another. We give an illuminating example in \cref{remark:writing_something_as_loop_models_requires_choice}. 

A priori, one limitation of the deformations obtained in this way is that they are entirely controlled by a theory of simplicial resolutions by a chosen class of ``free objects''. As a consequence, the spectral sequences and obstruction theories that may be derived from the spiral tower of (\ref{eq:intro:spiraltower}) are essentially only those which may be given in terms of corepresentable invariants, such as homotopy groups. This rules out, for example, the classical Adams spectral sequence, whose $E_2$-page is given in terms of $\mathbf{F}_p$-homology, which is not a corepresentable functor on spectra.

In general, if $ER$ is an $\bfE_1$-ring spectrum then the $R$-based Adams spectral sequence for a spectrum $X$ can be thought of as combining a projective resolution of $R_\ast X$ with descent along the adjunction $\spectra\rightleftarrows \LMod_R$. Thus, informally, the ingredient we are missing to encode a synthetic deformation along $R$-homology is \emph{descent}. 

In \cite{pstrkagowski2018synthetic, patchkoria2021adams}, these issues are dealt with by encoding the needed descent in terms of a Grothendieck topology. In the current work, we instead work directly with the comonad associated to this adjunction. This has the advantage of generalizing readily to unstable contexts, such as to synthetic spaces, and providing deformations with immediately identifiable special fibre.

Let $\pcat$ be a loop theory. Given an endofunctor $f \colon \Model_{\pcat}^{\Omega} \rightarrow \Model_{\pcat}^{\Omega}$, the composite 
\[
\begin{tikzcd}
	{\pcat} & {\Model_{\pcat}^{\Omega}} & {\Model_{\pcat}^{\Omega}} & {\Model_{\pcat}}
	\arrow[hook, from=1-1, to=1-2]
	\arrow["f", from=1-2, to=1-3]
	\arrow["\nu", hook, from=1-3, to=1-4]
\end{tikzcd},
\]
where the first arrow is the restricted Yoneda embedding and the last arrow is the natural inclusion, determines a derived functor $D(f) \colon \pcat\pto\pcat$. We show that this construction is compatible with composition in the following sense.

\begin{theorem}[{\ref{remark:associated_derived_functor_construction_is_monoidal}}]\label{thm:intro:laxderived}
The construction $f \mapsto D(f)$ refines to a lax monoidal right adjoint 
\[
\End(\Model_{\pcat}^{\Omega}) \rightarrow \Endo_!(\pcat), 
\]
where $\Endo_!(\pcat)\simeq\End^\sigma(\Model_\pcat)$ is the $\infty$-category of derived functors $\pcat\pto\pcat$, i.e.\ of geometric realization-preserving endofunctors of $\Model_\pcat$.
\end{theorem}

If $m$ is a monad on $\Model_{\pcat}^{\Omega}$, then this implies that $D(m)$ is a monad on $\Model_\pcat$. In fact, even more is true: as we explain in \S\ref{ssec:malcevmonads}, Malcev theories and loop theories are informally ``closed under monads'', entirely bypassing \cref{thm:intro:laxderived}. In particular, the $\infty$-category $\Alg_{D(m)}$ of algebras for $D(m)$ may be constructed directly as the $\infty$-category of models for the loop theory obtained as the essential image of the restriction of $D(m)\colon \Model_\pcat\to\Alg_{D(m)}$ to $\pcat\subset\Model_\pcat$, and in good cases $\Alg_m$ is equivalent to its $\infty$-category of loop models. 

Our main use of \cref{thm:intro:laxderived} is that it allows us to also work with comonads on $\Model_\pcat^\Omega$. In general, if $c$ is a comonad on $\Model_\pcat^\Omega$, then $D(c)$ need not be a comonad on $\Model_\pcat$, but it is if the lax structure maps $D(c)^{\circ n} \to D(c^{\circ n})$ are equivalences; we call such comonads \emph{\derivable{}}. In \cref{theorem:monoidal_functor_between_lmod_induced_by_induced_derived_functor_construction}, we show that the lax functor $D$ is in additional compatible with the tautological actions of these $\infty$-categories of endomorphisms. Using this, we show that if $c$ is a \derivable{} comonad then (under a minor size condition on $\pcat$) there is a generic fibre functor
\[
\CoAlg_{D(c)}(\Model_\pcat) \to \CoAlg_c(\Model_\pcat^\Omega),
\]
as well as a class of ``stably $c$-tame'' loop models for which there is an inclusion
\[
\CoAlg_c(\Model_\pcat^{\Omega,\stablyctame}) \hookrightarrow \CoAlg_{D(c)}(\Model_\pcat)
\]
which one can think of as a synthetic analogue functor. In the third paper in the series \cite{usd3}, we show that the formation of the spiral tower of a Malcev theory is in a precise sense lax functorial in derived functors. This provides a spiral tower for the deformation $\CoAlg_{D(c)}(\Model_\pcat)$ of $\CoAlg_c$, with attendant ``cofibre of $\tau$'' formalism, spectral sequences, and obstruction theories.

In the last section of the paper, we give some examples of deformations that may be constructed using our formalism. In particular, we discuss: 

 \begin{enumerate}
     \item[{(\S\ref{subsection:filtered_modules})}]\emph{Filtered modules}. If $R$ is an $\mathbf{E}_{1}$-ring spectrum, we write $\lfrees(R)\subset\LMod_R$ for the full subcategory generated by $R$ under suspensions, desuspensions, and direct sums. This is a loop theory in the sense of \cref{definition:introduction_loop-theory}, and in \cref{prop:models_of_free_r_modules_as_modules_over_the_whitehead_tower} we show that 
\[
\Model_{\lfrees(R)}^\Omega\simeq\LMod_R,\qquad \Model_{\lfrees(R)}\simeq \LMod_{\tau_{\geq \star}R}(\Sp^\fil).
\]
Thus, the Malcev theory $\lfrees(R)$ encodes the filtered deformation of $R$-modules based on the Whitehead tower of $R$.

\item[{(\S\ref{subsection:synthetic_spectra})}]\emph{Synthetic spectra}. The above example gives a natural way to encode a deformation of $R$-modules using Malcev theories. As outlined above, to go down to spectra, we deform the comonad $R \otimes_{\thesphere} -$ associated to the adjunction $\spectra\rightleftarrows\LMod_R.$ 

We show that the derivability conditions necessary to do so are satisfied if, for example, $R_\ast R$ is flat as a right $R_\ast$-module. The $\infty$-category of coalgebras for the deformed comonad is a variant of connective synthetic spectra, categorifying the $R$-Adams spectral sequence by construction, with special fibre given by the complete derived $\infty$-category of $R_{*}R$-comodules. We show in \cref{theorem:nilpotent_complete_synthetic_spectra_as_coalgebras_for_a_derived_functor} that if $R$ is Adams-type then our construction recovers the Postnikov completion of the connective synthetic spectra of \cite{pstrkagowski2018synthetic}. 

\item[{(\S\ref{subsection:examples_synthetic_spaces})}] \emph{Synthetic spaces}. Motivated by the previous example, in \cref{def:rsyntheticspaes} we \emph{define} the $\infty$-category $\Syn\spaces_R$ of \emph{$R$-synthetic spaces} as the $\infty$-category of coalgebras for the comonad on $\Model_{\lfrees(R)}\simeq \LMod_{\tau_{\geq \star}R}(\Sp^\fil)$ deforming the comonad associated to the adjunction
\[
R\otimes_{\thesphere} \Sigma^\infty_+ (\bs) : \spaces\rightleftarrows\LMod_R : \Omega^\infty.
\]
In \cref{proposition:comonad_defining_synthetic_spaces_is_self_tame}, we show that the necessary derivability conditions are satisfied if $R_\ast\Omega^{\infty-n}R$ is projective as a left $R_\ast$-module for all $n\in\integers$. For example, this holds when $R_{\ast}$ is a field or when $R = \MU$. 

By construction, the $\infty$-category of $R$-synthetic spaces categorifies the unstable $R$-Adams spectral sequence of \cite{bendersky1978unstable, benderskythompson2000bousfield}. Our tameness conditions are variations of the assumptions needed there to give good descriptions of the $E_2$-page in terms of unstable coalgebras.

\item[{(\S\ref{ssec:examples:syntheticekrings})}] \emph{Synthetic $\calO$-algebras}. 
If $R$ is (for simplicity) an $\bfE_\infty$-ring and $\calO$ is an $\infty$-operad, then by deforming the comonad associated to the adjunction
\[
\Alg_{\calO} \rightleftarrows \Alg_{\calO}(\Mod_R)
\]
we likewise obtain an $\infty$-category of \emph{$R$-synthetic $\calO$-algebras}. Here, the needed derivability assumption holds when $R_\ast R$ is flat over $R_\ast$, and the special fibre is an algebraic category of $R_{*}$-algebras equipped with a compatible $R_{*}R$-comodule structure and power operations.

Taking $R = \bfF_p$ and $\calO = \bfE_k$, this categorifies a version of the Goerss--Hopkins spectral sequence for maps between $\bfE_k$-rings developed in unpublished work of Senger \cite[\S4]{lawson2019calculating}. Our work therefore provides progress on the problem posed by Lawson \cite[Problem 1.9.1]{lawson2020en} to develop obstruction theories for maps between $\bfE_k$-algebras in a wide variety of contexts. 

A completed variation of this construction, made possible by our use of infinitary theories, similarly provides deformations of the $\infty$-category $\CAlg(\Sp_{K(n)})$ of $K(n)$-local $\bfE_\infty$-rings which categorify and extend the Goerss--Hopkins spectral sequence for maps between $K(1)$-local $\bfE_\infty$-rings \cite[Theorem 2.4.14]{moduli_problems_for_structured_ring_spectra} and provide one answer to a conjecture of Lawson \cite[Conjecture 1.6.6]{lawson2020en}. 
 \end{enumerate}

We remark that several ideas related to the ones explored here appear in upcoming work of Raksit \cite{arpon_spectral_algebraic_theories}. One of the constructions given in that paper, \emph{energization}, associates to an $\infty$-category with a suitable functor into graded spectra a \emph{derived $\infty$-category}, obtained by explicitly writing down a variant of a loop theory and considering its corresponding $\infty$-category of loop models. Raksit proves that applied to the $\infty$-category of spectra with homotopy groups concentrated in even degrees, this produces (nonconnective) even $\MU$-based synthetic spectra. This provides another construction of a variant of synthetic spectra using the techniques of higher universal algebra, independent from \S\ref{subsection:synthetic_spectra} discussed above. 

We note that the two approaches differ in focus. In the present work, loop theories encode deformations whose $\infty$-category of loop models recovers some existing objects of interest, and whose non-loop models then provide a useful connective deformation. In \cite{arpon_spectral_algebraic_theories}, the focus is on constructing nonconnective deformations, such as the theory of derived rings, as $\infty$-categories of the analogue of loop models in that context.

\subsection{A brief history of the Malcev condition} 
\label{subsection:malcevkanhistory}

Although the precise definition of an $\infty$-categorical Malcev theory we give is new, the general idea goes back much further. Since the notion may seem esoteric on first encounter, in this section we recall some of its history. At the end, we also ask a few questions about the behavior of general theories, not necessarily in the presence of the Malcev condition.

Malcev operations were first introduced by \textcyr{Мальцев} \cite{malcev1954general}, variously transliterated Malcev, Mal’cev, Mal’tsev, Maltsev, who proved, in modern language, that if $\ccat$ is the category of models for an algebraic theory, then the composition of internal equivalence relations in $\ccat$ is commutative if and only if the theory admits a Malcev operation. In \cite{findlay1960reflexive}, Findlay proved that this implies every internal reflexive relation in $\ccat$ is an internal equivalence relation, and the converse was shown by Shafaat \cite{shafaat1974note}. These conditions and generalizations have since seen considerable study in universal algebra; see \cite{borceuxbourn2004malcev,bourn2017groups} for textbook accounts.

Sets equipped with a Malcev operation were given the name \emph{herd} by Lambek \cite{lambek1955groups,lambek1992ubiquity}. A \emph{heap} is a set equipped with a Malcev operation $t$ which is \emph{associative} in the sense that $t(a,b,t(c,d,e)) = t(t(a,b,c),d,e)$. If $G$ is a group, then $t(x,y,z) = xy^{-1}z$ determines a heap, and conversely every heap $G$ with specified element $e \in G$ determines a group with binary operation $x\cdot y = t(x,e,y)$ and neutral element $e$. Thus heaps are an unpointed or affine generalization of groups. Heaps were introduced earlier by Wagner \cite{wagner1953theory}, and can be traced further back to work of Pr\"ufer \cite{prufer1924theorie}. See \cite[Section 3]{hollingslawson2017wagner} for further historical discussion. 

In \cite[Page 3]{barrgrilletosdol1971exact}, Barr observed without proof that, in an exact category, every simplicial object satisfies the Kan condition if and only if every internal reflexive relation is an equivalence relation. A proof appears, for instance, in work of Carboni--Kelly--Pedicchio \cite[Theorem 4.2]{carbonikellypedicchio}. In \cite{jibladzepirashvili2002kan}, Jibladze--Pirashvili give a direct proof of the fact that, in the category of models for an algebraic theory, every simplicial object satisfies the Kan condition if and only if the theory contains a Malcev operation. They also prove that in this case every surjection of simplicial objects is a Kan fibration, generalizing a classical theorem for simplicial groups \cite[II.3.8 Proposition 1]{quillen1967homotopical}. 

It is a classical result of Moore \cite{moore1954homotopie} that every simplicial object in the $1$-category of groups is a Kan complex, and the above results generalize this to simplicial objects in the $1$-category of herds, i.e.\ to simplicial sets equipped with a Malcev operation. A \emph{loop} is a set $G$ equipped with a unital binary operation for which both shearing maps
\[
G\times G \to G\times G,\qquad (x,y) \mapsto (x,xy)\quad(x,y)\mapsto(xy,x)
\]
are bijections. Thus a connected $H$-space can be regarded as a loop up to homotopy, and the theory of loops is Malcev, see \cref{ex:hspaceherd}. In \cite{klaus2001simplicial}, Klaus proves that the minimal simplicial set model of a connected $H$-space inherits the structure of a strict simplicial loop. \cref{thm:intro:ukanspaces} shows that every space which admits a Malcev operation is equivalent to a disjoint union of $H$-spaces, and together these imply that every space which admits a Malcev operation is equivalent to a geometric realization of a simplicial set equipped with a Malcev operation.

The implicit use of Malcev theories in homotopy theory can be traced back at least to Quillen \cite[II.4.2]{quillen1967homotopical}. In his seminal work introducing model categories, Quillen constructed a closed model structure on the category of simplicial objects in an ordinary category $\dcat$ under certain assumptions on $\dcat$ which ask essentially that $\dcat$ is the category of models for a classical finitary theory or for a possibly infinitary Malcev theory. From this perspective, Malcev theories form the natural basis for analogues of animation in non-compactly generated settings; see \S\ref{ssec:quillen} for more information.

In \cite[Corollary 4.2.7]{lurie2011dag8}, \cite[Corollary A.5.6.4]{lurie_spectral_algebraic_geometry}, Lurie proves a fully homotopical generalization of Moore's theorem, demonstrating that simplicial objects in the $\infty$-category of grouplike $\bfE_1$-spaces always satisfy the Kan condition. The fully coherent $\bfE_1$-structure is used in the proof in order to apply the equivalence between grouplike $\bfE_1$-spaces and pointed connected spaces. \cref{theorem:introduction:universally_kan_objects_are_the_same_as_those_admitting_a_malcev_operation}, whose proof is inspired by Bousfield and Friedlander's work on the $\pi_\ast$-condition for simplicial spaces \cite[Appendix B]{bousfieldfriedlander1978homotopy}, shows that this full coherence is not necessary: any simplicial space which merely admits a Malcev operation already satisfies the Kan condition.

In \cite[Definition 4.2.9]{lurie2011dag8}, Lurie introduced the notion of a \emph{socle}. In our language, a socle is a theory $\pcat$, in the infinitary sense of \cref{definition:introduction_theory}, all of whose objects admit the structure of a cogroup (i.e.\ admit the structure of a grouplike $\bfE_1$-object in $\pcat^\op$). This assumption is used to guarantee that if $P \in \pcat$, then the corresponding free object $\nu P \in \Model_\pcat$ is strongly projective (\cref{def:intro:stronglyprojective}), allowing $\Model_\pcat$ to be identified as the free cocompletion of $\pcat$ under geometric realizations \cite[Proposition 4.2.15]{lurie2011dag8}. Socles have also been used by Brantner in \cite{brantner2017lubin} to control some of the non-compactly generated $\infty$-categories that arise in $K(n)$-local homotopy theory.

An ad hoc generalization of Lurie's socles were introduced under the name of ``Malcev theories'' in \cite{balderrama2021deformations}, where the strict cogroup condition was replaced by a strict coherd condition. This is a stricter condition than the one considered in the current work; nevertheless, there is some overlap in our treatment of Malcev theories with the treatment there. In these cases, where we can give either a short or better proof, we have opted to do so. 

As we explain in \S\ref{ssec:stronglyprojectivelygenerated}, a theory $\pcat$ has the property that every representable is strongly projective in $\Model_\pcat$ if and only if the theory $\pcat$ is Malcev in the sense of \cref{def:intro:malcev}. Thus the Malcev theories we consider here may be viewed as a \emph{maximal} generalization of socles. This generalization has the advantage that it has good closure properties, such as under slices and certain pullbacks, see \cref{lem:malcevclosure}. Moreover, a theory $\pcat$ is Malcev in our sense if and only if its homotopy category $\h\pcat$ is Malcev in the classical sense, a condition which is easily verified in practice.

Our work in this series uses the Malcev condition in essentially two orthogonal ways. In this paper, it is used primarily to deal with issues that arise when considering the $\infty$-category of models of an infinitary theory. In the sequel, it is used as a key technical assumption needed to develop a working deformation theory. That the Malcev condition is useful for doing deformation theory in universal algebra appears to go back to work of Smith \cite[Chapter 6]{smith1976malcev}; similar technical assumptions also appear in work on resolution model structures \cite{dwyerkanstover1993e2,bousfield2003cosimplicial}, which form the historical precursor to the modern theory of synthetic deformations.

\begin{question}
\label{question:properties_of_models_outside_of_the_malcev_case}
There are, of course, theories that do not satisfy the Malcev condition. These are quite well understood in the finitary case, with the fundamentals developed in Lurie's work on nonabelian derived categories \cite[\S5.5.8]{lurie_higher_topos_theory}. In the absence of the Malcev condition, working with infinitary theories and their models is much more delicate. We have separated out some of the higher universal algebra that can be developed in this generality, particularly in \S\ref{sec:generaltheories} and \S\ref{sec:genmonadic}, but advertise here a few open problems that our treatment leaves unresolved:
\begin{enumerate}
\item Let $\pcat$ be a theory. Does $\Model_\pcat$ admit all small colimits?
\item Let $S$ be a set and let $\pcat$ be a theory generated under coproducts and retracts by a family of objects $\{P_s : s\in S\}$. Is $\Model_\pcat$ monadic over $\spaces_{/S}$?
\item Let $\ccat$ be an $\infty$-category monadic over $\spaces_{/S}$ for a space $S$. Does $\ccat$ admit all small colimits?
\end{enumerate}
All of these admit positive answers classically: if $\pcat$ is a theory generated under coproducts and retracts by a set of objects $\{P_s : s\in S\}$, then the category $\Model_\pcat^\heartsuit$ of set-valued models of $\pcat$ is easily seen to be monadic over $\sets_{/S}$ (compare the proof of \cref{prop:sometimesmonadic}), and every such category admits all small colimits \cite[Example following Proposition 4]{linton1969coequalizers}. They can also be resolved in the presence of the Malcev condition, or in the presence of certain boundedness or presentability conditions.

The motivating example of an infinitary theory to which these assumptions do not apply is the theory of compact Hausdorff spaces: if $\pcat$ is the full subcateory of compact Hausdorff spaces spanned by the Stone--\v{C}ech compactifications of arbitrary sets, then it is classical that the category $\Model_\pcat^\heartsuit$ of set-valued models of $\pcat$ is equivalent to the category of compact Hausdorff spaces. We do not know, for example, whether the $\infty$-category $\Model_\pcat$ of space-valued models of $\pcat$ admits all small colimits.
\end{question}

Orthogonal to these considerations, we also ask the following.

\begin{question}
Is the full subcategory of $\pretheories$ spanned by the class of bounded theories (see \S\ref{ssec:bounded}) closed under small limits?
\end{question}

\subsection{Conventions}

All $\infty$-categories are implicitly assumed to be locally small, with the exception of $\infty$-categories defined as subcategories of the very large $\infty$-category $\largecatinfty$  of large $\infty$-categories, or defined as subcategories of $\Fun(\ccat^\op,\spaces)$ for a locally small $\infty$-category $\ccat$.

Given an $\infty$-category $\ccat$, we write $\smallpresheaves(\ccat)\subset\Fun(\ccat^\op,\spaces)$  for the full subcategory spanned by those presheaves which are 
\emph{small}; that is, which may be written as a small colimit of representable presheaves. This is a locally small $\infty$-category when $\ccat$ is, even though $\Fun(\ccat^\op,\spaces)$ need not be. If $\ccat$ is small, then $\smallpresheaves(\ccat) = \Fun(\ccat^\op,\spaces)$.

\section{Simplicial objects and Malcev operations}\label{sec:simplicialspaces}

The $\infty$-categories of models of theories that we consider in this paper can be thought of as $\infty$-categories of \emph{formal simplicial resolutions}. To realize this, we will need some of the theory of simplicial spaces. Our primary goal in this section is to study the interaction between Kan fibrations of simplicial spaces and Malcev operations.

We start in \S\ref{ssec:kanfibrations} by recalling the definition and some basic properties of Kan fibrations of simplicial spaces, and in \S\ref{ssec:herds} by recalling the definition and establishing some fundamental properties of Malcev operations in $\infty$-categories. We identify several useful characterizations and properties of Malcev operations on spaces and maps of spaces that will be needed later.

We then give the main technical theorem of the section in \S\ref{ssec:surjherd}: any effective epimorphism of simplicial spaces which admits a Malcev operation is a Kan fibration. Using this, we give in \S\ref{ssec:univkan} several alternate characterizations and closure properties of objects in $\infty$-categories that admit a Malcev operation.

\subsection{Kan fibrations}\label{ssec:kanfibrations}

\begin{definition}
Let $\ccat$ be an $\infty$-category. Given a simplicial object $\Xss \in \Fun(\Delta^\op,\ccat)$ and simplicial space $\Lambda \in \simplicialspaces$, consider the presheaf on $\ccat$ defined by
\[
\map_{\simplicialspaces}(\Lambda,\map_\ccat(\bs,\Xss))\colon \ccat^\op\to \spaces.
\]
The \emph{cotensor} $\Xss[\Lambda]$, when it exists, is the representing object for this presheaf. That is, it is an object of $\ccat$ together with equivalences
\[
\map_\ccat(Y,\Xss[\Lambda])\simeq \map_{\simplicialspaces}(\Lambda,\map_\ccat(Y,\Xss)),
\]
natural in $Y\in \ccat$.
\end{definition}

\begin{example}
\label{example:cotensors_out_of_simplex}
If $\Xss$ is a simplicial object, then
\[
\Xss[\Delta^n] = X_n,\qquad \Xss[\Lambda^2_1] = \lim\left(X_1 \xrightarrow{d_0} X_0 \xleftarrow{d_1} X_1\right).
\]
\end{example}

\begin{example}
If $\Xss$ and $\Lambda$ are simplicial spaces, then
\[
\Xss[\Lambda] = \map_{\simplicialspaces}(\Lambda,\Xss).
\]
Thus, the cotensor can be thought of as a generalization of the mapping space construction. 
\end{example}

\begin{remark}
\label{remark:formation_of_cotensors_is_continuous}
As is clear from the definition, the formation of cotensors is continuous in the second variable and takes colimits in the first variable to limits in $\ccat$. 
\end{remark}

\begin{remark}
Combining the first part of \cref{example:cotensors_out_of_simplex} and the second part of \cref{remark:formation_of_cotensors_is_continuous} we see that if $\ccat$ admits finite limits then $\Xss[\Lambda]$ exists for any simplicial object $\Xss \in \Fun(\Delta^\op,\ccat)$ and any simplicial set $\Lambda$ with finitely many nondegenerate simplices.
\end{remark}

\begin{definition}[{\cite[\S 6.2.3]{lurie_higher_topos_theory}}] 
\label{def:effepi}
Let $\ccat$ be an $\infty$-category with pullbacks. We say that a morphism $f\colon X \to Y$ is an \emph{effective epimorphism} if its \v{C}ech nerve
\begin{center}\begin{tikzcd}
Y&\ar[l]X &\ar[l,shift right]\ar[l,shift left]X\times_Y X&\ar[l,shift right]\ar[l,shift left]\ar[l]X\times_Y X \times_Y X&\ar[l,shift right=0.5mm]\ar[l,shift right=1.5mm]\ar[l,shift left=0.5mm]\ar[l,shift left=1.5mm]\cdots
\end{tikzcd}\end{center}
is a colimit diagram in $\ccat$.
\end{definition}

\begin{example}
\label{example:effective_epimorphism_of_spaces_are_pi0_epis}
A map $f\colon X \to Y$ of spaces is an effective epimorphism if and only if it induces a surjection $\pi_0 X \to \pi_0 Y$ on path components \cite[{Proposition 7.2.1.14}]{lurie_higher_topos_theory}.
\end{example}

\begin{example}
Any morphism $f\colon X \to Y$ which admits a section is an effective epimorphism \cite[Corollary 6.2.3.12]{lurie_higher_topos_theory}.
\end{example}

\begin{definition}\label{def:kancondition}
Let $\ccat$ be a $\infty$-category with finite limits. We say that 
\begin{enumerate}
\item A map $\Xss\to \Yss$ of simplicial objects in $\ccat$ is a \emph{Kan fibration} if
\[
\Xss[\Delta^n] \to \Yss[\Delta^n]\times_{\Yss[\Lambda^n_i]}\Xss[\Lambda^n_i]
\]
is an effective epimorphism for all $0\leq i \leq n$;
\item A simplicial object $\Xss$ \emph{satisfies the Kan condition} if $\Xss \to \ast$ is a Kan fibration; that is, if $\Xss[\Delta^n] \to \Xss[\Lambda^n_i]$ is an effective epimorphism for all $0\leq i \leq n$.
\end{enumerate}
\end{definition}

\begin{example}
\label{example:cech_nerve_is_a_kan_complex}
Let $f\colon X \to Y$ be a morphism in an $\infty$-category $\ccat$ with finite limits. Then the \v{C}ech nerve $\check{C}(f)$ of $f$ satisfies the Kan condition. In fact, the map
\[
\check{C}(f)[\Delta^n] \to \check{C}(f)[\Lambda^n_i]
\]
is an equivalence for all $0 \leq i \leq n$; see \cite[{\S 6.1.2}]{lurie_higher_topos_theory}.
\end{example}

\begin{example}
Let $K \in \ccat$ be an object considered as a constant simplicial object. Then a map $f\colon \Xss \to K$ is a Kan fibration if and only if $\Xss$ satisfies the Kan condition.
\end{example}

\begin{remark}
In the context of simplicial sets, \cref{def:kancondition} is classical. An analogous notion was considered by Seymour \cite{seymour1980kan} for simplicial topological spaces. Kan fibrations of sheaves of simplicial spaces (i.e.\ $\infty$-groupoids) were later introduced by Henriques in \cite[Definition 2.3]{henriques2008integrating}, and a thorough account of the theory of simplicial objects in an $\infty$-topos appears in work of Lurie \cite[Section A.5]{lurie_spectral_algebraic_geometry}.
\end{remark}

We are mostly interested in the case where the theory of Kan fibrations in $\ccat$ may be reduced to the theory of Kan fibrations of simplicial spaces.

\begin{remark}
Recall that a simplicial object $\Xss$ is said to be a \emph{hypercovering} if
\[
\Xss[\Delta^n] \to \Xss[\partial \Delta^n]
\]
is an effective epimorphism for all $n\geq 0$. By \cite[Corollary A.5.6.4]{lurie_spectral_algebraic_geometry}, a simplicial object $\Xss$ in an $\infty$-topos $\xcat$ satisfies the Kan condition if and only if it is a hypercovering of its geometric realization $|\Xss|$ (i.e.\ is a hypercovering in the slice category $\xcat_{/|\Xss|}$).
\end{remark}

The Kan condition is useful for us due to the compatibilities it guarantees between geometric realizations and limits: for example, every Kan fibration is a realization fibration in the sense of Rezk \cite{rezk_realization_fibration}. Specifically, we will make use of the following.

\begin{prop}
\label{prop:kanlimits}
Geometric realizations interact with the Kan condition as follows: 
\begin{enumerate}
\item Suppose given a small family $\{\Xss(i) : i \in I\}$ of simplicial spaces. If each $\Xss(i)$ satisfies the Kan condition, then the canonical map 
\[
|\prod_{i\in I}\Xss(i)| \to \prod_{i\in I}|\Xss(i)|
\]
is an equivalence.
\item Suppose given a cartesian square
\begin{center}\begin{tikzcd}
\Xss'\ar[r]\ar[d]&\Xss\ar[d,"f"]\\
\Yss'\ar[r]&\Yss
\end{tikzcd}\end{center}
of simplicial spaces. If $f$ is a Kan fibration, then the canonical map
\[
|\Xss'| \to |\Yss'|\times_{|\Yss|}|\Xss|
\]
is an equivalence.
\end{enumerate}
\end{prop}
\begin{proof}
(1)~~As $\Xss(i)$ satisfies the Kan condition, it is a hypercovering of $|\Xss(i)|$ by \cite[Corollary A.5.6.4]{lurie_spectral_algebraic_geometry}. As a consequence of \cref{example:effective_epimorphism_of_spaces_are_pi0_epis}, effective epimorphisms of spaces are closed under all small products. It follows that $\prod_{i\in I}\Xss(i)$ is a hypercovering of $\prod_{i\in I}|\Xss(i)|$. Thus $|\prod_{i\in I}\Xss(i)| \simeq \prod_{i\in I}|\Xss(i)|$ as claimed.

(2)~~This is \cite[Theorem A.5.4.1]{lurie_spectral_algebraic_geometry}.
\end{proof}

We also record the following for later use.

\begin{lemma}\label{lem:kanpb}
Consider a cartesian square
\begin{center}\begin{tikzcd}
\Xss'\ar[r]\ar[d,"f'"]&\Xss\ar[d,"f"]\\
\Yss'\ar[r,"g"]&\Yss
\end{tikzcd}\end{center}
of simplicial spaces. If $f$ is a Kan fibration, then so is $f'$, and the converse holds if $g$ is an effective epimorphism.
\end{lemma}
\begin{proof}
The assumptions provide for us a cube
\begin{center}\begin{tikzcd}
{\Xss'[\Delta^n]}\ar[rr]\ar[dd]\ar[dr]&&{\Xss[\Delta^n]}\ar[dd]\ar[dr]\\
&{\Xss'[\Lambda^n_i]}\ar[rr]\ar[dd]&&{\Xss[\Lambda^n_i]}\ar[dd]\\
{\Yss'[\Delta^n]}\ar[rr]\ar[dr]&&{\Yss[\Delta^n]}\ar[dr]\\
&{\Yss'[\Lambda^n_i]}\ar[rr]&&{\Yss[\Lambda^n_i]}
\end{tikzcd}\end{center}
in which the front and back faces are cartesian. It follows easily that if $\Xss[\Delta^n]\to \Yss[\Delta^n]\times_{\Yss[\Lambda^n_i]}\Xss[\Lambda^n_i]$ is an effective epimorphism then so is $\Xss'[\Delta^n]\to \Yss'[\Delta^n]\times_{\Yss'[\Lambda^n_i]}\Xss'[\Lambda^n_i]$, and that the converse holds if $\Yss'[\Delta^n]\to \Yss[\Delta^n]$ is an effective epimorphism.
\end{proof}

\subsection{Malcev operations}\label{ssec:herds}

Our goal for the rest of this section is to develop the theory of \emph{Malcev operations} in $\infty$-categories. Our main goal in this subsection is to introduce these operations, and study them in the $\infty$-category of spaces. In particular, in \cref{cor:ukanspaces} we characterize the spaces which admit a Malcev operation in terms of the property of being \emph{supersimple} (\cref{def:supersimple}), a natural strengthening of the notion of a componentwise simple space.

\begin{definition}
\label{def:malcevop}
A \emph{Malcev operation} on an object $X$ in an $\infty$-category is a ternary operation
\[
t\colon X \times X \times X \to X,
\]
together with a choice of two homotopies filling the triangles
\begin{equation}
\label{eq:malcdef}
\begin{tikzcd}
X\times X\ar[r,"X\times \Delta"]\ar[dr,"\pi_1"']&X\times X \times X\ar[d,"t"]&X\times X\ar[l,"\Delta\times X"']\ar[dl,"\pi_2"]\\
&X
\end{tikzcd}.
\end{equation}
An \emph{$H$-herd object} is an object equipped with a Malcev operation.
\end{definition}

\begin{remark}
In equations, the two homotopies of \cref{def:malcevop} can be written as 
\[
t(a,a,b)\simeq b\qquad\text{and}\qquad t(a,b,b)\simeq a.
\]
\end{remark}

\begin{remark}
We have required that the data of a Malcev operation includes specified homotopies filling in the two triangles of (\ref{eq:malcdef}) in order to properly formulate the sense in which certain constructions are natural in the Malcev operation. However, these two triangles are not required to satisfy any further coherences. As a consequence, an object of an $\infty$-category $\ccat$ \emph{admits} a Malcev operation if and only if it admits a Malcev operation up to homotopy, i.e.\ its image in the homotopy category $\h\ccat$ admits a Malcev operation.
\end{remark}

\begin{remark}
We will sometimes (in particular, \cref{ex:hspaceherd} and \cref{ex:pseudotorsor} below) abuse terminology and say that a map $t \colon X \times X \times X \rightarrow X$ is a Malcev operation if there exists a choice of homotopies making it into a Malcev operation; that is, if (\ref{eq:malcdef}) commutes up to homotopy.
\end{remark}

\begin{example}
\label{ex:hspaceherd}
Let $X$ be a grouplike $H$-space object, i.e.\ a pointed object equipped with a unital product $m$ for which the shearing map
\[
s \colonequals (X\times m) \circ (\Delta \times X)\colon X \times X \to X \times X \times X \to X \times X
\]
is an equivalence. Then a diagram chase shows that
\[
t \colonequals m \circ (X \times (\pi_2 \circ s^{-1})) \colon X \times X \times X \to X\times X \times X \to X\times X \to X
\]
is a Malcev operation.
If $m$ is associative, i.e.\ $X$ is a group object inside the homotopy category $\h\ccat$, then this Malcev operation is given in equations by
$
t(a,b,c) = a \cdot b^{-1} \cdot c.
$
\end{example}

\begin{example}
\label{ex:pseudotorsor}
Let $\ccat$ be an $\infty$-category with finite limits and terminal object $\cterminal$. Given an object $A \in \ccat$ and maps $x,y\colon \cterminal \to A$, there is the associated path object
\[
\Omega_{x,y}A = \lim\left(\cterminal \xrightarrow{x} A \xleftarrow{y} \cterminal\right).
\]
Given maps $x,y,z\colon \cterminal \to A$, there are path composition and path inversion maps
\[
\Omega_{x,y} A \times\Omega_{y,z} A \to \Omega_{x,z}A,\qquad \Omega_{x,y} A \to \Omega_{y,x} A,
\]
and the composite
\[
\Omega_{x,y} A \times \Omega_{x,y}A \times \Omega_{x,y}A \to \Omega_{x,y}A \times \Omega_{y,x} A \times \Omega_{x,y} A \to \Omega_{x,y} A,
\]
given informally by
\[
(f,g,h) \mapsto f \circ g^{-1}\circ h,
\]
is a Malcev operation on $\Omega_{x,y}A$.
\end{example}

\begin{example}
Specializing \cref{ex:pseudotorsor} to the case where $\ccat = \xcat$ is an $\infty$-topos and $A = \mathrm{B}G$ for some group object $G$, we learn that every principal $G$-bundle $p\colon E \to X$ admits a Malcev operation in the slice category $\xcat{/X}$. For example, every $n$-gerbe for $n\geq 2$ and every abelian $1$-gerbe admits a natural Malcev operation. 

More generally, one can show that every $n$-connective and $(2n-2)$-truncated object of an $\infty$-topos admits a natural Malcev operation. A version of this observation was, to our knowledge, first made by Vokřínek \cite{vokrinek2013computing,vokrinek2014heaps}.
\end{example}

Our goal for this rest of this subsection is to establish some properties of objects and morphisms which admit Malcev operations in the $\infty$-category of spaces. We begin by giving an elementary characterization of these spaces.

\begin{definition}\label{def:supersimple}
Say that a space $X \in \spaces$ is \emph{supersimple} if for all pointed connected spaces $A$ and $B$, the map
\[
\map(A\times B,X) \to \map(A\vee B,X),
\]
restricting along the inclusion $A\vee B \to A \times B$, is an effective epimorphism.
\end{definition}

\begin{remark}
Recall that if $X$ is a space, $x\in X$, $a \in \pi_n (X,x)$ and $b \in \pi_r (X,x)$, then the \emph{Whitehead product} $[a,b] \in \pi_{n+r-1}(X,x)$ is defined as the top horizontal composite in a diagram
\begin{center}\begin{tikzcd}
S^{n+r-1}\ar[r,"w"]\ar[d]&S^n\vee S^r\ar[r,"a\vee b"]\ar[d]&X\\
\ast\ar[r]&S^n\times S^r
\end{tikzcd},\end{center}
where $w$ is the top attaching map for the standard cell structure on $S^n\times S^r$, defined so that the square is cocartesian. It follows that $[a,b] = 0$ for all $a \in \pi_n(X,x)$ and $a\in \pi_r(X,x)$ precisely when
\[
\map_\ast(S^n\times S^r,X) \to \map_\ast(S^n\vee S^r,X)
\]
is an effective epimorphism, and therefore $[a,b] = 0$ for all basepoints $x\in X$ and $a\in \pi_n(X,x)$ and $b\in \pi_r(X,x)$ precisely when
\[
\map(S^n\times S^r,X) \to \map(S^n\vee S^r,X)
\]
is an effective epimorphism. Thus the supersimple spaces form a natural subclass of those spaces with vanishing Whitehead products \cite{porter1965spaces}, itself a natural subclass of the spaces with simple path components. Note, however, that we do \emph{not} require supersimple spaces to be connected.

As we will see below, the supersimple spaces are in many ways better behaved than either the spaces with vanishing Whitehead products or the componentwise simple spaces. For example, if $X$ is supersimple then so is the cotensor $X^K$ for any space $K$, whereas the corresponding statement where $X$ is assumed just to have vanishing Whitehead products fails \cite[Example 6.6]{luptonsmith2010whitehead}.
\end{remark}

\begin{lemma}\label{lem:coprodmalcev}
Let $\ccat$ be an $\infty$-category with finite products and with coproducts that distribute across them. If $\{X_i : i\in I\}$ is a set of objects in $\ccat$ which admit Malcev operations, then $\coprod_{i\in I}X_i$ admits a Malcev operation.
\end{lemma}
\begin{proof}
We construct a Malcev operation
\[
t\colon (\coprod_{i\in I}X_i)^{\times 3} \to \coprod_{i\in I}X_i
\]
by hand. As coproducts distribute across products, we may identify
\[
(\coprod_{i\in I}X_i)^{\times 3}\simeq \coprod_{(i,j,k)\in I^3}X_i\times X_j \times X_k.
\]
It therefore suffices to specify maps
\[
t_{i,j,k}\colon X_i\times X_j \times X_k \to X_l
\]
for each triple $i,j,k \in I$, where $l\in I$ is some element depending on this triple. We specify this in cases.
If $i = j = k$, then
\[
t_{i,i,i}\colon X_i\times X_i \times X_i \to X_i
\]
is any Malcev operation on $X_i$. If $i = j \neq k$, then
\[
t_{i,i,k} = \pi_3\colon X_i\times X_i\times X_k \to X_k
\]
is the projection. If $i \neq j = k$, then 
\[
t_{i,k,k} = \pi_1\colon X_i\times X_k \times X_k \to X_i
\]
is the projection. If $i\neq j$ and $j\neq k$, then $t_{i,j,k}$ can be anything, such as the projection
\[
t_{i,j,k} = \pi_2\colon X_i\times X_j\times X_k \to X_j.
\]
An easy calculation shows that this defines a Malcev operation on $\coprod_{i\in I}X_i$.
\end{proof}

\begin{warning}
\cref{lem:coprodmalcev} implies, for example, that if $X$ and $Y$ are $H$-herds in spaces then $X \coprod Y$ admits the structure of an $H$-herd in a natural way. However, beware that $X\coprod Y$ is \emph{not} the coproduct of $X$ and $Y$ in the $\infty$-category of $H$-herds unless $X$ or $Y$ is empty.
\end{warning}

\begin{theorem}
\label{cor:ukanspaces}
Given a space $X$, the following are equivalent:
\begin{enumerate}
\item $X$ admits a Malcev operation.
\item Every path component of $X$ admits the structure of an $H$-space.
\item Let $A$ and $B$ be pointed connected spaces. Then the map
\[
\map(A\times B,X)\to \map(A\vee B,X),
\]
restricting along the inclusion $A\vee B \to A\times B$, admits a section.
\item $X$ is supersimple.
\end{enumerate}
Moreover, in (1)$\Rightarrow$(3), the choice of a section can be made natural in the choice of Malcev operation on $X$.
\end{theorem}
\begin{proof}
(1)$\Rightarrow$(2). By assumption, $X$ admits a Malcev operation
\[
t\colon X \times X \times X \to X.
\]
By inspection, we see that a ternary operation $t$ on $X$ is a Malcev operation if and only if its adjoint
\[
\hat{t}\colon X \to \map(X\times X,X),\qquad \hat{t}(x)(a,b) = t(a,x,b)
\]
sends every point $x\in X$ to a unital product on the pointed space $(X,x)$. If $a,b\in X$ are in the same path component as $x$, then choices of paths $a \sim x$ and $b\sim x$ induce homotopies
\[
t(a,x,b) \sim t(x,x,x)\simeq x,
\]
and therefore this unital product restricts to give the path component of $x$ the structure of an $H$-space.

(2)$\Rightarrow$(1). By \cref{lem:coprodmalcev}, it suffices to prove that connected $H$-spaces admit Malcev operations. A connected $H$-space is grouplike, so this follows from \cref{ex:hspaceherd}.

(2)$\Rightarrow$(3). A section is given by the composite
\[
s\colon \map(A\vee B,X) \to \map(A,X)\times X\times \map(B,X) \to \map(A\times B,X\times X \times X) \to \map(A\times B,X)
\]
sending a map
\[
f\vee g \colon A \vee B \to X
\]
to the map
\[
s(f\vee g)\colon A \times B \to X,\qquad s(f\vee g)(a,b) = t(f(a),f(\ast)=g(\ast),g(b)).
\]

(3)$\Rightarrow$(4) holds as every morphism which admits a section is an effective epimorphism.

(4)$\Rightarrow$(2). Without loss of generality we may suppose that $X$ is connected, and must show that $X$ admits the structure of an $H$-space. An $H$-space structure on $X$ is exactly a lift of the fold map $X\vee X \to X$ through $X\vee X \to X\times X$, which exists by taking $A = B = X$.
\end{proof}

We now record some additional facts about supersimple spaces and morphisms of $H$-herds that will be used in the sequel. Perhaps the most useful general fact is the following. For a space $A$ and point $a \in A$, write $A_a\subset A$ for the path component of $a$.

\begin{prop}
\label{prop:locallysplit}
Let $f\colon X \to Y$ be a map of $H$-herds, and suppose that the underlying map of spaces is equipped with section $s \colon Y \to X$. Then $f$ is locally split in the following sense: given $x \in X$, if we write
\[
F \colonequals \{ f(x) \} \times_{Y} X
\]
for the fibre over $f(x)$ and $i \colon F_{x} \to X_x$ for the inclusion of the path component of $x$, then
\[
 F_{x} \times Y_{f(x)}\to X_x,\qquad (x',y) \mapsto t(i(x'),x,s(y))
\]
is an equivalence.
\end{prop}
\begin{proof}
For $x \in X$, define
\[
s_x\colon Y \to X,\qquad s_x(y) = t(x,sf(x),s(y)).
\]
As
\[
f(s_x(y)) = f(t(x,sf(x),s(y)) \sim t(f(x),fsf(x),fs(y)) \sim t(f(x),f(x),y) \sim y,
\]
we see that $s_x$ is again a section to $f$. As
\[
s_x(f(x)) = t(x,sf(x),sf(x))\sim x,
\]
we see that $s_x$ defines a section to the restriction of $f$ to $X_x \to Y_{f(x)}$. We therefore reduce to the case where $f\colon X \to Y$ is a morphism of connected $H$-spaces with section $s$. In this case the map is the usual equivalence
\[
 F \times Y \to X,\qquad (x',y) \mapsto x' \cdot s(y).\qedhere
\]
\end{proof}

\begin{cor}\label{cor:truncatemalcevsquare}
Consider a cartesian square
\begin{equation*}\label{eq:malcevsquare}
\begin{tikzcd}
X'\ar[r,"g'"]\ar[d,"f'"']\ar[dr,phantom,"\sigma"]&X\ar[d,"f"]\\
Y'\ar[r,"g"']&Y
\end{tikzcd}
\end{equation*}
of spaces. Suppose that $f$ is a map of $H$-herds and that the underlying map of spaces is equipped with a section $s\colon Y \to X$.
\begin{enumerate}
\item Let $\spaces_0\subset\spaces$ be a full subcategory containing $\sigma$ and closed under subobjects. If $F\colon \spaces_0\to\spaces$ preserves coproducts and finite products, then $F(\sigma)$ is cartesian.
\item In particular, $\sigma$ remains cartesian after $n$-truncation for all $n$.
\item $(f',g')\colon X' \to Y' \times X$ admits a retraction, natural in the choice of Malcev operation on $f$ and section $s$. The associated idempotent on $Y'\times X$ may be described as
\[
(y',x) \mapsto (y',t(x,sf(x),sg(y')))
\]
\item If $g$ admits a retraction, then $g'$ admits a retraction, natural in the choice of Malcev operation on $f$, section $s$, and retraction $r$ of $g$. The associated idempotent on $X$ may be described by
\[
x \mapsto t(x,sf(x),sgrf(x)).
\]
Moreover, the resulting map
\[
X \to X'\times_{Y'} Y
\]
is an equivalence.
\end{enumerate}
\end{cor}
\begin{proof}
By \cref{prop:locallysplit}, such a square is equivalent to a disjoint union of squares of the form
\begin{center}\begin{tikzcd}
F\times Y'\ar[r,"F\times g"]\ar[d,"\pi_2"]&F\times Y\ar[d,"\pi_2"]\\
Y'\ar[r,"g"]&Y
\end{tikzcd}.\end{center}
This square clearly remains cartesian after application of any product-preserving functor, implying (1,2). For (3), the map
\[
Y'\times F \to Y'\times Y\times F
\]
admits a retraction projecting away from $Y$, and for (4), if $g$ admits a retraction $r$ then $F \times g$ admits the retraction $F \times r$, in which case
\[
F\times Y \to (F\times Y')\times_{Y'}Y
\]
is an equivalence as claimed. Tracing through the proof of \cref{prop:locallysplit} provides the given formulas for these retractions.
\end{proof}

\begin{example}\label{ex:pointedmaps}
Let $(A,a)$ be a pointed space and $(X,x)$ be a pointed supersimple space. Then
\begin{center}\begin{tikzcd}
\pi_0 \map_\ast(A,X)\ar[r]\ar[d]&\pi_0 \map(A,X)\ar[d,"f\mapsto f(a)"]\\
\{x\}\ar[r]&\pi_0 X
\end{tikzcd}\end{center}
is cartesian.
\end{example}

\begin{example}\label{ex:splitproduct}
Let $X$ be a supersimple space. For any pointed connected spaces $A$ and $B$, there is a cartesian square
\begin{center}\begin{tikzcd}
X^{A\wedge B}\ar[r]\ar[d]&X^{A\times B}\ar[d]\\
X\ar[r]&X^{A\vee B}
\end{tikzcd}.\end{center}
As $A\vee B$ is pointed, $X\to X^{A\vee B}$ is equipped with a retraction. As $X$ is supersimple, $X^{A\times B} \to X^{A\vee B}$ admits a section, and therefore \cref{cor:truncatemalcevsquare} implies that there exists an equivalence
\[
X^{A\times B}\simeq X^{A\vee (A\wedge B)\vee B}.
\]
A choice of Malcev operation on $X$ provides a choice of section of $X^{A\times B} \to X^{A\vee B}$ by \cref{cor:ukanspaces}, in which case the associated retraction of $X^{A\wedge B} \to X^{A\times B}$ sends a map $f\colon A \times B \to X$ to the map
\[
\tilde{f}\colon A \wedge B \to X,\qquad \tilde{f}(a\wedge b) = t(f(a,b),t(f(a,\ast),f(\ast,\ast),f(\ast,b)),f(\ast,\ast)).
\]
When $t$ is the Malcev operation associated to an $H$-group structure on $X$ with unit $\ast$, this formula reduces to the familiar construction $a\wedge b \mapsto f(a,b) \cdot (f(a,\ast) \cdot  f(\ast,b))^{-1}$.
\end{example}

\subsection{Interaction with the Kan condition}\label{ssec:surjherd}

We now wish to prove the following theorem: 

\begin{theorem}
\label{thm:surjherd}
Let $f\colon \Xss \to \Yss$ be an effective epimorphism of simplicial $H$-herds; in other words, an effective epimorphism of simplicial spaces which admits a Malcev operation in the arrow $\infty$-category $\Fun(\Delta^\op,\spaces)^{[1]}$. Then $f$ is a Kan fibration.
\end{theorem}

Our proof of \cref{thm:surjherd} is inspired by Bousfield and Friedlander's work on simplicial spaces \cite[Appendix B]{bousfieldfriedlander1978homotopy}. We require a couple preliminaries.

\begin{definition}
Let $T$ be a finite linear order and $S\subset T$ be a subset. Then
\[
\Lambda^T_S\subset\Delta^T
\]
is defined to be the union of those faces which are opposite to some vertex $s\in S$.
\end{definition}

\begin{example}
We have
\[
\Lambda^T_{\{t\}} \simeq\Delta^{T\setminus\{t\}},\qquad \Lambda^{[n]}_{[n]\setminus\{i\}}\simeq \Lambda^n_i.
\]
\end{example}

\begin{lemma}
Let $T$ be a finite linear order and $S\subset T$ be a subset. Suppose given $S' = S \cup \{s'\}$ for some $s' \in T \setminus S$. Then the square
\begin{center}\begin{tikzcd}
\Lambda^{T\setminus\{s'\}}_S\ar[r]\ar[d]&\Lambda^T_S\ar[d]\\
\Delta^{T\setminus\{s'\}}\ar[r]&\Lambda^T_{S'}
\end{tikzcd}\end{center}
is cartesian and cocartesian.
\end{lemma}
\begin{proof}
By definition, $\Lambda^T_{S'} \subset \Delta^T$ may be obtained from $\Lambda^T_S\subset\Delta^T$ by adjoining the face opposite to $s'$, realized by $\Delta^{T\setminus \{s'\}} \subset\Delta^T$. The lemma then amounts to the observation that $\Delta^{T\setminus\{s'\}}\cap \Lambda^T_S = \Lambda^{T\setminus\{s'\}}_S$ consists of those faces of $\Delta^{T\setminus\{s'\}}$ contained in a face of $\Delta^T$ opposite to some $s\in S$.
\end{proof}

\begin{lemma}
A simplicial space $\Xss$ satisfies the Kan condition if and only if, for every finite linear order $T$ and nonempty proper subset $S\subsetneq T$, the map
\[
\Xss[\Delta^T] \to \Xss[\Delta^T_S]
\]
is an effective epimorphism.
\end{lemma}
\begin{proof}
As $\Delta^{[n]}_{[n]\setminus\{i\}} \simeq \Lambda^n_i$, clearly if these maps are all effective epimorphisms then $\Xss$ satisfies the Kan condition.

For the converse, we induct up on $|T|$ and, with $|T|$ fixed, induct down on $|S|$. For the base case $|T| = 2$, necessarily $|S| = |T| - 1$. In general, if $|S| = |T| - 1$, then $\Xss[\Delta^T] \to \Xss[\Lambda^T_S]$ is isomorphic to $\Xss[\Delta^n] \to \Xss[\Lambda^n_i]$ for some $0\leq i \leq n$, which is an effective epimorphism as $\Xss$ satisfies the Kan condition. This establishes the base case for both $S$ and $T$.

For the inductive step, it suffices to prove that if $S' = S \cup \{s'\} \subsetneq T$ for some $s' \in T \setminus S$, then $\Xss[\Lambda^T_{S'}] \to \Xss[\Lambda^T_S]$ is an effective epimorphism. Consider the cartesian square
\begin{center}\begin{tikzcd}
{\Xss[\Lambda^T_{S'}]}\ar[r]\ar[d]&{\Xss[\Delta^{T\setminus\{s'\}}]}\ar[d]\\
{\Xss[\Lambda^T_S]}\ar[r]&{\Xss[\Lambda^{T\setminus\{s'\}}_S]}
\end{tikzcd}.\end{center}
The right vertical map is an effective epimorphism by induction, hence so is the left vertical map.
\end{proof}

\begin{prop}\label{prop:pimatching}
The following assertions hold for all $k\geq 0$:
\begin{enumerate}
\item[($\ast_k$)] For every simplicial $H$-herd $\Xss$, finite linear order $T$, and subset $S\subsetneq T$ of cardinality $k$, the map $\Xss[\Delta^T] \to \Xss[\Lambda^T_S]$ admits a section.
\item[($\ast'_k$)] For every simplicial $H$-herd $\Xss$, finite linear order $T$, and $S\subset T$ of cardinality $k$, the map $\pi_0(\Xss[\Lambda^T_S]) \to (\pi_0 \Xss)[\Lambda^T_S]$ is a bijection.
\end{enumerate}
\end{prop}
\begin{proof}
We will prove that $(\ast_{k-1} \text{ and }\ast_{k-1}') \Rightarrow (\ast_{k}')$ and that $(\ast_k') \Rightarrow (\ast_k)$. The equivalence $\Lambda^T_{\{i\}}\simeq\Delta^{T\setminus\{i\}}$ implies that $(\ast_1')$ is always satisfied, and so this implies that both statements hold for all $k\geq 1$.

($\ast_{k-1}$ and $\ast_{k-1}'$)$\Rightarrow$($\ast_{k}'$). Suppose that $|S'| = k$ and write $S' = S \cup \{s'\}\subsetneq T$ for some $s' \in T\setminus S$. Consider the cartesian square
\[
\begin{tikzcd}
{\Xss[\Lambda^T_{S'}]}\ar[r]\ar[d]&{\Xss[\Delta^{T\setminus\{s'\}}]}\ar[d]\\
{\Xss[\Lambda^T_S]}\ar[r]&{\Xss[\Lambda^{T\setminus\{s'\}}_S]}
\end{tikzcd}.
\]
By ($\ast_{k-1}$), the right vertical map admits a section. By \cref{cor:truncatemalcevsquare}, the square therefore remains cartesian on taking path components. Combined with $(\ast_{k-1}')$, this implies that
\begin{align*}
\pi_0(\Xss[\Lambda^T_{S'}]) &\cong \pi_0(\Xss[\Lambda^T_{S}])\times_{\pi_0(\Xss[\Lambda^{T\setminus\{s'\}}_S])} \pi_0 \Xss[\Delta^{T\setminus\{s'\}}_S]\\
&\cong (\pi_0 \Xss)[\Lambda^T_S]\times_{(\pi_0 \Xss)[\Lambda^{T\setminus\{s'\}}_S]}(\pi_0 \Xss)[\Delta^{T\setminus\{s'\}}_S] \cong (\pi_0 \Xss)[\Lambda^T_{S'}].
\end{align*}
This establishes $(\ast_k')$.

($\ast_k'$)$\Rightarrow$($\ast_k$). Let $T$ be a finite linear order and $S\subsetneq T$ be a subset of cardinality $k$, so that we wish to prove that $\Xss[\Delta^T] \to \Xss[\Delta^T_S]$ admits a section. As $\Xss$ is a simplicial $H$-herd, it follows that if $A$ is any space then $\pi_0 \map(A,\Xss)$ is a simplicial herd, which is therefore a Kan complex by \S\ref{subsection:malcevkanhistory}. In particular, the composite
\[
\pi_0 \map(A,\Xss[\Delta^T]) \to \pi_0\map(A,\Xss[\Delta^T_S]) \to (\pi_0 \map(A,\Xss))[\Delta^T_S]
\]
is a surjection. As $\map(A,\Xss)$ is a simplicial $H$-herd, $(\ast_k')$ implies that
\[
\pi_0\map(A,\Xss[\Delta^T_S]) \to (\pi_0 \map(A,\Xss))[\Delta^T_S]
\]
is a bijection, and therefore
\[
\map(A,\Xss[\Delta^T]) \to \map(A,\Xss[\Delta^T_S])
\]
is an effective epimorphism. As $A$ was arbitrary, we may take $A = \Xss[\Delta^T_S]$ to deduce that $\Xss[\Delta^T] \to \Xss[\Delta^T_S]$ admits a section.
\end{proof}

We can now give the following.

\begin{proof}[Proof of \cref{thm:surjherd}]
By \cref{lem:kanpb}, after possibly pulling $f$ back along itself, we may suppose that $f$ admits a section. To prove that $f$ is a Kan fibration, we must prove that
\[
\Xss[\Delta^n] \to \Yss[\Delta^n]\times_{\Yss[\Lambda^n_i]}\Xss[\Lambda^n_i]
\]
is an effective epimorphism for all $0\leq i \leq n$. As $\Xss \to \Yss$ admits a section, so does $\Xss[\Lambda^n_i] \to \Yss[\Lambda^n_i]$. Therefore \cref{cor:truncatemalcevsquare} implies that 
\[
\pi_0 \left(\Yss[\Delta^n]\times_{\Yss[\Lambda^n_i]}\Xss[\Lambda^n_i]\right) \to \pi_0 (\Yss[\Delta^n])\times_{\pi_0 (\Yss[\Lambda^n_i])}\pi_0 (\Xss[\Lambda^n_i])
\]
is a bijection. By \cref{prop:pimatching}, the map
\[
\pi_0(\Zss[\Lambda^n_i]) \to (\pi_0 \Zss)[\Lambda^n_i]
\]
is a bijection for any simplicial $H$-herd $\Zss$. Applied to $\Xss$ and $\Yss$, this combines with the above to prove that 
\[
\pi_0 \left(\Yss[\Delta^n]\times_{\Yss[\Lambda^n_i]}\Xss[\Lambda^n_i]\right) \to (\pi_0 \Yss)[\Delta^n]\times_{(\pi_0 \Yss)[\Lambda^n_i]}(\pi_0 \Xss)[\Lambda^n_i]
\]
is a bijection. It follows that $\Xss \to \Yss$ is a Kan fibration if and only if $\pi_0 \Xss \to \pi_0 \Yss$ is a Kan fibration, which is classical (see \S\ref{subsection:malcevkanhistory}).
\end{proof}

\subsection{Universally Kan objects}\label{ssec:univkan}

In this subsection, we formulate and prove a precise sense in which admitting a Malcev operation is the minimal algebraic condition that guarantees a simplicial space satisfies the Kan condition.

\begin{definition}
\label{def:univkan}
Let $\ccat$ be an $\infty$-category with finite products. We say  that an object $X\in \ccat$ is \emph{universally Kan} if the following condition is satisfied:
\begin{enumerate}
\item[{($\ast$)}] For every finite product-preserving functor $F\colon \ccat\to\Fun(\Delta^\op,\spaces)$, the simplicial space $F(X)$ satisfies the Kan condition 
\end{enumerate}
\end{definition}

\begin{theorem}
\label{prop:univkancharacterize}
Let $\ccat$ be an $\infty$-category with finite products. Given $X\in\ccat$, the following are equivalent:
\begin{enumerate}
\item $X$ is universally Kan.
\item For every finite product-preserving functor $F\colon \ccat\to\Fun(\Delta^\op,\sets)$, the simplicial set $F(X)$ satisfies the Kan condition.
\item For every cosimplicial object $K^\bullet \in \Fun(\Delta,\ccat)$, the simplicial space $\map_\ccat(K^\bullet,X)$ satisfies the Kan condition.
\item Let $S^1_\bullet = \Delta^1/\partial\Delta^1$ denote the simplicial circle. Then the simplicial set $\pi_0 \map_\ccat(X^{S^1_\bullet},X)$ satisfies the horn filling condition with respect to $\Lambda^2_1\subset\Delta^2$.
\item $X$ admits a Malcev operation.
\item For every finite product-preserving functor $F\colon \ccat \to \Fun(\Delta^\op,\spaces)^{[1]}$, if $F(X)$ is an effective epimorphism then it is a Kan fibration.
\end{enumerate}
\end{theorem}
\begin{proof}
(1)$\Rightarrow$(2 and 3). Clear.

(2)$\Rightarrow$(4). Clear.

(3)$\Rightarrow$(4). In general, if $\Zss$ is a simplicial space, then
\[
\pi_0(\Zss[\Lambda^2_1]) \to (\pi_0 \Zss)[\Lambda^2_1]
\]
is an effective epimorphism. In particular, if $\Zss$ satisfies the Kan condition then $\pi_0 \Zss$ satisfies the horn filling condition with respect to $\Lambda^2_1 \subset\Delta^2$. Therefore (3) implies (4) by taking $\Zss = \map_\ccat(X^{S^1_\bullet},X)$.

(4)$\Rightarrow$(5). This argument is due to Jibladze--Pirashvili \cite{jibladzepirashvili2002kan}. By assumption,
\[
\pi_0 \map_\ccat(X^{S^1_\bullet},X)[\Delta^2] \to (\pi_0 \map_\ccat(X^{S^1_\bullet},X))[\Lambda^2_1]
\]
is an effective epimorphism. In general, a point of $(\pi_0 \map_\ccat(X^{S^1_\bullet},Y))[\Lambda^2_1]$ consists of two maps $f_1,f_2\colon X\times X \to Y$ for which the outer diagram in
\begin{center}\begin{tikzcd}
X\ar[r,"\Delta"]\ar[d,"\Delta"]&X\times X\ar[d,"X\times\Delta"]\ar[ddr,bend left,dashed,"f_1"]\\
X\times X\ar[r,"\Delta\times X"]\ar[drr,bend right,dashed,"f_2"']&X\times X \times X\ar[dr,dotted]\\
&&Y
\end{tikzcd}\end{center}
commutes up to unspecified homotopy. A lift to $\pi_0 \map_\ccat(X^{S^1_\bullet},Y)[\Delta^2]$ is given by a map $X \times X \times X \to Y$ making the outer triangles commute up to unspecified homotopy. Taking $Y = X$ and $f_i = \pi_i$ the projection, such lifts are exactly Malcev operations on $X$. 

(5)$\Rightarrow$(6). If $X$ admits a Malcev operation and $F\colon \ccat\to\Fun(\Delta,\spaces)^{[1]}$ preserves finite products, then $F(X)$ admits the structure of a map between simplicial herds. Therefore if $F(X)$ is an effective epimorphism then it is a Kan fibration by \cref{thm:surjherd}.

(6)$\Rightarrow$(1). Let $F\colon \ccat\to\Fun(\Delta^\op,\spaces)$ be a finite product preserving-functor. If $F(X) = \emptyset$, then $F(X)$ satisfies the Kan condition. Otherwise, $F(X) \to \ast$ is an effective epimorphism, and the unique natural transformation $F \to \ast$ determines a finite product-preserving functor $\ccat \to \Fun(\Delta^\op,\spaces)^{[1]}$. It follows that $F(X) \to \ast$ is a Kan fibration, and thus $F(X)$ satisfies the Kan condition.
\end{proof}

\begin{cor}
A space $X$ is universally Kan if and only if it is supersimple.
\end{cor}

\begin{proof}
Combine \cref{prop:univkancharacterize} and \cref{cor:ukanspaces}.
\end{proof}

\begin{example}\label{ex:additiveukan}
Let $\ccat$ be an additive $\infty$-category. Then every object of $\ccat$ is universally Kan.
\end{example}

Universally Kan objects satisfy a number of pleasant closure properties.

\begin{prop}\label{prop:univkanclosure}
Let $\ccat$ be an $\infty$-category with finite products.
\begin{enumerate}
\item If $\{X_i \in \ccat : i\in I\}$ are universally Kan and the product $\prod_{i\in I}X_i$ exists in $\ccat$, then it is universally Kan.
\item Let $f\colon X \to Y$ be a morphism in $\ccat$ which admits a section. If $X$ is universally Kan, then $f$ is universally Kan as an object of $\ccat^{[1]}$.
\item If $X \in \ccat$ is universally Kan and $Y$ is a retract of $X$, then $Y$ is universally Kan.
\item Suppose that $X \in \Fun(\jcat,\ccat)$ is universally Kan. If $\lim X \in \ccat$ exists, then it is universally Kan.
\item If $\ccat$ admits coproducts which distribute across products and $\{X_i : i\in I\}$ is a set of universally Kan objects in $\ccat$, then the coproduct $\coprod_{i\in I}X_i$ is universally Kan.
\item If $F\colon \ccat\to\dcat$ preserves finite products and $X\in \ccat$ is universally Kan, then $F(X)$ is universally Kan.
\item If $X\in \ccat$ is universally Kan and $K$ is a space for which the cotensor $X^K$ exists, then it is universally Kan.
\item If $X \in \ccat$, then $X$ is universally Kan if and only if its image in $\h\ccat$ is universally Kan.
\item Given a map $f\colon Y \to X$ in $\ccat$, the object $X$ is universally Kan as an object of $\ccat$ if and only if it is universally Kan as an object of $\ccat_{Y/}$.
\item Suppose given a cartesian square
\begin{center}\begin{tikzcd}
\ccat'\ar[r,"g'"]\ar[d,"f'"]&\ccat\ar[d,"f"]\\
\dcat'\ar[r,"g"]&\dcat
\end{tikzcd}\end{center}
of $\infty$-categories which admit finite products and functors which preserve them. If $f$ is full, then an object $X' \in \ccat'$ is universally Kan if and only if $g'(X') \in \ccat$ and $f'(X') \in \dcat'$ are universally Kan.
\item Suppose given a family $\{\ccat_i : i \in I\}$ of $\infty$-categories which admit finite products. An object $\{X_i : i \in I\} \in \prod_{i\in I}\ccat_i$ is universally Kan if and only if $X_i\in \ccat_i$ is universally Kan for all $i\in I$.
\end{enumerate}
\end{prop}
\begin{proof}
(1)~~This holds as simplicial spaces which satisfy the Kan condition are closed under products.

(2)~~Let $s\colon Y \to X$ be a section, so that $f\circ s = \id_Y$. If $t\colon X \times X \times X \to X$ is a Malcev operation on $X$, then the composite
\[
f \circ t \circ (s \times s \times s)\colon Y\times Y \times Y \to X \times X \times X \to X \to Y
\]
is a Malcev operation on $Y$ compatible with $f$. Therefore $f$ admits a Malcev operation, and so is universally Kan.

(3)~~This holds as simplicial spaces which satisfy the Kan condition are closed under retracts.

(4)~~If $X \in \Fun(\jcat,\ccat)$ is universally Kan, then it admits a Malcev operation. If $\lim X$ exists in $\ccat$, then $ \lim (X^3) \to (\lim X)^3$ is an equivalence, and therefore a Malcev operation on $X$ induces one on $\lim X$, implying that $\lim X$ is universally Kan.

(5)~~This follows from \cref{prop:univkancharacterize} and \cref{lem:coprodmalcev}.

(6)~~This follows from \cref{prop:univkancharacterize}, as if $X\in\ccat$ admits a Malcev operation and $F$ preserves products then $F(X)$ admits a Malcev operation.

(7)~~As the diagonal $\ccat \to \Fun(K,\ccat)$ preserves finite products, this follows from (3) and (5).

(8)~~This follows from \cref{prop:univkancharacterize}, as $X \in \ccat$ admits a Malcev operation if and only if $\tau(X) \in \h\ccat$ admits a Malcev operation.

(9)~~As the forgetful functor $\ccat_{Y/} \to \ccat$ preserves finite products, it follows from (5) that if $X$ is universally Kan in $\ccat_{Y/}$ then it is universally Kan in $\ccat$. Suppose conversely that $X$ is universally Kan in $\ccat$. By \cref{prop:univkancharacterize}, it suffices to prove that if $K^\bullet \in \Fun(\Delta,\ccat_{Y/})$ is a cosimplicial object then $\map_{\ccat_{Y/}}(K^\bullet,X)$ satisfies the Kan condition. This simplicial space sits in a cartesian square
\begin{center}\begin{tikzcd}
\map_{\ccat_{Y/}}(K^\bullet,X)\ar[r]\ar[d]&\map_\ccat(K^\bullet,X)\ar[d]\\
\{f\}\ar[r]&\map_\ccat(Y,X)
\end{tikzcd}.\end{center}
As every map from a simplicial space satisfying the Kan condition to a constant simplicial space is a Kan fibration, it follows from \cref{lem:kanpb} that $\map_{\ccat_{Y/}}(K^\bullet,X)\to\{f\}$ is a Kan fibration, meaning that $\map_{\ccat_{Y/}}(K^\bullet,X)$ satisfies the Kan condition as claimed.

(10)~~Essentially the same argument as (9) applies. As $g'$ and $f'$ preserve finite products, if $X'\in \ccat'$ is universally Kan then by (5) it follows that $g'(X')$ and $f'(X')$ are universally Kan. Conversely, suppose given $X' \in \ccat'$ for which $g'(X')$ and $f'(X')$ are universally Kan. By \cref{prop:univkancharacterize}, it suffices to prove that if $\Kcs \in \ccat'$ is a cosimplicial object then $\map_{\ccat'}(K^\bullet,X)$ satisfies the Kan condition. This simplicial space sits in a cartesian square of the form
\begin{center}\begin{tikzcd}
\map_{\ccat'}(\Kcs,X)\ar[r]\ar[d,"f'_\ast"]&\map_\ccat(g'(\Kcs),g'(X))\ar[d,"f_\ast"]\\
\map_{\dcat'}(f'(\Kcs),f'(X))\ar[r]&\map_\dcat(fg'(\Kcs),fg'(X))
\end{tikzcd}.\end{center}
As $f$ is full and $X$ is universally Kan, $f_\ast$ is an effective epimorphism of simplicial spaces which admits a Malcev operation, and is therefore a Kan fibration by \cref{thm:surjherd}. By \cref{lem:kanpb}, it follows that $f'_\ast$ is a Kan fibration. As $f'(X)$ is universally Kan, $\map_{\dcat'}(f'(\Kcs),f'(X))$ satisfies the Kan condition, and therefore the same is true of $\map_{\ccat'}(\Kcs,X)$ as needed.

(11)~~As
\[
\map_{\prod_{i\in I}\ccat_i}(\{Y_i : i\in I\},\{X_i : i\in I\})\simeq\prod_{i\in I}\map_{\ccat_i}(Y_i,X_i),
\]
this holds as simplicial spaces which satisfy the Kan condition are closed under products.
\end{proof}

We also record the following for later use.

\begin{lemma}\label{lem:productsplitting}
Let $\ccat$ be an $\infty$-category with finite products and $X \in \ccat$ be an $H$-herd object. Suppose given pointed connected spaces $A$ and $B$ for which the cotensor $X^{A\times B}$ exists in $\ccat$. Then there is an equivalence
\[
X^{A\times B}\simeq X^{A\vee (A\wedge B)\vee B}.
\]
In particular, if $X \in \ccat$ is universally Kan and $X^{A\times B}$ exists in $\ccat$, then $X^{A\vee (A\wedge B)\vee B}$ exists in $\ccat$.
\end{lemma}
\begin{proof}
By \cref{ex:splitproduct}, there is for every $Y\in \ccat$ a natural equivalence 
\[
\map_\ccat(Y,X^{A\times B})\simeq \map_\ccat(Y,X)^{A\times B} \simeq \map_\ccat(Y,X)^{A\vee (A\wedge B)\vee B},
\]
and therefore $X^{A\times B}\simeq X^{A\vee (A\wedge B)\vee B}$ as claimed.
\end{proof}

\section{Infinitary algebraic theories}\label{sec:generaltheories}

We now turn our attention to the study of infinitary algebraic theories and their models. This section establishes those basic results that can be proved for general (i.e.\ non-Malcev) theories. We give the definitions and establish some basic properties and constructions in \S\ref{ssec:nonboundedtheories}--\ref{ssec:imagefactorizations}.

The naive definition of an infinitary theory as a (locally small) $\infty$-category with coproducts is not suitable for some purposes. We work in this generality when it is most natural, calling such $\infty$-categories \emph{pretheories}, but otherwise require that our theories are generated by a small set of objects. This gives access to a reasonable theory of free resolutions, which we establish in \S\ref{ssec:sorts}.

In \S\ref{ssec:bounded}, we relate these general infinitary theories to the notion of a $\kappa$-ary theory for a regular cardinal $\kappa$. As we explain, every $\kappa$-ary theory can be completed to a general theory with the same models, and if $\pcat$ is a theory then $\Model_\pcat$ is presentable if and only if it arises this way for some $\kappa$. In particular, the theories and models we consider strictly generalize the more familiar finitary theories and their nonabelian derived categories (obtained by taking $\kappa = \omega$).

\subsection{Theories and their models}\label{ssec:nonboundedtheories}

\begin{definition}
\label{definition:pretheory_and_pretheory_homomorphism}
We say that
\begin{enumerate}
\item A \emph{pretheory} is a (locally small) $\infty$-category $\pcat$ which admits small coproducts,
\item A \emph{homomorphism} $f\colon \pcat \to \qcat$ of pretheories is a functor which preserves small coproducts.
\end{enumerate}
\end{definition}

The main reason we use the adjective \emph{pre} in the above definition is technical, as in practice we want to work with pretheories satisfying the following natural smallness condition: 

\begin{definition}\label{def:theory}
A \emph{theory} is a pretheory $\pcat$ satisfying the following additional condition:
\begin{enumerate}
\item[{($\ast$)}] There is a small set $\{P_s : s\in S\}$ of objects generating $\pcat$ under coproducts and retracts.
\end{enumerate}
In this case, we say that $\{P_s : s\in S\}$ is a \emph{set of generators} (or \emph{set of sorts}) for $\pcat$ and call the full subcategory $\pcat_0\subset \pcat$ it generates a \emph{generating subcategory} of $\pcat$. 
\end{definition}

\begin{remark}
Let $f\colon \pcat\to\qcat$ be an essentially surjective homomorphism of pretheories. If $\pcat$ is a theory, then $\qcat$ is a theory. Indeed, the image of any set of generators of $\pcat$ forms a set of generators of $\qcat$. 
\end{remark}

Pretheories and homomorphisms assemble into a (very large) $\infty$-category $\pretheories$, which is a non-full subcategory of $\largecatinfty$. We note the following.

\begin{prop}[{\cite[{Corollary 5.3.6.10}]{lurie_higher_topos_theory}}]
The $\infty$-category $\pretheories$ has all small limits, and these are preserved by the forgetful functor $\pretheories \to \largecatinfty$ to locally small $\infty$-categories.
\qed
\end{prop}

Fix for the rest of this subsection a pretheory $\pcat$. There are (at least) two good definitions for the $\infty$-category of models of $\pcat$. We will show that these definitions coincide when $\pcat$ is a theory in \S\ref{ssec:sorts} below.

\begin{definition}\label{def:modelsnonbounded}
The $\infty$-category of \emph{nonbounded models} of $\pcat$ is the full subcategory
\[
\largemodels_\pcat\subset\Fun(\pcat^\op,\spaces)
\]
of product-preserving presheaves on $\pcat$, i.e.\ the $\infty$-category of functors $X\colon \pcat^\op\to\spaces$ for which the canonical map
\[
X(\coprod_{i\in I}P_i) \to \prod_{i\in I}X(P_i)
\]
is an equivalence for any set of objects $\{P_i : i \in I\}$ in $\pcat$.
\end{definition}

For a general pretheory $\pcat$, the $\infty$-category $\largemodels_\pcat$ could fail to be locally small. For this and other reasons we also consider the following subcategory:

\begin{definition}\label{def:modelsbounded}
The $\infty$-category of \emph{models} of $\pcat$ is the full subcategory
\[
\Model_\pcat = \largemodels_\pcat \times_{\Fun(\pcat^\op,\spaces)} \smallpresheaves(\pcat)\subset\largemodels_\pcat
\]
spanned by those nonbounded models whose underlying presheaf is small, i.e.\ may be written as a small colimit of representable presheaves. 
\end{definition}

In either case, we write
\[
\nu\colon \pcat\to\Model_\pcat \subset \largemodels_\pcat
\]
for the restricted Yoneda embedding.

\begin{prop}
\label{prop:limitsinlargemodels}
The full subcategory $\largemodels_\pcat\subset\Fun(\pcat^\op,\spaces)$ is closed under:
\begin{enumerate}
\item Small limits;
\item Geometric realizations of simplicial objects $\Xss$ which satisfy the Kan condition levelwise, i.e.\ for which the simplicial space $\Xss(P)$ satisfies the Kan condition for all $P \in \pcat$;
\item Geometric realizations of simplicial objects $\Xss$ which are levelwise split in the sense of \cite[Definition 4.7.2.2]{higher_algebra}.
\end{enumerate}
\end{prop}
\begin{proof}
(1)~~As limits commute with limits, the pointwise limit of any diagram of nonbounded models of $\pcat$ is again a nonbounded model.

(2)~~Fix a simplicial object $\Xss \in \Fun(\Delta^\op,\largemodels_\pcat)$ for which $\Xss(P) \in \simplicialspaces$ satisfies the Kan condition for all $P \in \pcat$. Let $|\Xss|$ denote the geometric realization of $\Xss$ taken levelwise in $\Fun(\pcat^\op,\spaces)$. We must show that $|\Xss|$ is again a model. In other words, we must show that if $\{P_i : i \in I\}$ is a set of objects in $\pcat$, then the natural map
\[
|\Xss|(\coprod_{i\in I} P_i) \to \prod_{i\in I}|\Xss|(P_i)
\]
is an equivalence. As $\Xss$ is a simplicial model, this map factors as
\[
|\Xss|(\coprod_{i\in I}P_i) = |\Xss(\coprod_{i\in I} P_i)| \simeq |\prod_{i\in I}\Xss(P_i)| \to \prod_{i\in I}|\Xss(P_i)| = \prod_{i\in I}|\Xss|(P_i),
\]
where the first and last equalities just expand the definition of $|\Xss|$. We therefore reduce to verifying that the map
\[
|\prod_{i\in I}\Xss(P_i)| \to \prod_{i\in I}|\Xss(P_i)|
\]
is an equivalence. By \cref{prop:kanlimits}, this holds under the assumption that $\Xss(P_i)$ satisfies the Kan condition for all $i\in I$.

(3)~~Arguing as in (2), it suffices to show instead that if $\{\Xss(i) : i \in I\}$ is a set of split simplicial spaces, then the map
\[
|\prod_{i\in I}\Xss(i)| \to \prod_{i\in I}|\Xss(i)|
\]
is an equivalence. Indeed, if we choose lifts of each $\Xss(i) \in \Fun(\Delta^\op,\spaces)$ to an object $\widetilde{\Xss}(i) \in \Fun(\Delta^\op_{-\infty},\spaces)$, then $\prod_{i\in I}\Xss(i)$ lifts to $\prod_{i\in I}\widetilde{\Xss}(i)$, implying that
\[
|\prod_{i\in I}\Xss(i)| \simeq \prod_{i\in I}\widetilde{X}_{-1}(i) \simeq \prod_{i\in I} |\Xss(i)|
\]
as claimed.
\end{proof}

\begin{cor}
A simplicial object $\Xss$ in $\largemodels_\pcat$ satisfies the Kan condition if and only if it satisfies the Kan condition levelwise.
\end{cor}

\begin{proof}
As we observed in \cref{example:cech_nerve_is_a_kan_complex}, \v{C}ech nerves satisfy the Kan condition, and so \cref{prop:limitsinlargemodels} implies that $\largemodels_\pcat\subset\Fun(\pcat^\op,\spaces)$ is closed under geometric realizations of \v{C}ech nerves. In particular, a morphism $X \to Y$ in $\largemodels_\pcat$ is an effective epimorphism if and only if $X(P) \to Y(P)$ is an effective epimorphism for all $P \in \pcat$. The claim follows. 
\end{proof}

\begin{cor}
\label{cor:nearmonadicity}
Let $f\colon \pcat\to\qcat$ be a homomorphism of pretheories. Then restriction along $f$ defines a functor
\[
f^\ast\colon \largemodels_\qcat\to\largemodels_\pcat
\]
which preserves all small limits and all geometric realizations of simplicial objects which either satisfy the Kan condition or are levelwise split. If $f$ is essentially surjective, then $f^\ast$ is conservative.
\qed
\end{cor}

\subsection{Constructions of theories} \label{ssec:constructionsoftheories}

We briefly describe two pleasant closure properties of pretheories.

\begin{prop}
\label{prop:slicecat}
Let $X \in \largemodels_\pcat$ be a nonbounded model of $\pcat$. Then the $\infty$-category
\[
\pcat_{/X} = \pcat\times_{\Fun(\pcat^\op,\spaces)}\Fun(\pcat^\op,\spaces)_{/X}
\]
of elements of $X$ is again a pretheory, and there is a natural equivalence
\[
\largemodels_{\pcat_{/X}} \simeq (\largemodels_\pcat)_{/X}.
\]
If $X \in \Model_\pcat$ is a model, then this restricts to an equivalence
\[
\Model_{\pcat_{/X}}\simeq (\Model_\pcat)_{/X}.
\]
\end{prop}

\begin{proof}
The slice category $\pcat_{/X}$ admits all small coproducts, created by the forgetful functor $\pcat_{/X} \to \pcat$. By unstraightening, there is an equivalence
\[
\Fun(\pcat,\spaces)_{/X}\simeq\Fun(\pcat_{/X},\spaces)
\]
sending a map $Y \to X$ to the presheaf $\widetilde{Y}$ on $\pcat_{/X}$ given by
\[
\widetilde{Y}(\alpha\colon P \to X) = Y(P)\times_{X(P)}\{\alpha\}.
\]
If $X$ is small, then this restricts to an equivalence between categories of small presheaves. Under the assumption that $X$ preserves products, we see that $Y\colon \pcat^\op\to\spaces$ preserves products if and only if $\widetilde{Y}\colon (\pcat_{/X})^\op\to\spaces$ preserves products, and so this further restricts to the claimed equivalences $\largemodels_{\pcat_{/X}} \simeq (\largemodels_\pcat)_{/X}$ and, when $X$ is small, $\Model_{\pcat_{/X}}\simeq (\Model_{\pcat})_{/X}$.
\end{proof}

Given an $\infty$-category $\ccat$, write $\ccat^\sharp$ for the idempotent completion of $\ccat$.

\begin{prop}\label{prop:idempotentcompletion}
Let $\pcat$ be a pretheory. Then the idempotent completion $\pcat^\sharp$ is again a pretheory, and restriction along the inclusion $\pcat\subset\pcat^\sharp$ induces equivalences
\[
\largemodels_\pcat\simeq\largemodels_{\pcat^\sharp},\qquad \Model_\pcat \simeq \Model_{\pcat^\sharp}.
\]
\end{prop}
\begin{proof}
That $\pcat^\sharp$ is again a pretheory is clear. By \cite[Proposition 5.1.4.9]{lurie_higher_topos_theory}, restriction along the inclusion $\pcat\subset\pcat^\sharp$ defines an equivalence $\Fun((\pcat^\sharp)^\op,\spaces)\to\Fun(\pcat^\op,\spaces)$, and this is easily seen to restrict to an equivalence between full subcategories of nonbounded models. As the walking idempotent is small, the Yoneda embedding $\nu\colon \pcat^\sharp\to \Fun((\pcat^\sharp)^\op,\spaces)\simeq\Fun(\pcat^\op,\spaces)$ lands in the full subcategory $\smallpresheaves(\pcat)\subset\Fun(\pcat^\op,\spaces)$ of small presheaves, so this further restricts to an equivalence between full subcategories of models.
\end{proof}

\subsection{Higher image factorizations}\label{ssec:imagefactorizations}

A basic operation in ordinary universal algebra is the formation of \emph{image factorizations}: any homomorphism $f\colon X \to Y$ between algebraic objects canonically factors as an effective epimorphism $X \to \im f$ followed by monomorphism $\im f \to Y$. This section is devoted to the analogues of this construction in higher universal algebra. 

\begin{definition}\label{def:truncated}
Let $\ccat$ be an $\infty$-category with finite limits. A morphism $X \to Y$ in $\ccat$ is said to be
\begin{enumerate}
\item \emph{$(-2)$-truncated} if it is an equivalence;
\item \emph{$(n+1)$-truncated} for some $n \geq -2$ if the diagonal $X \to X\times_Y X$ is $n$-truncated.
\end{enumerate}
An object $X \in \ccat$ is \emph{$n$-truncated} if $X \to \ast$ is $n$-truncated.
\end{definition}

\begin{definition}\label{def:connected}
Let $\ccat$ be an $\infty$-category with finite limits. A morphism $X \to Y$ is said to be \begin{enumerate}
\item \emph{$(-1)$-connective} if it is any morphism.
\item \emph{$(n+1)$-connective} for some $n\geq -1$ if it is an effective epimorphism (\cref{def:effepi}) and the diagonal $X \to X\times_Y X$ is $n$-connective.
\end{enumerate}
\end{definition}

For the $\infty$-category of spaces, $n$-truncated and $n$-connective morphisms are exactly the $n$-truncated and $n$-connected maps of classical homotopy theory, see \cite[Section 6.5]{lurie_higher_topos_theory}. As a consequence, we have the following.

\begin{prop}\label{prop:levelwiseimage}
A morphism $X \to Y$ in $\largemodels_\pcat$ is $n$-connective, resp., $n$-truncated, if and only if $X(P) \to Y(P)$ is an $n$-connective, resp., $n$-truncated map of spaces for all $P \in \pcat$.
\end{prop}
\begin{proof}
As the \v{C}ech nerve of a map of spaces is a simplicial space satisfying the Kan condition, \cref{prop:limitsinlargemodels} implies that the conditions of \cref{def:truncated} and \cref{def:connected} may be checked levelwise.
\end{proof}

\begin{cor}
Let $f\colon \pcat\to\qcat$ be a homomorphism of pretheories. Then
\[
f^\ast\colon\largemodels_\qcat\to\largemodels_\pcat
\]
preserves $n$-connective and $n$-truncated morphisms. If $f$ is essentially surjective then $f^\ast$ reflects $n$-connective and $n$-truncated morphisms.
\qed
\end{cor}

\cref{prop:levelwiseimage} leads to the desired higher image factorizations.

\begin{prop}\label{prop:imagefactorizations}
The $(n+1)$-connective and $n$-truncated morphisms form a factorization system on $\largemodels_\pcat$ (see \cite[Definition 5.2.8.8]{lurie_higher_topos_theory}), with factorizations constructed levelwise. 
\end{prop}
\begin{proof}
It is essentially classical that $(n+1)$-connective and $n$-truncated morphisms form a factorization system on $\spaces$, see for example \cite[Example 5.2.8.16]{lurie_higher_topos_theory}, and this extends to show that they form a factorization system on any presheaf category with factorizations constructed levelwise. As fibrewise truncations preserve infinite products, this levelwise factorization system on $\Fun(\pcat^\op,\spaces)$ preserves the full subcategory $\largemodels_\pcat\subset\Fun(\pcat^\op,\spaces)$.
\end{proof}

\begin{example}
Taking $n = -1$, every morphism $f\colon X \to Y$ in $\largemodels_\pcat$ factors uniquely as an effective epimorphism $X \to \im f$ followed by a monomorphism $\im f \to Y$. This factorization may be constructed explicitly:
\[
\im f = \colim\left(\begin{tikzcd}
Y&\ar[l]X &\ar[l,shift right]\ar[l,shift left]X\times_Y X&\ar[l,shift right]\ar[l,shift left]\ar[l]X\times_Y X \times_Y X&\ar[l,shift right=0.5mm]\ar[l,shift right=1.5mm]\ar[l,shift left=0.5mm]\ar[l,shift left=1.5mm]\cdots
\end{tikzcd}\right)
\]
is nothing more than the geometric realization of the \v{C}ech nerve of $f$.
\end{example}

\subsection{Existence of free resolutions}\label{ssec:sorts}

The main feature distinguishing theories from pretheories is the existence of \emph{free resolutions}.

\begin{lemma}\label{lem:resolutions}
Let $\pcat$ be a pretheory. The following are equivalent:
\begin{enumerate}
\item The inclusion $\Model_\pcat\subset\largemodels_\pcat$ is an equivalence;
\item For every $X \in \largemodels_\pcat$, there exists an effective epimorphism $\nu P \to X$ for some $P \in \pcat$;
\item Every $X \in \largemodels_\pcat$ admits a free simplicial resolution: there exists an equivalence
\[
X \simeq | \nu \Pss|
\]
for some simplicial object $\Pss \in \Fun(\Delta^\op,\pcat)$ for which $\nu \Pss$ satisfies the Kan condition.
\end{enumerate}
\end{lemma}
\begin{proof}
(1)$\Rightarrow$(2): We must show that if $X \in \Model_\pcat$, then there exists some $P \in \pcat$ and an effective epimorphism $\nu P \to X$. As the underlying presheaf of $X$ is small, there exists a small set $\{P_s : s\in S\}$ of objects of $\pcat$ and an effective epimorphism $\coprod_{s\in S} \nu P_s \to X$ of presheaves. As $X$ is a model, this map factors as
\[
\coprod_{s\in S} \nu P_s \to \nu \coprod_{s\in S} P_s \to X.
\]
Therefore $\nu \coprod_{s\in S} P_s \to X$ is an effective epimorphism, so we may take $P = \coprod_{s\in S} P_s$.

(2)$\Rightarrow$(3): For finitary theories, this appears in \cite[Proof of Lemma 5.5.8.13]{lurie_higher_topos_theory}, and the same proof applies here.

(3)$\Rightarrow$(1): This is clear as a geometric realization of representables is in particular a small colimit of representables.
\end{proof}

\begin{theorem}
\label{thm:splithypercovering}
Let $\pcat$ be a theory. Then the three equivalent properties of \cref{lem:resolutions} are satisfied.
\end{theorem}

\begin{proof}
We verify (2). Let $\pcat_0\subset\pcat$ be a small generating subcategory. As effective epimorphisms of spaces are closed under products and retracts, a map $f\colon X \to Y$ in $\largemodels_\pcat$ is an effective epimorphism if and only if $f(P) \colon X(P)\to Y(P)$ is an effective epimorphism for all $P \in \pcat_0$. As $\pcat_0$ is essentially small and $X$ is valued in small spaces, the slice category $(\pcat_0)_{/X}$ is essentially small. By choosing representatives for each path component of $((\pcat_0)_{/X})^\core$, we obtain a small family of objects $\{P_i \in \pcat_0 : i \in I\}$ and a map
\[
\coprod_{i\in I}\nu P_i \to X
\]
which restricts to an effective epimorphism in $\presheaves(\pcat_0)$. This now factors through a map $\nu \coprod_{i\in I}P_i \to X$ of models which is an effective epimorphism on restriction to $\pcat_0$, and is therefore an effective epimorphism.
\end{proof}

\begin{cor}\label{cor:modelscomplete}
For any theory $\pcat$, the $\infty$-category $\Model_\pcat$ admits all small limits.
\end{cor}
\begin{proof}
As $\Model_\pcat\simeq\largemodels_\pcat$, this follows from \cref{prop:limitsinlargemodels}.
\end{proof}

\begin{cor}\label{cor:nearmonadic2}
Let $f\colon \pcat\to\qcat$ be a homomorphism between theories. Then restriction along $f$ defines a functor
\[
f^\ast\colon \Model_\qcat\to\Model_\pcat
\]
which creates all small limits and all geometric realizations of simplicial objects which either satisfy the Kan condition or are levelwise split. If $f$ is essentially surjective, then $f^\ast$ is conservative.
\end{cor}
\begin{proof}
As $\Model_\pcat\simeq\largemodels_\pcat$, this follows from  \cref{cor:nearmonadicity}.
\end{proof}

\subsection{Bounded theories and presentability}\label{ssec:bounded}

The theories that arise in nature are often generated by a smaller amount of data.

\begin{definition}
Let $\kappa$ be a regular cardinal. A \emph{$\kappa$-ary theory} is a small $\infty$-category $\pcat_\kappa$ which admits $\kappa$-small coproducts. The $\infty$-category of \emph{models} of $\pcat_\kappa$ is the full subcategory
\[
\presheaves_\Sigma^\kappa(\pcat_\kappa)\subset \presheaves(\pcat_\kappa)
\]
of presheaves $X \colon \pcat_\kappa^\op\to\spaces$ that preserve $\kappa$-small products.
\end{definition}

The observations of the previous sections admit variants for $\kappa$-ary theories that we shall not spell out.

\begin{example}
The $\omega$-ary theories are the \emph{finitary theories}. For a finitary theory $\ccat$, we abbreviate
\[
\presheaves_\Sigma(\ccat) = \presheaves_\Sigma^\omega(\ccat).
\]
This $\infty$-category is often called the \emph{nonabelian derived category} or \emph{animation} of $\ccat$, and is studied in \cite[Section 5.5.8]{lurie_higher_topos_theory}.
\end{example}

\begin{warning}
Let $\kappa$ be a regular cardinal and $\pcat_\kappa$ be a $\kappa$-ary theory. Then $\pcat_\kappa$ is also a $\lambda$-ary theory for every regular cardinal $\lambda \leq \kappa$, but the $\infty$-categories $\presheaves_\Sigma^\kappa(\pcat_\kappa)$ and $\presheaves_\Sigma^\lambda(\pcat_\kappa)$ are almost never equivalent for $\kappa\neq\lambda$. 
\end{warning}

\begin{prop}\label{prop:kappaaccessible}
Let $\pcat_\kappa$ be a $\kappa$-ary theory. Then $\presheaves_\Sigma^\kappa(\pcat_\kappa)$ is a $\kappa$-accessible localization of $\presheaves(\pcat_\kappa)$. In particular, $\presheaves_\Sigma^\kappa(\pcat_\kappa)$ is presentable.
\end{prop}
\begin{proof}
As $\kappa$-filtered colimits preserve $\kappa$-small products, we see that the inclusion $\presheaves_\Sigma^\kappa(\pcat_\kappa) \to \presheaves(\pcat_\kappa)$ preserves $\kappa$-filtered colimits. By construction, this inclusion realizes $\presheaves_\Sigma^\kappa(\pcat_\kappa)$ as the full subcategory of objects which are local with respect to $\coprod_{i\in I}\nu P_i \to \nu \coprod_{i\in I}P_i$ for any $\kappa$-small collection $\{P_i : i \in I\}$ of object of $\pcat_\kappa$. As $\pcat_\kappa$ is small, this constitutes a small set of morphisms, and therefore $\presheaves_\Sigma^\kappa(\pcat_\kappa)$ is a localization of $\presheaves(\pcat_\kappa)$ \cite[Proposition 5.5.4.15]{lurie_higher_topos_theory}.
\end{proof}

It is inconvenient to have a different notion of $\kappa$-ary theory for every regular cardinal $\kappa$. We now explain how $\kappa$-ary theories for varying $\kappa$ may be uniformly embedded into the framework of general infinitary theories.

\begin{definition}\label{def:boundedtheory}
Let $\kappa$ be a regular cardinal. We say  that a theory $\pcat$ is \emph{$\kappa$-bounded} if there exists an essentially small full subcategory $\pcat_\kappa\subset\pcat$ satisfying the following conditions:
\begin{enumerate}
\item Every object of $\pcat$ is a retract of a small coproduct of objects of $\pcat_\kappa$;
\item $\pcat_\kappa$ is closed under $\kappa$-small coproducts in $\pcat$;
\item For every $P\in \pcat_\kappa$ and every set of objects $\{P_i : i \in I\}$ in $\pcat$, the canonical map $\colim_{F\subset I,\, |F|<\kappa}\map_\pcat(P,\coprod_{i\in F}P_i)\to \map_{\pcat}(P,\coprod_{i\in I}P_i)$ is an equivalence.
\end{enumerate}
In this case, $\pcat_\kappa$ is a $\kappa$-ary theory, and we say that $\pcat$ is \emph{generated by the $\kappa$-ary theory} $\pcat_\kappa$. 
\end{definition}

\begin{definition}
A \emph{bounded theory} is a theory which is $\kappa$-bounded for some $\kappa$.
\end{definition}

This definition goes back in some form at least to Wraith \cite[Page 13]{wraith1969algebraic}.

\begin{warning}
Let $\pcat$ be a theory and $\pcat_0\subset\pcat$ be a generating subcategory in the sense of \cref{def:theory} which is closed under $\kappa$-small coproducts. Then $\pcat_0$ is a $\kappa$-ary theory, but $\pcat$ need not be generated by the $\kappa$-ary theory $\pcat_0$ in the sense of \cref{def:boundedtheory} as condition (3) may not be satisfied.
\end{warning}

\begin{example}
Not every interesting theory is bounded. For example, let $\pcat$ be the full subcategory of the $1$-category of compact Hausdorff topological spaces spanned by the Stone--\v{C}ech compactifications of sets. Then it is classical that the category of set-valued models of $\pcat$ is the category of compact Hausdorff topological spaces, but the theory $\pcat$ is not bounded.
\end{example}

\begin{example}\label{ex:kappabounded}
Let $\pcat_\kappa$ be a $\kappa$-ary theory. By \cref{prop:kappaaccessible}, if $P \in \pcat$ then $\nu P \in \presheaves_\Sigma^\kappa(\pcat_\kappa)$ is $\kappa$-compact. It follows that the full subcategory of $\presheaves_\Sigma^\kappa(\pcat_\kappa)$ generated under coproducts by $\nu P$ for $P \in \pcat_\kappa$ forms a $\kappa$-bounded theory generated by $\pcat_\kappa$.
\end{example}

The main observation here is that \cref{ex:kappabounded} accounts for all bounded theories.

\begin{theorem}\label{thm:bounded}
Suppose that $\pcat$ is generated by the $\kappa$-ary theory $\pcat_\kappa$. Then restriction along $\pcat_\kappa\subset\pcat$ defines an equivalence
\[
\Model_\pcat\simeq\presheaves_\Sigma^\kappa(\pcat_\kappa).
\]
\end{theorem}
\begin{proof}
By \cref{prop:idempotentcompletion}, we are free to replace $\pcat$ with the full subcategory generated by $\pcat_\kappa$ under coproducts, and so assume that every object of $\pcat$ is a coproduct of objects in $\pcat_\kappa$. Clearly restriction along $\pcat_\kappa\subset\pcat$ defines a functor
\[
R\colon \Model_\pcat\to\presheaves_\Sigma^\kappa(\pcat_\kappa).
\]
First we claim that the composite 
\[
R\nu\colon \pcat\to\Model_\pcat \to \presheaves_\Sigma^\kappa(\pcat_\kappa)
\]
preserves coproducts. As $\pcat$ is generated by $\pcat_\kappa$, to show that $R\nu$ preserves coproducts it suffices to show that $R\nu$ preserves coproducts of sets of objects in $\pcat_\kappa\subset\pcat$. Let $\{P_i : i \in I\}$ be a set of objects in $\pcat_\kappa$. The poset of $\kappa$-small subsets of $I$ is $\kappa$-filtered, so as $\presheaves_\Sigma^\kappa(\pcat_\kappa)\subset\presheaves(\pcat_\kappa)$ is closed under $\kappa$-filtered colimits by \cref{prop:kappaaccessible}, we deduce that
\[
R\nu \coprod_{i\in I}P_i\simeq R\nu \colim_{F\subset I,\, |F|<\kappa} \coprod_{i\in F}P_i\simeq \colim_{F\subset I,\, |F|<\kappa}R\nu \coprod_{i\in F}P_i\simeq \colim_{F\subset I,\,|F|<\kappa}\coprod_{i\in F}R\nu P_i \simeq \coprod_{i\in I}R\nu P_i.
\]
Here, the second equivalence uses part $(3)$ of \cref{def:boundedtheory}.

Next we claim that $R\nu$ is fully faithful. As $R\nu$ preserves coproducts and $\pcat_\kappa$ generates $\pcat$, it suffices to show that if $Q \in \pcat_\kappa\subset\pcat$ and $P \in \pcat$ then
\[
\map_\pcat(Q,P)\simeq\map_{\pcat_\kappa}(R\nu Q,R\nu P).
\]
This holds tautologically by the definition of $R$.

As the composite $\pcat\to\Model_\pcat\to\presheaves_\Sigma^\kappa(\pcat_\kappa)$ preserves coproducts, there is a restricted Yoneda embedding
\[
G\colon \presheaves_\Sigma^\kappa(\pcat_\kappa)\to\Model_\pcat.
\]
We claim that $R \dashv G$. It suffices to prove that if $X\in \Model_\pcat$ then there exists a natural equivalence
\[
\map_\pcat(X,G(\bs))\simeq \map_{\pcat_\kappa}(RX,\bs).
\]
If $X = \nu P$ for some $P \in \pcat$, then these two functors agree by definition of $G$. In general, by \cref{thm:splithypercovering}, we may choose a resolution $X \simeq |\nu \Pss|$ which satisfies the levelwise Kan condition with $P_n \in \pcat$. It is easily seen that $R$ preserves geometric realizations of simplicial objects which satisfy the Kan condition levelwise, as these are computed levelwise in either $\Model_\pcat$ or $\presheaves_\sigma^\kappa(\pcat_\kappa)$, and therefore
\[
\map_\pcat(X,G(\bs))\simeq \Tot \map_\pcat(\nu \Pss,G(\bs))\simeq \Tot \map_{\pcat_\kappa}(R \nu \Pss,\bs)\simeq \map_{\pcat_\kappa}(RX,\bs)
\]
as needed.

If $Y \in \presheaves_\Sigma^\kappa(\pcat_\kappa)$, then the counit $RG Y \to Y$ is tautologically an equivalence. Therefore $R \dashv G$ realizes realizes $\presheaves_\Sigma^\kappa(\pcat_\kappa)$ as a colocalization of $\Model_\pcat$, and to be an adjoint equivalence it is sufficient that $R$ be conservative. This is clear given the assumption that $\pcat$ is generated by $\pcat_\kappa$.
\end{proof}

\begin{cor}\label{cor:presentable}
A theory $\pcat$ is bounded if and only if $\Model_\pcat$ is presentable.
\end{cor}
\begin{proof}
\cref{thm:bounded} and \cref{prop:kappaaccessible} together imply that if $\pcat$ is bounded then $\Model_\pcat$ is presentable. Conversely, suppose given a theory $\pcat$ for which $\Model_\pcat$ is presentable. Given a set $\{P_s : s\in S\}$ of generators for $\pcat$, there exists a regular cardinal $\kappa$ for which $\nu P_s \in \Model_\pcat$ is $\kappa$-compact for all $s\in S$. If $\pcat_\kappa\subset\pcat$ is the full subcategory generated by $\{P_s : s\in S\}$ under $\kappa$-small products, then it follows that $\nu P \in \Model_\pcat$ is $\kappa$-compact for all $P \in \pcat_\kappa$. Therefore $\pcat$ is generated by the $\kappa$-ary theory $\pcat_\kappa$ in the sense of \cref{def:boundedtheory}, implying that $\pcat$ is bounded. 
\end{proof}

\begin{warning}
Although the condition that $\pcat$ is bounded guarantees that $\Model_\pcat$ is presentable, it is \emph{not} the case that $\pcat$ itself must be an accessible category. However, accessibility of $\pcat$ is neither a reasonable nor useful condition to impose. For example, if $\pcat$ is the ordinary category of free abelian groups, then accessibility of $\pcat$ is independent of the usual axioms of set theory \cite[\S 5.5]{makkaipare1989accessible}.
\end{warning}

\section{Malcev theories}\label{sec:malcev}

We now commence our study of Malcev theories, building on the framework developed in the preceding sections. We give the definition and establish basic properties in \S\ref{ssec:malcevtheories}, then in \S\ref{ssec:cocompletion} establish the fundamental result: the $\infty$-category of models for a Malcev theory is its free cocompletion under geometric realizations. In \S\ref{ssec:quillen}, we then recall the connection between $\infty$-categories of models for discrete theories and classical homotopy theories of simplicial objects.

\subsection{The Malcev condition}\label{ssec:malcevtheories}

\begin{definition}
\label{def:malcevtheory}
An $\infty$-category $\pcat$ with finite coproducts is \emph{Malcev} if every object of $\pcat$ is couniversally Kan. That is, when every object is universally Kan when considered as an object of the opposite category $\pcat^\op$.
\end{definition}

From \cref{def:malcevtheory}, one obtains notions of Malcev (pre)theories, $\kappa$-ary Malcev theories, and so on, as those (pre)theories, $\kappa$-ary theories, and so on whose underlying $\infty$-category is Malcev.

Malcev pretheories and homomorphisms form a full subcategory $\malcevtheories\subset\pretheories$ of the $\infty$-category of all pretheories. We make some basic observations about this subcategory.

\begin{prop}\label{prop:malcevcharacterizations}
Let $\pcat$ be a pretheory. Then the following are equivalent:
\begin{enumerate}
\item $\pcat$ is Malcev.
\item If $\Xss$ is a simplicial object in $\pcat$ and $P \in \pcat$, then the simplicial space $\map_\pcat(P,\Xss)$ satisfies the Kan condition.
\item Every simplicial model of $\pcat$ satisfies the Kan condition.
\item Every simplicial nonbounded model of $\pcat$ satisfies the levelwise Kan condition.
\item Every $P \in \pcat$ admits a co-Malcev operation. In other words, for every $P\in \pcat$, there exists some map $t\colon P \to P \coprod P \coprod P$ for which the diagram
\begin{center}\begin{tikzcd}[column sep=large]
&P\ar[dr,"i_1"]\ar[dl,"i_2"']\ar[d,"t"]\\
P\coprod P&P\coprod P \coprod P\ar[l,"\nabla\coprod P"]\ar[r,"P\coprod\nabla"']&P\coprod P
\end{tikzcd}\end{center}
commutes up to homotopy.
\end{enumerate}
Moreover, if $\pcat$ is generated under coproducts and retracts by a subclass of objects $\pcat_0\subset\pcat$, then in (2--5) one need only check the given condition on objects $P \in \pcat_0$.
\end{prop}
\begin{proof}
These follow from \cref{prop:univkancharacterize}.
\end{proof}

\begin{example}
Let $\pcat$ be a pretheory. If $\pcat$ is additive, then it is Malcev.
\end{example}

\begin{example}\label{ex:malcevtheoryofkanobjects}
Let $\ccat$ be an $\infty$-category with small products, and let $\ccat_h\subset\ccat$ be the full subcategory of universally Kan objects. Then $\ccat_h^\op$ is a (possibly trivial) Malcev pretheory.
\end{example}

We record some closure properties of the Malcev condition.

\begin{lemma}\label{lem:malcevclosure}
Let $\pcat$ be a pretheory.
\begin{enumerate}
\item $\pcat$ is Malcev if and only if its homotopy category $\h\pcat$ is Malcev. 
\item If $\pcat$ is a $1$-category, then $\pcat$ is Malcev in the sense of \cref{def:malcevtheory} if and only if it is Malcev in the classical sense (see \S\ref{subsection:malcevkanhistory}).
\item Let $\qcat\to\pcat$ be an essentially surjective homomorphism. If $\qcat$ is Malcev then $\pcat$ is Malcev.
\item If $X \in \largemodels_\pcat$, then the slice category $\pcat_{/X}$ is Malcev.
\end{enumerate}
\end{lemma}
\begin{proof}
(1)~~This follows from \cref{prop:univkanclosure}(8).

(2)~~This follows from \cref{prop:malcevcharacterizations} and the discussion of \S\ref{subsection:malcevkanhistory}.

(3)~~This follows from \cref{prop:malcevcharacterizations}(5).

(4)~~This follows from \cref{prop:univkanclosure}(9).
\end{proof}

In general, there seems no reason to expect that the limit of a small diagram of Malcev pretheories and homomorphisms is again Malcev. Nonetheless we have the following.

\begin{prop}
\label{prop:limitmalcevtheory}
Malcev pretheories have the following closure properties: 
\begin{enumerate}
\item Suppose we are given small family $\{\pcat(i) : i \in I\}$ of pretheories. If each $\pcat(i)$ is Malcev, then the product $\prod_{i\in I}\pcat(i)$ is Malcev.
\item Suppose we are given a cartesian diagram
\begin{center}\begin{tikzcd}
\pcat'\ar[r,"g'"]\ar[d,"f'"]&\pcat\ar[d,"f"]\\
\qcat'\ar[r,"g"]&\qcat
\end{tikzcd}\end{center}
of pretheories and homomorphisms. If $f$ is full and $\pcat$ and $\qcat'$ are Malcev, then $\pcat'$ is Malcev.
\end{enumerate}
\end{prop}
\begin{proof}
These follow from the dual forms of \cref{prop:univkanclosure}(10,11).
\end{proof}

Given a Malcev pretheory $\pcat$, consider the $\infty$-categories
\[
\Model_\pcat\subset\largemodels_\pcat \subset\Fun(\pcat^\op,\spaces)
\]
of models of $\pcat$ introduced in \cref{def:modelsnonbounded} and \cref{def:modelsbounded}. The key difference from the non-Malcev case in studying these categories ultimately boils down to the following.

\begin{prop}\label{prop:pointwisegeoreal}
Let $\pcat$ be a Malcev pretheory. Then the full subcategories 
\[
\Model_\pcat \subset \largemodels_\pcat\subset\Fun(\pcat^\op,\spaces)
\]
are closed under geometric realizations.
\end{prop}
\begin{proof}
As $\Model_\pcat\subset\largemodels_\pcat$ is closed under any small colimits that are preserved by $\largemodels_\pcat\subset\Fun(\pcat^\op,\spaces)$, it suffices to prove that $\largemodels_\pcat\subset\Fun(\pcat^\op,\spaces)$ is closed under geometric realizations. By \cref{prop:malcevcharacterizations}, every simplicial object in $\largemodels_\pcat$ satisfies the levelwise Kan condition, so this follows from \cref{prop:limitsinlargemodels}(2).
\end{proof}

This has several pleasant consequences.

\begin{cor}\label{corollary:for_models_of_malcev_theory_geometric_realization_commutes_with_products_and_pullbacks_along_effective_epi}
Let $\pcat$ be a Malcev pretheory. 
\begin{enumerate}
\item Suppose given a small family $\{\Xss(i) : i \in I\}$ of simplicial objects in $\largemodels_\pcat$. Then
\[
|\prod_{i\in I}\Xss(i)| \to \prod_{i\in I}|\Xss(i)|
\]
is an equivalence.
\item Suppose given a cartesian square
\begin{center}\begin{tikzcd}
\Xss'\ar[r]\ar[d]&\Xss\ar[d,"f"]\\
\Yss'\ar[r]&\Yss
\end{tikzcd}\end{center}
of simplicial objects in $\largemodels_\pcat$. If $f$ is an effective epimorphism, then
\[
|\Xss'| \to |\Yss'|\times_{|\Yss|}|\Xss|
\]
is an equivalence.
\end{enumerate}
\end{cor}
\begin{proof}
As limits and geometric realizations in $\largemodels_\pcat$ are computed levelwise, it suffices to verify these claims after evaluation on any $P\in \pcat$. As $\pcat$ is Malcev, each $\Xss(i)(P)$ satisfies the Kan condition, and the map $\Xss(P) \to \Yss(P)$ is a Kan fibration, so the corollary follows from \cref{prop:kanlimits}.
\end{proof}

\begin{cor}\label{cor:universality}
Let $\pcat$ be a Malcev pretheory. Then geometric realizations in $\largemodels_\pcat$ are strongly universal. In other words, the comparison map
\[
(\largemodels_\pcat)_{/|\Xss|} \to \Tot \left((\largemodels_\pcat)_{/\Xss}\right),
\]
is an equivalence of $\infty$-categories.
\end{cor}
\begin{proof}
By descent, the comparison map
\[
\Fun(\pcat^\op,\spaces)_{/|\Xss|} \to \Tot\left( \Fun(\pcat^\op,\spaces)_{/\Xss}\right)
\]
is an equivalence. This functor sends a morphism $Y \to |\Xss|$ to the object of the totalization represented by the compatible collection $\{Y\times_{|\Xss|} X_n \to X_n\}$, and has inverse sending a compatible collection $\{\Yss \to \Xss\}$ to its geometric realization $|\Yss| \to |\Xss|$. Both these functors preserve full subcategories of nonbounded models, so this restricts to the claimed equivalence.
\end{proof}

\subsection{Models of a Malcev theory as cocompletion}\label{ssec:cocompletion}

The $\infty$-category of models for a Malcev pretheory is categorically quite well behaved. To start, we have the following.

\begin{lemma}
\label{lem:malcevcocomplete}
Let $\pcat$ be a Malcev pretheory.
\begin{enumerate}
\item The inclusion $i\colon \Model_\pcat \to \smallpresheaves(\pcat)$ admits a left adjoint.
\item $\Model_\pcat\subset\smallpresheaves(\pcat)$ is the smallest full subcategory containing the representables and closed under geometric realizations.
\item The $\infty$-category $\Model_\pcat$ admits all small colimits.
\end{enumerate}
\end{lemma}
\begin{proof}
(1)~~For the inclusion $i\colon \Model_\pcat\to\smallpresheaves(\pcat)$ to admit a left adjoint, we must prove that if $X \in \smallpresheaves(\pcat)$ then 
\[
\map_{\smallpresheaves(\pcat)}(X,i(\bs))\colon \Model_\pcat\to\spaces
\]
is corepresentable. If $X \in \presheaves(\pcat)$ is a small presheaf, then by definition $X$ is a small colimit of representables. By the Bousfield--Kan formula for colimits, it follows that $X$ admits a presentation of the form $X\simeq \colim_{n\in\Delta^\op}\coprod_{i\in I_n}\nu(P_{n,i})$ for some sets $I_n$ and objects $P_{n,i} \in \pcat$, and thus
\[
\map_{\smallpresheaves(\pcat)}(X,i(\bs))\simeq \map_{\smallpresheaves(\pcat)}(\colim_{n\in\Delta^\op}\nu(\coprod_{i\in I_n} P_{n,i}),i(\bs)).
\]
As $\colim_{n\in\Delta^\op}\nu(\coprod_{i\in I_n}P_{n,i}) \in \Model_\pcat$ by \cref{prop:pointwisegeoreal}, it follows that $i$ admits a left adjoint as claimed.

(2)~~The proof of (1) gives an explicit way of writing every model of $\pcat$ as a geometric realization of representables. Alternately, at least if $\pcat$ is a theory, (2) follows from \cref{thm:splithypercovering} as the Malcev condition ensures that every simplicial model satisfies the levelwise Kan condition.

(3)~~This follows from (1) as $\smallpresheaves(\pcat)$ admits all small colimits.
\end{proof}

\begin{cor}
\label{cor:malcevcomplete}
Let $\pcat$ be a Malcev theory. Then $\Model_\pcat$ admits all small limits and colimits.
\end{cor}

\begin{proof}
\cref{lem:malcevcocomplete} shows that $\Model_\pcat$ admits all small colimits. The further assumption that $\pcat$ is a theory implies by \cref{cor:modelscomplete} that $\Model_\pcat$ admits all small limits.
\end{proof}

\cref{lem:malcevcocomplete}.(2) can be reformulated as saying that if $\pcat$ is a Malcev pretheory, then $\Model_\pcat$ is obtained by freely adjoining geometric realizations to $\pcat$. This leads to the following universal property of the $\infty$-category $\Model_\pcat$.

\begin{theorem}
\label{thm:freecocompletion}
Let $\pcat$ be a Malcev pretheory and $\dcat$ be an $\infty$-category which admits geometric realizations.
\begin{enumerate}
\item The pointwise left Kan extension of any functor $f\colon \pcat \to \dcat$ along $\nu \colon \pcat\to\Model_\pcat$ exists. In particular, left Kan extension and restriction along $\nu$ defines an adjunction
\[
\Fun(\pcat,\dcat) \rightleftarrows \Fun(\Model_\pcat,\dcat).
\]
\item The above adjunction restricts to an equivalence between arbitrary functors $f\colon \pcat\to\dcat$ and functors $f_!\colon \Model_\pcat\to\dcat$ which preserve geometric realizations.
\item Suppose given a functor $f\colon \pcat\to\dcat$ which is fully faithful and lands in the full subcategory of projective objects of $\dcat$. Then $f_!\colon \Model_\pcat\to\dcat$ is fully faithful.
\item A geometric realization-preserving functor $f_!\colon \Model_\pcat\to\dcat$ preserves all small colimits if and only if $f\colon \pcat\to\dcat$ preserves all small coproducts. In particular, if $\dcat$ admits all small colimits, then the above adjunction restricts to an equivalence between coproduct-preserving functors $f\colon \pcat\to\dcat$ and functors $f_!\colon \Model_\pcat\to\dcat$ which preserve all small colimits.
\end{enumerate}
\end{theorem}
\begin{proof}
(1, 2)~~These follow from \cref{lem:malcevcocomplete}(2) and \cite[Remark 5.3.5.9]{lurie_higher_topos_theory}.

(3)~~We must verify that if $X,Y\in \Model_\pcat$, then the map
\[
\map_\pcat(X,Y) \to \map_\dcat(f_! X,f_! Y)
\]
is an equivalence. As $f_!$ preserves geometric realizations, by resolving $X$ by representables we may as well assume that $X = \nu P$ for some $P \in \pcat$, so that $f_! X \simeq f P$. As $fP$ is projective by assumption, by resolving $Y$ by representables we may as well assume that $Y = \nu P'$ for some $P' \in \pcat$. We are now asking that
\[
\map_\pcat(P,P') \to \map_\dcat(f P,fP')
\]
is an equivalence, which holds as $f$ is fully faithful by assumption.

(4)~~As $\nu\colon \pcat\to\Model_\pcat$ preserves coproducts, if $f\colon \pcat\to\dcat$ is a functor for which $f_!\colon \Model_\pcat\to\dcat$ preserves colimits then $f$ preserves coproducts. Conversely, as all colimits can be built out of geometric realizations and coproducts, it suffices to prove that if $f\colon \pcat\to\dcat$ preserves coproducts then so does $f_!\colon \Model_\pcat\to\dcat$. If $\{X_i : i \in I\}$ is a set of objects of $\pcat$ and we choose resolutions $X_i\simeq |\nu(P_{\bullet,i})|$, then 
\[
\coprod_{i\in I}X_i\simeq \coprod_{i\in I}|\nu(P_{\bullet,i})| \simeq |\coprod_{i\in I}\nu(P_{\bullet,i})| \simeq |\nu(\coprod_{i\in I}P_{\bullet,i})|,
\]
so this follows from the fact, guaranteed by (2), that $f_!$ necessarily preserves geometric realizations.
\end{proof}

We also note the following.

\begin{prop}\label{cor:restriction}
Let $\pcat$ be a Malcev theory and $\dcat$ be an $\infty$-category which admits all small colimits, and let $f\colon \pcat\to\dcat$ be a coproduct-preserving functor.
\begin{enumerate}
\item The induced colimit-preserving functor $f_!\colon \Model_\pcat\to\dcat$ admits a right adjoint given by
\[
f^\ast\colon \dcat\to \Model_\pcat,\qquad (f^\ast D)(P) = \map_\dcat(f(P),D).
\]
\item $f^\ast$ preserves geometric realizations if and only if $f(P) \in \dcat$ is projective for all $P \in \pcat$.
\item $f_! \dashv f^\ast$ is an adjoint equivalence if and only if the following conditions are satisfied:
\begin{enumerate}
\item $f$ is fully faithful;
\item $f(P)\in \dcat$ is projective for all $P \in \pcat$;
\item $f^\ast$ is conservative.
\end{enumerate}
\end{enumerate}
\end{prop}
\begin{proof}
(1)~~The condition that $f$ preserves coproducts ensures that if $D \in \dcat$, then $f^\ast D \colon \pcat^\op\to\spaces$ preserves coproducts. As $\pcat$ is a theory, $f^\ast D$ is therefore a model of $\pcat$. This object satisfies
\[
\map_\pcat(\nu P,f^\ast D)\simeq \map_\dcat(fP,D)\simeq \map_\dcat(f_! \nu P,D)
\]
for $P \in \pcat$ by definition, and as both sides are compatible with geometric realizations in $P$ this implies that
\[
\map_\pcat(X,f^\ast D)\simeq \map_\dcat(f_! X,D)
\]
for all $X \in \Model_\pcat$. Therefore $f_! \dashv f^\ast$ as claimed.

(2)~~Clear.

(3)~~If $f$ is fully faithful and lands in the full subcategory of projective objects of $\dcat$, then $f_!$ is fully faithful by \cref{thm:freecocompletion}.(4), and is therefore an equivalence if and only if its right adjoint $f^\ast$ is conservative as claimed.
\end{proof}

\subsection{Discrete theories and animation}
\label{ssec:quillen}

The Malcev condition is well established in classical universal algebra; see \S\ref{subsection:malcevkanhistory}. For example, any variety of algebras which consists of groups with extra structure arises as the category of set-valued models for a Malcev theory. We introduce some notation to access these categories more easily.

\begin{definition}
Given a theory $\pcat$, we write
\[
\Model_\pcat^\heartsuit \subset\Model_\pcat
\]
for the full subcategory of $0$-truncated objects.
\end{definition}

\begin{remark}
\label{remark:for_a_theory_discrete_models_are_the_same_as_for_its_homotopy_category}
For any theory $\pcat$, restriction along the homomorphism $\tau\colon \pcat\to\h\pcat$ defines an equivalence
\[
\Model_\pcat^{\heartsuit}\simeq\Model_{\h\pcat}^{\heartsuit}.
\]
In particular, we lose nothing in this subsection by restricting to the case where $\pcat$ is a discrete theory, i.e.\ is a $1$-category.
\end{remark}

In \cite[II.4.2]{quillen1967homotopical}, Quillen constructed a closed model structure on the category $\Fun(\Delta^\op,\dcat)$ of simplicial objects in a $1$-category $\dcat$ under certain assumption on $\dcat$, which correspond essentially to asking that $\dcat\simeq\Model_\pcat^\heartsuit$ for $\pcat$ a theory which is either $\omega$-bounded \emph{or} Malcev. The weak equivalences and fibrations in either case are defined levelwise. Under these assumptions, we have the following characterization of the $\infty$-category $\Model_\pcat$.

\begin{theorem}
\label{thm:animation}
Let $\pcat$ be a discrete theory, and suppose that $\pcat$ is either $\omega$-bounded or Malcev. Then geometric realization
\[
|\bs|\colon \Fun(\Delta^\op,\Model_\pcat^\heartsuit)\subset\Fun(\Delta^\op,\Model_\pcat) \to \Model_\pcat
\]
is a localization realizing $\Model_\pcat$ as the underlying $\infty$-category of Quillen's model structure on $\Fun(\Delta^\op,\Model_\pcat^\heartsuit)$.
\end{theorem}
\begin{proof}
Note that, in the $\omega$-bounded case, we may realize the given functor as the levelwise geometric realization functor
\[
|\bs|\colon \Fun(\Delta^\op,\presheaves_\Sigma(\pcat_\omega,\sets))\to\presheaves_\Sigma(\pcat_\omega).
\]
Here the theorem goes back essentially to work of Badzioch \cite{badzioch2002algebraic} and Bergner \cite{bergner_rigidification_of_algebras}, see \cite[Proposition 5.5.9.2]{lurie_higher_topos_theory}. Essentially the same proof applies in the Malcev case, see for example \cite[Proposition 2.2.2]{balderrama2021deformations}.
\end{proof}

Thus, if $\pcat$ is a discrete $\omega$-bounded or Malcev theory, then $\Model_{\pcat}$ is a familiar ``algebraic'' $\infty$-category.

\begin{example}
Suppose that $\pcat$ is a discrete additive Malcev theory. Then $\Model_\pcat^\heartsuit$ is a cocomplete abelian category with enough projectives, and $\Model_\pcat$ is equivalent to its connective derived $\infty$-category.
\end{example}

\begin{example}
Let $k$ be a commutative ring and $\pcat$ be the category of polynomial $k$-algebras (on possibly infinitely many variables). Then $\Model_\pcat$ is equivalent to the localization of the category of simplicial commutative $k$-algebras at the weak equivalences; that is, to the $\infty$-category of animated rings under $k$.
\end{example}

\section{Recognition principles}

So far, we have been interested in the study of $\infty$-categories of the form $\Model_\pcat$ for a theory $\pcat$. In this context, the choice of theory $\pcat$ is additional structure on the $\infty$-category $\Model_\pcat$, analogous to the choice of a presentation of a ring. 

Our goal in this section is to give an intrinsic characterization of those $\infty$-categories of the form $\Model_\pcat$ for various types of $\pcat$, and to explain the sense in which they admit canonical presentations. We deal with general theories in \S\ref{subsection:weakly_projectively_generated_infty_cats}, Malcev theories in \S\ref{ssec:stronglyprojectivelygenerated}, and finally additive theories in \S\ref{subsection:prestable_infty_cats}.

\subsection{Weakly projectively generated \texorpdfstring{$\infty$}{infty}-categories}
\label{subsection:weakly_projectively_generated_infty_cats}

Recall that an object $P$ of an $\infty$-category $\dcat$ is said to be \emph{projective} if $\map_\dcat(P,\bs)$ preserves geometric realizations. To characterize the categories of models for theories we will consider the following natural weakening of this notion.

\begin{definition}
\label{def:weakprojective}
Let $\dcat$ be an $\infty$-category with finite limits which admits geometric realizations of simplicial objects satisfying the Kan condition. We say  that $P \in \dcat$ is \emph{weakly projective} if
\[
\map_\dcat(P,\bs)\colon \dcat\to\spaces
\]
preserves geometric realizations of simplicial objects which satisfy the Kan condition.
\end{definition}

\begin{example}
A space $X\in \spaces$ is weakly projective if and only if it is equivalent to a set, whereas it is projective if and only if it is equivalent to a finite set.
\end{example}

\begin{prop}\label{prop:wbprojective}
Let $\pcat$ be a theory. Then the idempotent completion $\pcat^\sharp$ may be identified as the full subcategory of $\Model_\pcat$ spanned by the weakly projective models.
\end{prop}
\begin{proof}
By \cref{prop:limitsinlargemodels}, if $P \in \pcat$ then $\nu P \in \Model_\pcat$ is weakly projective. As weakly projective objects are closed under retracts, it follows that all objects of $\pcat^\sharp\subset\Model_\pcat$ are weakly projective.

Conversely, suppose that $X \in \Model_\pcat$ is projective. By \cref{thm:splithypercovering}, we may write $X \simeq |\nu \Pss|$ for some $\Pss \in \Fun(\Delta^\op,\pcat)$ for which $\nu \Pss$ satisfies the Kan condition. As $X$ is weakly projective, it follows that
\[
\map_\pcat(X,X)\simeq |\map_\pcat(X,\nu \Pss)|.
\]
Therefore the identity of $X$ factors through $\nu P_n$ for some $n\geq 0$, implying $X$ is a retract of $\nu P_n$ and therefore lives in $\pcat^\sharp\subset\Model_\pcat$.
\end{proof}

\begin{corollary}
Let $\pcat$ and $\qcat$ be theories. Then there exists an equivalence $\Model_\pcat\simeq\Model_\qcat$ if and only if there exists an equivalence $\pcat^\sharp\simeq\qcat^\sharp$; that is, if and only if $\pcat$ and $\qcat$ are equivalent up to idempotent completion. 
\qed
\end{corollary}

We also have the following bounded variant.

\begin{prop}\label{prop:compactweaklyprojective}
Suppose that $\pcat$ is generated by a $\kappa$-ary theory $\pcat_\kappa\subset\pcat$. Then the idempotent completion $\pcat_\kappa^\sharp$ may be identified as the full subcategory of $\kappa$-compact weakly projective objects of $\Model_\pcat$.
\end{prop}
\begin{proof}
By the same reasoning as in \cref{prop:wbprojective}, every object of $\pcat_\kappa^\sharp$ is $\kappa$-compact and weakly projective in $\Model_\pcat$. Conversely, suppose that $X \in \Model_\pcat$ is $\kappa$-compact and weakly projective. By \cref{prop:wbprojective}, it follows that $X\in \pcat^\sharp$. As $\pcat$ is generated by $\pcat_\kappa$, it follows that there is a set $\{P_i : i \in I\}$ of objects of $\pcat_\kappa$ for which $X$ is a retract of $\coprod_{i\in I} P_i$. As $X$ is $\kappa$-compact, we have
\[
\map_\pcat(X,\coprod_{i\in I}P_i)\simeq \colim_{F\subset I,\, |F| < \kappa}\map_\pcat(X,\coprod_{i\in F}P_i),
\]
and therefore $X$ is a retract of $\coprod_{i\in F}P_i$ for some $\kappa$-small set $F$. As $\pcat_\kappa$ is closed under $\kappa$-small products we have $\coprod_{i\in F}P_i \in \pcat_\kappa$, and therefore $X \in \pcat_\kappa^\sharp$.
\end{proof}

We can now characterize the $\infty$-categories of the form $\Model_\pcat$ for a theory $\pcat$.

\begin{definition}
Given an $\infty$-category $\dcat$, say that a class $\{P_s : s\in S\}$ of objects of $\dcat$ \emph{detects equivalences} if the following condition holds:
\begin{enumerate}
\item[{($\ast$)}] 
A morphism $f\colon X \to Y$ in $\dcat$ is an equivalence provided $\map_\dcat(P_s,X)\to\map_\dcat(P_s,Y)$ is an equivalence for all $s\in S$.
\end{enumerate}
\end{definition}

\begin{remark}
Let $\{P_s : s\in S\}$ be a class of objects in an $\infty$-category $\dcat$, and let $\pcat_0\subset\dcat$ be the full subcategory generated by this class. Then $\{P_s : s\in S\}$ detects equivalences if and only if the restricted Yoneda embedding $\dcat \to \Fun(\pcat_0^\op,\spaces)$ is conservative.
\end{remark}

\begin{definition}
Let $\dcat$ be an $\infty$-category which admits finite limits, small coproducts of weakly projective objects, and geometric realizations of simplicial objects satisfying the Kan condition. We say that $\dcat$ is \emph{weakly projectively generated} if it satisfies the following condition:
\begin{itemize}
\item There is a set $\{P_s : s\in S\}$ of weakly projective objects in $\dcat$ which detects equivalences.
\end{itemize}
In this case, we say that $\{P_s : s\in S\}$ is a set of weakly projective generators for $\dcat$.
\end{definition}

\begin{lemma}
Let $\dcat$ be an $\infty$-category which admits finite limits and geometric realizations of simplicial objects satisfying the Kan condition. Let $P \in \pcat$ be a weakly projective object. 
\begin{enumerate}
\item If $f$ is an effective epimorphism in $\dcat$, then $\map_\dcat(P,f)$ is an effective epimorphism of spaces.
\item If a simplicial object $\Xss \in \Fun(\Delta^\op,\dcat)$ satisfies the Kan condition, then the simplicial space $\map_\dcat(P,\Xss)$ satisfies the Kan condition.
\end{enumerate}
If $\{P_s : s\in S\}$ is a set of weakly projective objects in $\dcat$ which detect equivalences, then
\begin{enumerate}[resume]
\item A morphism $f$ in $\dcat$ is an effective epimorphism if and only if $\map_\dcat(P_s,f)$ is an effective epimorphism of spaces for all $s\in S$.
\item A simplicial object $\Xss\in\Fun(\Delta^\op,\dcat)$ satisfies the Kan condition if and only if the simplicial space $\map_\dcat(P_s,\Xss)$ satisfies the Kan condition for all $s\in S$.
\end{enumerate}
\end{lemma}

\begin{proof}
(1)~~As $P$ is weakly projective, in particular $\map_\dcat(P,\bs)$ preserves geometric realizations of \v{C}ech nerves. As $\map_\dcat(P,\bs)$ always preserves any limits that exist in $\dcat$, it follows that $\map_\dcat(P,\bs)$ preserves the formation of \v{C}ech nerves, and therefore preserves effective epimorphisms.

(2)~~If $\Xss \in \Fun(\Delta^\op,\dcat)$, then by construction
\[
\map_\dcat(P,\Xss[\Lambda])\simeq \map_\dcat(P,\Xss)[\Lambda]
\]
for any simplicial space $\Lambda$ for which $\Xss[\Lambda]$ exists in $\dcat$. By (1), it follows that if $\Xss[\Delta^n] \to \Xss[\Lambda^n_i]$ is an effective epimorphism then so is $\map_\dcat(P,\Xss)[\Delta^n] \to \map_\dcat(P,\Xss)[\Lambda^n_i]$.

(3,4)~~We just prove (3), as (4) follows as in the proof of (2). Suppose that $f$ is a morphism in $\dcat$ for which $\map_\dcat(P_s,f)$ is an effective epimorphism for all $s\in S$. As in the proof of (1), this implies that
\[
\map_\dcat(P_s,\check{C}(f)) \to \map_\dcat(P_s,Y)
\]
is a colimit diagram for all $s\in S$. As $P_s$ is weakly projective, this means that
\[
\map_\dcat(P_s,|\check{C}(f)|) \simeq \map_\dcat(P_s,Y)
\]
for all $s\in S$. Therefore $|\check{C}(f)|\simeq Y$ and $f$ is an effective epimorphism.
\end{proof}

\begin{theorem}
\label{thm:weaklyprojectivecats}
Let $\dcat$ be an $\infty$-category. Then there exists an equivalence $\dcat\simeq\Model_\pcat$ for a theory $\pcat$ if and only if $\dcat$ is weakly projectively generated. 
\end{theorem}

\begin{proof}
We have already established that if $\pcat$ is a theory then $\Model_\pcat$ is weakly projectively generated, so suppose conversely that $\dcat$ is a weakly projectively generated $\infty$-category. Let $\pcat\subset\dcat$ be the full subcategory generated by a set of weakly projective generators under coproducts and retracts. Then $\pcat$ is a theory, and there exists a restricted Yoneda embedding
\[
\nu\colon \dcat\to\Model_\pcat.
\]
We claim that $\nu$ is an equivalence.

The assumption that $\pcat$ is generated by a set of weakly projective generators for $\dcat$ implies that $\nu$ is conservative. It therefore suffices to prove that $\nu$ admits a fully faithful left adjoint $L$. To show that $\nu$ admits a left adjoint $L$, we must verify that if $X \in \Model_\pcat$ then
\[
\map_\pcat(X,\nu(\bs))\colon \dcat\to\spaces
\]
is corepresentable; $LX$ is then a corepresenting object. By \cref{thm:splithypercovering}, we may choose a resolution $X\simeq |\nu \Pss|$ for which $\nu \Pss$ satisfies the Kan condition. It follows that
\begin{align*}
\map_\pcat(X,\nu(\bs))&\simeq \map_\pcat(|\nu \Pss|,\nu(\bs))\\
&\simeq \Tot \map_\pcat(\nu \Pss,\nu(\bs))\\
&\simeq \Tot\map_\dcat(\Pss,\bs) \simeq \map_\dcat(|\Pss|,\bs),
\end{align*}
implying that $LX = |\Pss|$ is the desired corepresenting object. To see that $L$ is fully faithful, fix $X,Y\in \Model_\pcat$ and choose resolutions $X\simeq |\nu \Pss|$ and $Y\simeq |\nu \Qss|$ as above. As each $P_n \in \dcat$ is weakly projective, we may compute
\begin{align*}
\map_\dcat(LX,LY) &\simeq \map_\dcat(|\Pss|,|\Qss|)\\
&\simeq \lim_{n\in\Delta}\map_\dcat(P_n,|\Qss|)\\
&\simeq\lim_{n\in\Delta}\colim_{m\in\Delta^\op}\map_\dcat(P_n,Q_m)\\
&\simeq\lim_{n\in\Delta}\colim_{m\in\Delta^\op}\map_\pcat(\nu P_n,\nu Q_m)\\
&\simeq \map_\pcat(|\nu \Pss|,|\nu \Qss|) = \map_\pcat(X,Y),
\end{align*}
proving that $L$ is fully faithful as claimed.
\end{proof}

\begin{example}
Let $\ccat$ be a small $\infty$-category. Then the representables form a set of weakly projective generators for the presheaf category $\presheaves(\ccat)$. Therefore if we write $\ccat^\amalg \subset \presheaves(\ccat)$ for the full subcategory generated under coproducts by the representables, then 
\[
\Model_{\ccat^\amalg}\simeq\presheaves(\ccat).
\]
This is a manifestation of the fact that $\ccat^\amalg$ is equivalent to the free cocompletion of $\ccat$ under small coproducts.
\end{example}

\subsection{Strongly projectively generated \texorpdfstring{$\infty$}{infty}-categories}
\label{ssec:stronglyprojectivelygenerated}

In the previous section, we characterized $\infty$-categories of the form $\Model_{\pcat}$ for a theory $\pcat$ in terms of a weakening of the classical notion of projectivity. In this section, we will prove analogous results in the case where $\pcat$ is Malcev, where we instead work with the following \emph{strengthening} of projectivity: 

\begin{definition}
\label{def:stronglyprojective}
Let $\ccat$ be an $\infty$-category which admits small colimits. An object $P \in \ccat$ is said to be \emph{strongly projective} if for every simplicial object $\Xss$ in $\ccat$, the simplicial space $\map_\ccat(P,\Xss)$ is a hypercovering of $\map_\ccat(P,|\Xss|)$.
\end{definition}

\begin{example}
Let $\pcat$ be a Malcev theory. Then $\nu P \in \Model_\pcat$ is strongly projective for $P \in \pcat$.
\end{example}

\begin{lemma}
\label{lem:strongproj}
Let $\ccat$ be a cocomplete $\infty$-category. Then an object $P \in \ccat$ is strongly projective if and only if it is projective and couniversally Kan.
\end{lemma}
\begin{proof}
First suppose that $P$ is strongly projective. Let $\Xss$ be a simplicial object in $\ccat$. As $\map_\ccat(P,\Xss)$ is a hypercovering of $\map_\ccat(P,|\Xss|)$, in particular $|\map_\ccat(P,\Xss)|\simeq \map_\ccat(P,|\Xss|)$, and so $P$ is projective. By \cite[Corollary A.5.6.4]{lurie_spectral_algebraic_geometry}, $\map_\ccat(P,\Xss)$ satisfies the Kan condition. As $\Xss$ was arbitrary, it follows from \cref{prop:univkancharacterize}.(3) that $P$ is couniversally Kan.

Next suppose that $P$ is projective and couniversally Kan. Let $\Xss$ be a simplicial object in $\ccat$. As $P$ is couniversally Kan, it follows from \cite[Corollary A.5.6.4]{lurie_spectral_algebraic_geometry} that $\map_\ccat(P,\Xss)$ is a hypercovering of $|\map_\ccat(P,\Xss)|$. As $P$ is projective, $|\map_\ccat(P,\Xss)|\simeq\map_\ccat(P,|\Xss|)$. Therefore $\map_\ccat(P,\Xss)$ is a hypercovering of $\map_\ccat(P,|\Xss|)$, and so $P$ is strongly projective.
\end{proof}

\begin{prop}
\label{prop:projectivesinomalcevmodels}
Let $\pcat$ be a Malcev theory. For $X \in \Model_\pcat$, the following are equivalent:
\begin{enumerate}
\item $X$ is a retract of a representable;
\item $X$ is strongly projective;
\item $X$ is projective;
\item $X$ is weakly projective.
\end{enumerate}
In particular, the idempotent completion $\pcat^\sharp$ may be identified as the full subcategory of $\Model_\pcat$ spanned by the projective objects, all of which are strongly projective.
\end{prop}
\begin{proof}
(1)$\Rightarrow$(2): By \cref{prop:pointwisegeoreal}, if $P \in \pcat$ then $\nu P \in \Model_\pcat$ is projective. The Malcev condition implies that $\nu P$ is universally Kan when considered as an object of $\Model_\pcat^\op$. As these two conditions are closed under retracts, it follows from \cref{lem:strongproj} that if $X$ is a retract of a representable then $X$ is strongly projective.

The implications (2)$\Rightarrow$(3)$\Rightarrow$(4) are clear, and (4)$\Rightarrow$(1) follows from \cref{lem:malcevcocomplete} by the same argument as in the proof of \cref{prop:wbprojective}.
\end{proof}

We also have the following bounded variant.

\begin{prop}
Let $\pcat$ be a $\kappa$-bounded Malcev theory, generated by a $\kappa$-ary theory $\pcat_\kappa\subset\pcat$. Then the idempotent completion $\pcat_\kappa^\sharp$ may be identified as the full subcategory of $\kappa$-compact projective objects of $\Model_\pcat$.
\end{prop}
\begin{proof}
This follows from \cref{prop:projectivesinomalcevmodels} and \cref{prop:compactweaklyprojective}.
\end{proof}

We now give the promised characterization of categories of models of Malcev theories.

\begin{definition}\label{def:stronglyprojectivelgenerated}
Let $\dcat$ be a cocomplete $\infty$-category. We say  that $\dcat$ is \emph{strongly projectively generated} if it satisfies the following condition:
\begin{enumerate}
\item[{($\ast$)}] There exists a set $\{P_s : s\in S\}$ of strongly projective objects in $\dcat$ which detect equivalences.
\end{enumerate}
In this case, we say that $\{P_s : s\in S\}$ is a set of projective generators for $\dcat$.
\end{definition}

\begin{theorem}
\label{thm:characterization_of_malcev_theories_in_terms_of_strong_projectivity}
Let $\dcat$ be a cocomplete $\infty$-category. Then there exists an equivalence $\dcat \simeq \Model_\pcat$ for some Malcev theory $\pcat$ if and only if $\dcat$ is strongly projectively generated.
\end{theorem}

\begin{proof}
We have already established that if $\pcat$ is a Malcev theory then $\Model_\pcat$ is strongly projectively generated, with set of strongly projective generators given by any set of generators for $\pcat$.

Conversely, suppose that $\dcat$ is a strongly projectively generated $\infty$-category. Let $\pcat_0 \subset \dcat$ be the full subcategory generated by a set of strongly projective generators for $\dcat$ and let $\pcat\subset\dcat$ be the full subcategory generated by $\pcat_0$ under coproducts. Clearly $\pcat$ is a theory, and the condition that $\pcat_0$ consists of strong projective objects ensures by \cref{lem:strongproj} that $\pcat$ is Malcev. Now \cref{cor:restriction} applied to the inclusion $\pcat\subset\dcat$ implies that the restricted Yoneda embedding defines an equivalence $\dcat\simeq\Model_\pcat$.
\end{proof}

\subsection{The additive case}
\label{subsection:prestable_infty_cats}

We can further specialize the above discussion to the case where the Malcev theory under consideration is additive.

\begin{definition}
Let $\dcat$ be a cocomplete prestable $\infty$-category (see \cite[Appendix C]{lurie_spectral_algebraic_geometry}). We say  that $\dcat$ is \emph{projectively generated} if there exists a set of projective objects $\pcat_0\subset\ccat$ satisfying the following condition:
\begin{itemize}
\item[($\ast$)] If $X \in \dcat$, then $X\simeq 0$ if and only if $\map_\dcat(P,X)\simeq \ast$ for all $P \in \pcat_0$.
\end{itemize}
In this case, we say that $\pcat_0$ is a set of \emph{projective generators} for $\dcat$.
\end{definition}

We first compare this notion of projective generation to the previous notions.

\begin{lemma}\label{lem:suspensiondelooping}
Let $\dcat$ be an additive $\infty$-category. If $P \in \dcat$ is projective, then
\[
\map_\dcat(P,\Sigma X)\simeq B\map_\dcat(P,X)
\]
for any $X \in \dcat$ which admits a suspension.
\end{lemma}
\begin{proof}
The suspension $\Sigma X$, if it exists, may be written as the geometric realization
\[
\Sigma X \simeq \colim \left(\begin{tikzcd}
0&\ar[l,shift right]\ar[l,shift left]X&\ar[l,shift right]\ar[l,shift left]\ar[l]X \oplus X&\ar[l,shift right=0.5mm]\ar[l,shift right=1.5mm]\ar[l,shift left=0.5mm]\ar[l,shift left=1.5mm]\cdots
\end{tikzcd}\right).
\]
If $P \in \dcat$ is projective, then it follows that
\[
\map_\dcat(P,\Sigma X)\simeq \colim \left(\begin{tikzcd}
\ast&\ar[l,shift right]\ar[l,shift left]\map_\dcat(P,X)&\ar[l,shift right]\ar[l,shift left]\ar[l]\map_\dcat(P,X)\times \map_\dcat(P,X)&\ar[l,shift right=0.5mm]\ar[l,shift right=1.5mm]\ar[l,shift left=0.5mm]\ar[l,shift left=1.5mm]\cdots
\end{tikzcd}\right).
\]
This is the bar resolution for $\map_\dcat(P,X)$, so $\map_\dcat(P,\Sigma X)\simeq   B \map_\dcat(P,X)$ as claimed.
\end{proof}

It seems worth observing the following consequence. 

\begin{cor}
\label{cor:neverstable}
Let $\pcat$ be a Malcev theory. If $\Model_\pcat$ is stable, then $\pcat$ is the terminal theory.
\end{cor}
\begin{proof}
As $\Model_\pcat$ is stable, it is additive. If $X \in \Model_\pcat$ and $P \in \pcat$, then iterating \cref{lem:suspensiondelooping} we find that
\[
\map_\pcat(P,X)\simeq \map_\pcat(P,\Sigma^n \Omega^n X) \simeq B^n \Omega^n \map_\pcat(P,X)\simeq \map_\pcat(P,X)_{\geq n}
\]
for all $n \geq 0$. It follows that $\map_\pcat(P,X)\simeq \ast$. As $X$ was arbitrary, this implies that $\pcat$ is the terminal theory.
\end{proof}

\begin{lemma}
\label{lem:additiveprojgen}
Let $\dcat$ be a prestable $\infty$-category. Then $\dcat$ is projectively generated if and only if it is strongly projectively generated in the sense of \cref{def:stronglyprojectivelgenerated}.
\end{lemma}
\begin{proof}
Clearly if $\dcat$ is strongly projectively generated then it is projectively generated. Suppose conversely that $\dcat$ is projectively generated, with $\pcat_0\subset\dcat$ a set of projective generators. As $\dcat$ is additive, it follows that the objects of $\pcat_0$ are strongly projective. We must therefore verify that if $f\colon X \to Y$ is a map in $\dcat$ for which $\map_\dcat(P,X) \to \map_\dcat(P,Y)$ is an equivalence for all $P \in \pcat_0$, then in fact $f$ is an equivalence.

As $\dcat$ is prestable, $\Sigma\colon \dcat\to\dcat$ is fully faithful, so it suffices to prove that $\Sigma f$ is an equivalence. If $P \in \pcat_0$, then as $\map_\dcat(P,\Sigma X)\simeq B \map_\dcat(P,X)$ by \cref{lem:suspensiondelooping}, we find that if $f\colon X \to Y$ induces an equivalence $\map_\dcat(P,X) \to \map_\dcat(P,Y)$ then $\Sigma f \colon \Sigma X \to \Sigma Y$ induces an equivalence $\map_\dcat(P,\Sigma X) \to \map_\dcat(P,\Sigma Y)$. The morphism $\Sigma f\colon \Sigma X \to \Sigma Y$ sits in a bicartesian square
\begin{center}\begin{tikzcd}
Z\ar[r]\ar[d]&\Sigma X \ar[d,"\Sigma f"]\\
0\ar[r]&\Sigma Y
\end{tikzcd}.\end{center}
As this square is cocartesian, to show that $\Sigma f$ is an equivalence it suffices to show that $Z\simeq 0$. As this square is cartesian, we see that if $\map(P,\Sigma X)\to\map(P,\Sigma Y)$ is an equivalence then $\map(P,Z) \simeq 0$. If this holds for all $P\in \pcat_0$, then $Z\simeq 0$ by the assumption that $\pcat_0$ is a set of projective generators for $\dcat$, completing the proof.
\end{proof}

\begin{prop}
\label{prop:projprestable}
Let $\dcat$ be an $\infty$-category. Then $\dcat\simeq\Model_\pcat$ for an additive Malcev theory if and only if $\dcat$ is a projectively generated prestable $\infty$-category. Moreover, in this case,
\begin{enumerate}
\item One may take $\pcat\subset\dcat$ to be the full subcategory of projective objects;
\item $\Model_\pcat$ embeds into its stabilization $\LMod_\pcat$, which may be identified with the full subcategory
\[
\LMod_\pcat\subset\Fun(\pcat^\op,\spectra)
\]
of product-preserving functors $\pcat^\op\to\spectra$.
\end{enumerate}
\end{prop}
\begin{proof}
The claim that if $\pcat$ is an additive Malcev theory then $\Model_\pcat$ is a prestable $\infty$-category with stabilization $\LMod_\pcat$ as described appears, for finitary theories, in \cite[Proposition C.1.5.7, Remark C.1.5.9]{lurie_spectral_algebraic_geometry}, and the same proof applies here. The converse holds by combining \cref{thm:characterization_of_malcev_theories_in_terms_of_strong_projectivity} and \cref{lem:additiveprojgen}.
\end{proof}

\begin{example}
\label{ex:connectivering}
Let $R$ be a connective $\bfE_1$-ring spectrum, and let $\frees_0(R)\subset\LMod_R$ be the full subcategory generated under $R$ by direct sums. Then $\frees_0(R)$ is an $\omega$-bounded Malcev theory and the restricted Yoneda embedding
\[
\nu\colon \LMod_R^{\geq 0} \to \Model_{\frees_0(R)}
\]
defines an equivalence between the $\infty$-category of connective $R$-modules and the $\infty$-category of models for $\frees_0(R)$. 
\end{example}

\begin{example}
Let $R$ be a connective $\bfE_1$-ring spectrum, and suppose that $\pi_0 R$ contains no infinitely $p$-divisible elements. Let $\frees_0^{p}(R)\subset\LMod_R$ be the full subcategory generated by objects of the form $(R^{\oplus I})_p^\wedge$ for some set $I$. Then $\frees_0^{p}(R)$ is an $\aleph_1$-bounded Malcev which is generally not $\omega$-bounded, and the restricted Yoneda embedding
\[
\nu\colon \LMod_R^{{\geq 0},\Cpl(p)} \to \Model_{\frees_0^{p}(R)}
\]
defines an equivalence between the $\infty$-category of $p$-complete connective $R$-modules and the $\infty$-category of models for $\frees_0^{p}(R)$.
\end{example}

\begin{remark}
Let $\pcat_\omega$ be a \emph{finitary} additive theory. Then there is an equivalence
\[
\Fun^R(\presheaves_\Sigma(\pcat_\omega)^\op,\spaces)\simeq \Fun^\times(\pcat_\omega,\spaces^\op)^\op\simeq \presheaves_\Sigma(\pcat_\omega^\op),
\]
where $\Fun^R$ and $\Fun^\times$ denote limit-preserving and finite product-preserving functors respectively. In particular, every \emph{compactly} projectively generated prestable $\infty$-category is dualizable as a presentable $\infty$-category.
\end{remark}

\section{Derived functors}

Recall from \S\ref{ssec:quillen} that if $\pcat$ is a discrete Malcev theory, then $\Model_\pcat$ is equivalent to the classical homotopy theory of simplicial set-valued models of $\pcat$, or what is the same, to the animation of $\Model_\pcat^\heartsuit$. This makes $\Model_\pcat$ a convenient setting for doing homotopical algebra. 

In this section, we introduce and begin our study of the \emph{derived functors} between Malcev theories, as a natural generalization of the left-derived functors of classical homotopical algebra. We make basic observations in \S\ref{subsection:basic_definitions_of_derived_functors} and study exactness properties of derived functors in \S\ref{subsection:left_and_right_exactness_of_derived_functors}. As we explain, derived functors are rarely left exact, and in \S\ref{ssec:weakleftexact} we introduce and study the more general property of being weakly left exact which is satisfied in many more examples of interest.

\subsection{Basic definitions}
\label{subsection:basic_definitions_of_derived_functors}

In \cref{thm:freecocompletion}, we observed that the $\infty$-category of models of a Malcev theory is freely generated under geometric realizations. This motivates the following definition.

\begin{definition}
\label{definition:derived_functors_between_malcev_pretheories}
Let $\pcat$ and $\qcat$ be Malcev pretheories. A \emph{derived functor} $f\colon \pcat\pto\qcat$ is either of the following equivalent pieces of data:
\begin{enumerate}
\item An arbitrary functor $f\colon \pcat\to\Model_\qcat$;
\item A geometric realization-preserving functor $f_!\colon \Model_\pcat\to\Model_\qcat$.
\end{enumerate}
We write
\[
\fun_!(\pcat,\qcat)\simeq\Fun(\pcat,\Model_\qcat)\subset \Fun(\Model_\pcat,\Model_\qcat)
\]
for the full subcategory of derived functors.
\end{definition}

\begin{example}
Let $f\colon \pcat\to\qcat$ be a homomorphism of Malcev pretheories. Then left Kan extension of $\nu\circ f \colon \pcat\to\qcat\to\Model_\qcat$ along $\nu\colon \pcat\to\Model_\pcat$ defines a functor
\[
f_!\colon \Model_\pcat\to\Model_\qcat
\]
which preserves all colimits by \cref{thm:freecocompletion}. This identifies homomorphisms $f\colon \pcat\to\qcat$ with the subclass of derived functors $f\colon \pcat\pto\qcat$ which preserve coproducts and satisfy $f(\pcat)\subset\qcat$.
\end{example}

\begin{example}
\label{example:restriction_along_a_morphism_of_malcev_theories_is_a_derived_functor}
Let $f\colon \pcat\to\qcat$ be a homomorphism between Malcev theories. Then restriction defines a right adjoint
\[
f^\ast\colon \Model_\qcat\to\Model_\pcat
\]
to $f_!$ which preserves geometric realizations by \cref{cor:restriction}, and which can therefore be regarded as a derived functor $\qcat\pto\pcat$.
\end{example}

Derived functors in the sense of \cref{definition:derived_functors_between_malcev_pretheories} are a generalization of the \emph{left-derived functors} of homological algebra.

\begin{example}\label{ex:classicalderived}
Let $\acat$ be a cocomplete abelian category with enough projectives. By \cref{thm:animation}, if $\pcat\subset\acat$ is the full subcategory of projective objects, then $\pcat$ forms a Malcev theory for which $\Model_\pcat$ is the connective derived $\infty$-category $\dcat_{\geq 0}(\acat)$ of $\acat$. Given another abelian category $\bcat$ with enough projectives and a functor $F\colon \acat \to \bcat$, the composite
\[
f\colon \pcat\subset\acat\to\bcat\to\dcat_{\geq 0}(\bcat)
\]
therefore determines a total derived functor
\[
f_!\colon \dcat_{\geq 0}(\acat) \to \dcat_{\geq 0}(\bcat).
\]
The composite
\begin{center}\begin{tikzcd}
\acat\ar[r,tail]&\dcat_{\geq 0}(\acat)\ar[r,"\Sigma^n"]&\dcat_{\geq 0}(\acat)\ar[r,"f_!"]&\dcat_{\geq 0}(\bcat)\ar[r,"\pi_q"]&\bcat
\end{tikzcd}\end{center}
is isomorphic to the \emph{$q$th derived functor of $F$ of type $n$} in the sense of Dold--Puppe \cite{doldpuppe1961homologie}. If $F$ is additive, then this is the $(q-n)$th left-derived functor of $F$ in the sense of Cartan--Eilenberg, i.e.\ in the sense of ordinary homological algebra.
\end{example}

\begin{prop}\label{prop:derivedfunctorcc}
Let $f\colon \pcat\pto\qcat$ be a derived functor of Malcev pretheories, and also write $f_!\colon \presheaves(\pcat) \to \presheaves(\qcat)$ for the left Kan extension of $f\colon \pcat\to\presheaves(\qcat)$ along $\nu\colon \pcat\to\presheaves(\pcat)$. Then the following diagram commutes:
\begin{center}\begin{tikzcd}
\Model_\pcat\ar[r,"f_!"]\ar[d,tail]&\Model_\qcat\ar[d,tail]\\
\presheaves(\pcat)\ar[r,"f_!"]&\presheaves(\qcat)
\end{tikzcd}.\end{center}
\end{prop}
\begin{proof}
First consider the top square. By \cref{lem:malcevcocomplete}.(2), the left Kan extension $f_!\colon \Model_\pcat\to\Model_\qcat$ extends $f\colon \pcat\to\Model_\qcat$ by geometric realizations. Commutativity follows as these are preserved by $\Model_\qcat\to\presheaves(\qcat)$ by \cref{prop:pointwisegeoreal}.
\end{proof}

\begin{cor}
The assignment to a Malcev theory $\pcat$ its $\infty$-category $\Model_\pcat$ of models is functorial in arbitrary functors between theories. More precisely, let $\malcevtheories'\subset\largecatinfty$ denote the \emph{full} subcategory on the Malcev theories. Then the assignment
\[
\pcat \mapsto \Model_\pcat
\]
forms a subfunctor of
\[
\presheaves(\bs)\colon \malcevtheories' \subset\largecatinfty \to \largecatinfty.
\]
\end{cor}
\begin{proof}
This is the special case of \cref{prop:derivedfunctorcc} obtained when $f\colon \pcat\pto\qcat$ is a derived functor satisfying $f(\pcat)\subset\qcat$.
\end{proof}

We note that there is a universal way to approximate an arbitrary functor by a derived one: 

\begin{prop}
\label{lemma:derivization_of_functor}
Let $\pcat, \qcat$ be Malcev pretheories. Then the inclusion 
\[
\fun_!(\pcat,\qcat) \subset \Fun(\Model_{\pcat}, \Model_{\qcat}) 
\]
of the full subcategory spanned by the derived functors admits a right adjoint 
\[
\der \colon \Fun(\Model_{\pcat}, \Model_{\qcat}) \rightarrow \fun_!(\pcat,\qcat).
\] 
\end{prop}

\begin{proof}
Under the equivalence $\fun_!(\pcat,\qcat)\simeq\Fun(\pcat,\Model_\qcat)$, the above inclusion can be identified with the left Kan extension along $\pcat \hookrightarrow \Model_{\pcat}$, which has a right adjoint given by restriction. 
\end{proof}

\begin{remark}
Concretely, if $f \colon \Model_{\pcat} \rightarrow \Model_{\qcat}$ is any functor, then $\der(f)\colon \pcat\pto\qcat$ is the derived functor associated to the composite $f\circ \nu\colon \pcat\to\Model_\pcat\to\Model_\qcat$.
\end{remark}

\begin{remark}
\label{remark:lax_monoidal_structure_on_the_derivization_of_a_functor}
In the special case of endomorphisms, the subcategory $\Endo_!(\pcat) \subseteq \End(\Model_\pcat)$ of derived endofunctors is closed under compositions. It follows that the right adjoint 
\[
\der \colon \End(\Model_\pcat) \rightarrow \Endo_!(\pcat)
\]
of \cref{lemma:derivization_of_functor} is lax monoidal with respect to composition \cite[Corollary 7.3.2.7]{higher_algebra}. 
\end{remark}

\subsection{Left and right exactness}
\label{subsection:left_and_right_exactness_of_derived_functors}

By definition, if $f\colon \pcat\pto\qcat$ is a derived functor then
\[
f_! \colon \Model_\pcat\to\Model_\qcat
\]
preserves geometric realizations. In this sense one can regard derived functors as being \emph{right-exact}. They satisfy the following further right-exactness property.

\begin{prop}
\label{prop:derivedfunctorpreservesconnectivity}
Let $f\colon \pcat\pto\qcat$ be a derived functor. Then $f_!$ preserves $n$-connective morphisms for all $-1\leq n \leq \infty$.
\end{prop}
\begin{proof}
By \cref{prop:derivedfunctorcc}, the derived functor $f_!\colon \Model_\pcat\to\Model_\qcat$ is obtained by restriction from the left adjoint $f_!\colon \presheaves(\pcat)\to\presheaves(\qcat)$. As connectivity in $\Model_\pcat$ is a pointwise condition by \cref{prop:levelwiseimage}, we therefore reduce to proving the following generalization of \cite[Proposition 6.5.1.16.(4)]{lurie_higher_topos_theory}:
\begin{itemize}
\item Let $L\colon \xcat\to\ycat$ be a colimit-preserving functor between $\infty$-topoi. Then $L$ preserves $n$-connective morphisms for all $n\geq 0$.
\end{itemize}
As $n$-connective and $(n-1)$-truncated morphisms form a factorization system on any $\infty$-topos \cite[Example 5.2.8.16]{lurie_higher_topos_theory}, it suffices to prove that if $f$ is an $n$-connective morphism in $\xcat$, then $Lf \perp g$ for all $(n-1)$-truncated morphisms $g$ in $\ycat$. Write $R$ for the right adjoint of $L$. Then $Lf \perp g$ if and only if $f \perp Rg$ \cite[Remark 5.2.8.7]{lurie_higher_topos_theory}, so it suffices to prove that $R$ preserves $(n-1)$-truncated morphisms, which holds as $R$ preserves limits \cite[Proposition 5.5.6.16]{lurie_higher_topos_theory}.
\end{proof}

On the other hand, we have the following.

\begin{definition}
A derived functor $f\colon \pcat\pto\qcat$ is said to be \emph{left exact} if $f_!\colon \Model_\pcat\to\Model_\qcat$ preserves finite limits.
\end{definition}

Left exactness is a very stringent condition on a derived functor. 

\begin{example}
\label{ex:derived functor need not preserve pullback}
Let $\pcat \subseteq \spectra$ be the full subcategory spanned by coproducts of the sphere spectrum. Then the derived functor associated to the truncation $\tau\colon \pcat\to\h\pcat$ may be identified with the functor of base-change
\[
\integers \otimes_{\thesphere} - \colon \spectra^{\geq 0} \rightarrow \Mod_{\integers}^{\geq 0}. 
\]
As limits in $\Mod_\integers^{\geq 0}$ are given by connective covers of the corresponding limit taken in $\Mod_\integers$, we may compute
\[
0 = \lim (0 \rightarrow \integers \leftarrow 0) \neq \integers \otimes_{\thesphere} \lim (0 \rightarrow \thesphere \leftarrow 0) \simeq \integers \otimes_\thesphere \Omega \thesphere_{\geq 1}
\]
in $\Model_\integers^{\geq 0}$. In particular, $\tau_!$ is not left exact.
\end{example}

Left exactness is sufficient to guarantee a counterpart to \cref{prop:derivedfunctorpreservesconnectivity}.

\begin{prop}\label{prop:leftexactpreservestruncations}
If $f\colon \pcat\pto\qcat$ is left exact, then $f_!\colon \Model_\pcat\to\Model_\qcat$ preserves $n$-truncated morphisms for all $-2\leq n \leq \infty$.
\end{prop}

\begin{proof}
This is immediate from the definition of an $n$-truncated morphism recalled in \cref{def:truncated}.
\end{proof}

\subsection{Weak left exactness}
\label{ssec:weakleftexact}

Although exactness is a very stringent condition on a derived functor, many of the derived functors that arise in higher universal algebra turn out to satisfy the following weaker exactness condition.

\begin{definition}
A derived functor $f\colon \pcat\pto\qcat$ is \emph{weakly left exact} if for every cartesian square
\begin{center}\begin{tikzcd}
X'\ar[r]\ar[d]&X\ar[d,"\alpha"]\\
Y'\ar[r]&Y
\end{tikzcd}\end{center}
in $\Model_\pcat$ in which $\alpha$ is an effective epimorphism, the comparison map
\[
f_!(X') \to f_!(Y')\times_{f_!(Y)} f_!(X)
\]
is an equivalence.
\end{definition}

For example, we have the following. 

\begin{lemma}\label{lem:prestablepushpull}
Suppose given a square
\begin{center}\begin{tikzcd}
X'\ar[r]\ar[d]&X\ar[d,"\alpha"]\\
Y'\ar[r]&Y
\end{tikzcd}\end{center}
in a prestable $\infty$-category $\dcat$, and suppose that $\alpha$ induces an epimorphism $\pi_0 X \to \pi_0 Y$ in $\dcat^\heartsuit$. Then this square is cartesian if and only if it is cocartesian. In particular, this holds if $\dcat$ is projectively generated and $\alpha$ is an effective epimorphism.
\end{lemma}

\begin{proof}
This is an easy consequence of the Mayer--Vietoris sequence. 
\end{proof}

\begin{proposition}
\label{proposition:additive_derived_functors_preserve_pullbacks_along_pi0_epis}
Let $f\colon \pcat\pto\qcat$ be a colimit-preserving derived functor between additive Malcev theories. Then $f$ is weakly left exact.
\end{proposition}

\begin{proof}
As $f_!$ preserves colimits and effective epimorphisms, this follows from \cref{lem:prestablepushpull}.
\end{proof}

On the other hand, derived functors may quickly fail to be weakly left exact in nonabelian situations.

\begin{example}
\label{warning:truncation_need_not_preserve_pullbacks_along_epis}
For $k\geq 0$, let $\spheres_k\subset\spaces_\ast$ denote the full subcategory generated by the $k$-dimensional sphere under coproducts. This is clearly an $\omega$-bounded theory, Malcev if and only if $k\geq 1$. As $S^k$ forms a weakly projective generator for $\spaces_\ast^{\geq k}$, \cref{thm:weaklyprojectivecats} implies that
\[
\Model_{\spheres_k}\simeq \spaces_\ast^{\geq k}.
\]
If $k\leq 1$, then $\spheres_k$ is a discrete theory, equivalent to the classical theory of pointed sets when $k=0$ and groups when $k=1$. If $k\geq 2$, then the Hurewicz theorem implies that the homotopy category $\h(\spheres_k)$ may be identified as the category of free abelian groups. In this case the derived functor associated to the truncation $\tau\colon \spheres_k \to \h(\spheres_k)$ can be identified with the functor
\[
\Sigma^{-k}\widetilde{C}_\bullet(\bs;\integers)\colon \spaces_\ast^{\geq k} \to \Mod_{\integers}^{\geq 0}
\]
of desuspended integral chains, which does not preserve products. As $X \to \pt$ is an effective epimorphism for any $X \in \spaces_\ast^{\geq k}$, it follows that $\tau$ is not weakly left exact.
\end{example}

\begin{example}
Let $\pcat$ be a loop theory in the sense of \cref{def:looptheory}; that is, a Malcev theory which admits $S^{1}$-tensors. In the sequel \cite{usd2}, we will show that the homomorphism
\[
\tau\colon \pcat\to\h\pcat
\]
projecting onto the homotopy category of $\pcat$ is weakly left exact. That is, the behaviour observed in \cref{warning:truncation_need_not_preserve_pullbacks_along_epis} cannot occur in the setting of loop theories. 
\end{example}

We end this section by proving a natural criterion for a derived functor to be weakly left exact, and by observing some general closure properties of such functors.

\begin{prop}
\label{theorem:criterion_for_preserving_pullbacks_along_epis}
Let $f\colon \pcat\pto\qcat$ be a derived functor. For $f$ to be weakly left exact, it suffices to assume that $f_!$ preserves pullbacks of spans of representables
\[
\nu Q' \to \nu Q \xleftarrow{\alpha} \nu P
\]
in which $\alpha$ admits a section.
\end{prop}
\begin{proof}
To show that $f$ is weakly left exact, we must show that $f_!\colon \Model_\pcat\to\Model_\qcat$ preserves cartesian diagrams of the form
\begin{center}\begin{tikzcd}
X'\ar[r]\ar[d]&X\ar[d,"\alpha"]\\
Y'\ar[r]&Y
\end{tikzcd}\end{center}
where $\alpha$ is an effective epimorphism. Choose a free resolution $Y\simeq |\nu \Qss|$ of $Y$, and consider the diagram
\begin{center}\begin{tikzcd}
(\Model_\pcat)_{/Y}\ar[r]\ar[d,"f_!"]&\Tot \left((\Model_\pcat)_{/\nu \Qss}\right)\ar[d,"\Tot f_!"]\\
(\Model_\qcat)_{/f_! Y}\ar[r]&\Tot\left( (\Model_\pcat)_{/f_! \nu \Qss}\right)
\end{tikzcd}\end{center}
of $\infty$-categories. The claim is that the left vertical morphism preserves certain products. The horizontal morphisms are equivalences by \cref{cor:universality}, as $f_! Y \simeq |f_! \nu \Qss|$. As limits in the totalizations on the right hand side are computed levelwise, and as the functors
\[
(\Model_\pcat)_{/Y} \to (\Model_\pcat)_{/\nu Q_n},\qquad (X \to Y) \mapsto (\nu Q_n \times_Y X \to \nu Q_n)
\]
send effective epimorphisms over $Y$ to effective epimorphisms over $\nu Q_n$, we therefore reduce to the case where $Y \simeq \nu Q$ is representable.

Choose a free resolution $X\simeq |\nu \Pss|$. In particular, each $\nu P_n \to X$ is an effective epimorphism, and hence so is the composite $\nu P_n \to X \to Y$. By universality of geometric realizations, the maps
\[
|Y'\times_Y \nu \Pss| \to Y'\times_Y X,\qquad |f_! Y'\times_{f_! Y} f_! \nu \Pss| \to f_! Y'\times_{f_! Y} f_! X
\]
are equivalences, so we reduce to the case where $X \simeq \nu P$ is representable. In the same way we reduce to the case where $Y'\simeq \nu Q'$ is representable. 

We conclude by observing that every effective epimorphism between representable functors admits a section.
\end{proof}

\begin{prop}
\label{prop:preexactclosureproperties}
Consider derived functors $f\colon \pcat\pto\qcat$ and $g\colon \qcat\pto\rcat$.
\begin{enumerate}
\item Suppose that $f$ and $g$ are weakly left exact. Then $g\circ f$ is weakly left exact.
\item Suppose that $g$ and $g\circ f$ are weakly left exact and that $g_!$ is conservative. Then $f$ is weakly left exact.
\item Suppose that $f$ and $g\circ f$ are weakly left exact, and that $f$ is derived from an essentially surjective homomorphism $f\colon \pcat\to\qcat$. Then $g$ is weakly left exact.
\end{enumerate}
\end{prop}
\begin{proof}
(1, 2)~~Fix a cartesian square
\begin{equation}\label{eq:wlesquare}\begin{tikzcd}
X'\ar[r]\ar[d]&X\ar[d,"\alpha"]\\
Y'\ar[r]&Y
\end{tikzcd}\end{equation}
in $\Model_\pcat$ with $\alpha$ an effective epimorphism. Then the comparison map
\[
(gf)_!(X') \to (gf)_!(Y')\times_{(gf)_!(Y)}(gf)_!(X)
\]
factors as the composite
\[
g_!f_!(X') \to g_!(f_!(Y')\times_{f_!(Y)}f_!(X)) \to g_! f_!(Y')\times_{g_! f_!(Y)}g_! f_!(X).
\]
By \cref{prop:derivedfunctorpreservesconnectivity}, $f_!$ preserves effective epimorphisms, and therefore as $g$ is weakly left exact the second map is an equivalence. For (1), if $f$ is weakly left exact, then the first morphism is an equivalence, hence the composite is, proving that $g\circ f$ is weakly left exact. For (2), if $g\circ f$ is weakly left exact, then the composite is an equivalence, hence the first map is. As $g_!$ is conservative, this implies that $f_!(X') \to f_!(Y')\times_{f_!(Y)} f_!(X)$ is an equivalence, and hence $f$ is weakly left exact.

(3)~~Fix a cartesian square as in \cref{eq:wlesquare}. As we will discuss in further detail in \S\ref{sec:monads}, it follows easily from the monadicity theorem that $f^\ast\colon \Model_\qcat\to\Model_\pcat$ is the forgetful functor of a monadic adjunction. In particular, any model $X \in \Model_\qcat$ may be written as the geometric realization of its bar resolution:
\[
X\simeq |(f_! f^\ast)^{\bullet+1}X|.
\]
This is functorial in $X$, and therefore the cartesian square \cref{eq:wlesquare} is equivalent to the geometric realization of its bar resolution, which takes the form
\begin{center}\begin{tikzcd}
(f_! f^\ast)^{\bullet+1}X'\ar[r]\ar[d]&(f_! f^\ast)^{\bullet+1}X\ar[d,"(f_! f^\ast)^{\bullet+1}\alpha"]\\
(f_! f^\ast)^{\bullet+1}Y'\ar[r]&(f_! f^\ast)^{\bullet+1}Y
\end{tikzcd}.\end{center}
As both $f_!$ and $f^\ast$ preserve connectivity, the assumption that $f_!$ is weakly left exact implies that this square of simplicial objects is levelwise cartesian. As $g\circ f \simeq h$ is weakly left exact and each $(f_! f^\ast)^{\bullet+1}\alpha$ is an effective epimorphism, it follows that the square
\begin{center}\begin{tikzcd}
g_!(f_! f^\ast)^{\bullet+1}X'\ar[r]\ar[d]&g_!(f_! f^\ast)^{\bullet+1}X\ar[d,"g_!(f_! f^\ast)^{\bullet+1}\alpha"]\\
g_!(f_! f^\ast)^{\bullet+1}Y'\ar[r]&g_!(f_! f^\ast)^{\bullet+1}Y
\end{tikzcd}\end{center}
in $\Model_\rcat$ is levelwise cartesian. As $g_!(f_! f^\ast)^{\bullet+1}\alpha$ is an effective epimorphism, it therefore remains cartesian after taking geometric realizations by \cref{corollary:for_models_of_malcev_theory_geometric_realization_commutes_with_products_and_pullbacks_along_effective_epi}
. As $g_!$ preserves geometric realizations, this proves that
\[
g_! X' \to g_! X \times_{g_! Y'} g_! Y
\]
is an equivalence, and so $g$ is weakly left exact as claimed.
\end{proof}

\section{Loop theories and loop models}
\label{section:loop_theories_and_models_as_a_deformation}

In \cref{thm:weaklyprojectivecats}, we characterized $\infty$-categories which can be written as $\infty$-categories of models of a Malcev theory as exactly those which are weakly projectively generated. However, many naturally occuring $\infty$-categories, such as any non-zero stable $\infty$-category, do not satisfy this condition. In this section, we introduce the more refined notion of a \emph{loop theory}, for which we can speak of the subcategory $\Model_{\pcat}^{\Omega} \subseteq \Model_{\pcat}$ of \emph{loop models}.

The advantage of this concept is that many more $\infty$-categories, such as $\infty$-categories of spectra and algebras therein, can be written in the form $\Model_{\pcat}^{\Omega}$ for a suitable choice of loop theory $\pcat$. In this case, the larger $\infty$-category $\Model_{\pcat}$ provides a useful deformation for which $\Model_\pcat^\Omega$ can be thought of as a generic fibre $\Model_\pcat^\Omega$.

We set up the basics of the theory in \S\ref{subsection:what_is_a_loop_theory}, and describe a useful variant in \S\ref{subsection:deloop_theories}. In \S\ref{subsection:loop_theories_and_derived_functors}, we discuss how functors defined on the subcategory of loop models can be extended to the whole $\infty$-category of models, and in \S\ref{subsection:lax_monoidality_and_actions} we study the behavior of this extension procedure with respect to compositions of endofunctors. Finally, in \S\ref{subsection:coalgebras_in_models_and_loop_models} we apply these results to study when a comonad on an $\infty$-category of loop models may be derived to a comonad on the full $\infty$-category of models.

\subsection{What is a loop theory?}
\label{subsection:what_is_a_loop_theory}

Recall that if $X \in \ccat$ is an object of an $\infty$-category and $T \in \spaces$ is a space, then the \emph{tensor} $T \otimes X \in \ccat$, if it exists, is the colimit of the constant diagram $T \to \ccat$ on $X$, characterized by the existence of a natural equivalence $\map_{\ccat}(T\otimes X,\bs)\simeq \map_{\spaces}(T,\map_\ccat(X,\bs))$.

\begin{definition}
\label{def:looptheory}
A \emph{loop theory} is a Malcev theory $\pcat$ which admits tensors by $S^1$. A \emph{loop homomorphism} $f\colon \pcat\to\qcat$ of loop theories is a homomorphism which preserves tensors by $S^1$. 
\end{definition}

\begin{definition}
\label{def:loop_model}
A model $X \in \Model_{\pcat}$ of a loop theory $\pcat$ is a \emph{loop model} if for every $P \in \pcat$, the canonical comparison map 
\[
X(S^1 \otimes P) \to X(P)^{S^1}
\]
is an equivalence. We write 
\[
\nu \colon \Model_\pcat^\Omega\subset\Model_\pcat
\]
for the full subcategory spanned by loop models. 
\end{definition}

\begin{remark}\label{rmk:loopdefns}
The notion of loop theory presented here is a generalization of the notions previously studied in \cite{balderrama2021deformations} and \cite{pstrkagowski2023moduli}. \cref{prop:loopidempotentcompletion} below implies that for idempotent complete loop theories, the only difference between \cref{def:looptheory} and \cite[{Definition 3.1.1}]{balderrama2021deformations} is our more flexible notion of a Malcev theory, see \S\ref{subsection:malcevkanhistory}. 
\end{remark}

\begin{example}
Let $\ccat$ be a cocomplete stable $\infty$-category and let $\pcat \subseteq \ccat$ be a full subcategory generated under coproducts and suspensions by a small set of objects. Then $\pcat$ is a loop theory, with $S^1 \otimes P \simeq P \oplus \Sigma P$ for $P \in \pcat$.
\end{example}

\begin{example}
Let $\pcat$ be a discrete Malcev theory. Then $\pcat$ is a loop theory, with
\[
S^1 \otimes P \simeq P
\]
for all $P \in \pcat$. The $\infty$-category $\Model_\pcat^\Omega$ of loop models of $\pcat$ may be identified as the ordinary category $\Model_\pcat^\heartsuit$ of set-valued models of $\pcat$.
\end{example}

\begin{example}\label{ex:idempotentcompletion}
If $\pcat$ is a loop theory, then so is its idempotent completion $\pcat^\sharp$, and as in \cref{prop:idempotentcompletion} restriction along the inclusion $\pcat\subset\pcat^\sharp$ induces an equivalence $\Model_\pcat^\Omega\simeq\Model_{\pcat^\sharp}^\Omega$. 
\end{example}

Because of \ref{ex:idempotentcompletion}, it is generally harmless to assume that a loop theory is idempotent complete. This is particularly convenient as a consequence of the following.

\begin{defn}
Write $\Sph\subset\spaces$ for the full subcategory generated by $S^1$ under products, wedges, and smash products (with respect to any basepoint).
\end{defn}

In particular, $\Sph$ contains all spaces equivalence to a finite product of finite wedges of positive-dimensional spheres.

\begin{prop}
\label{prop:loopidempotentcompletion}
Let $\pcat$ be an idempotent complete loop theory. Then
\begin{enumerate}
\item $\pcat$ admits tensors by all $T\in \Sph$.
\item If $T_1,T_2\in \Sph$, then for each $P \in \pcat$ there exists a (noncanonical) equivalence
\[
(T_1\times T_2)\otimes P\simeq (T_1\vee T_1\wedge T_2 \vee T_2) \otimes P
\]
in $\pcat$. In particular, if $T\in \Sph$ then $T\otimes P$ is (noncanonically) a retract of $(S^1)^n \otimes P$ for some $n \geq 0$.
\item If $X \in \Model_{\pcat}^\Omega$, then the comparison map
\[
X(T\otimes P) \to X(P)^T
\]
is an equivalence for $T \in \Sph$.
\end{enumerate}
\end{prop}

\begin{proof}
(1)~~Let $\tcat\subset\spaces$ be the class of connected spaces for which $\pcat$ admits tensors by $T \in \tcat$. We must show that $\Sph\subset\tcat$. As $S^1 \in \tcat$ by construction, it suffices to show that $\tcat$ is closed under products, wedges, and smash products. Indeed, if $T_1,T_2 \in \tcat$, then $T_1\times T_2\in \tcat$ via the formula
\[
(T_1\times T_2) \otimes P \simeq T_1 \otimes (T_2\otimes P).
\]
Applying \cref{lem:productsplitting} to $\pcat^\op$, we learn that if $P \in \pcat$ then there exists a (noncanonical) equivalence
\[
(T_1\times T_2) \otimes P \simeq (T_1 \vee (T_1\wedge T_2)\vee T_2) \otimes P;
\]
in particular, $\pcat$ admits tensors by $T_1 \vee (T_1\wedge T_2)\vee T_2$. As $\pcat$ is idempotent complete, so is $\tcat$. As $T_1\vee T_2$ and $T_1\wedge T_2$ are retracts of $T_1 \vee (T_1\wedge T_2)\vee T_2$, it follows that $\pcat$ admits tensors by $T_1\vee T_2$ and $T_1\wedge T_2$ as claimed.

(2)~~As in (1), this follows from \cref{lem:productsplitting}.

(3)~~By (2), if $T\in \Sph$ then the comparison map $X(T\otimes P) \to X(P)^T$ is a retract of the comparison map for $T = (S^1)^n$ for some $n\geq 0$, and so we reduce to the case where $T = (S^1)^n$. If $n = 1$, then this is the assumption that $X$ is a loop model, so it suffices to show that the class of $T \in \Sph$ for which $X(T\otimes P) \to X(P)^T$ is an equivalence is closed under finite products. Indeed, suppose given $T_1,T_2 \in \Sph$ for which the comparison maps
\[
X(T_1\otimes P) \to X(P)^{T_1},\qquad X(T_2\otimes P)\to X(P)^{T_2}.
\]

It follows that the comparison map filling in the top of the square
\begin{center}\begin{tikzcd}
X((T_1\times T_2)\otimes P)\ar[rr]\ar[d,"\simeq"]&&X(P)^{T_1\times T_2}\\
X(T_1\otimes (T_2\otimes P))\ar[r]&X(T_2\otimes P)^{T_1}\ar[r]&(X(P)^{T_2})^{T_1}\ar[u,"\simeq"]
\end{tikzcd}\end{center}
is an equivalence as needed.
\end{proof}

The following variant is also useful.

\begin{prop}
\label{prop:pointedlooptheory}
Let $\pcat$ be an idempotent complete pointed Malcev theory, and suppose that the objects of $\pcat$ admit grouplike $H$-comultiplications. Then $\pcat$ is a loop theory if and only if it admits suspensions, in which case $X \in \Model_\pcat$ is a loop model if and only if $X(\Sigma P) \to \Omega X(P)$ is an equivalence for all $P \in \pcat^\sharp$.
\end{prop}
\begin{proof}
The proof is essentially identical to that of \cref{prop:loopidempotentcompletion}, only using the grouplike $H$-comultiplication on an object $P \in \pcat$ to realize $S^1\otimes P \simeq \Sigma P \coprod P$.
\end{proof}

From here, loop theories and their models may be developed just as in \cite{balderrama2021deformations}. For example, the following alternate characterization of the loop models of $\pcat$ is often useful.

\begin{prop}\label{prop:loopmodelifspecialfiberdiscrete}
Let $\pcat$ be a loop theory. Then the following square commutes and is cartesian:
\begin{center}\begin{tikzcd}
\Model_\pcat^\Omega\ar[r,tail]\ar[d,"\pi_0"]&\Model_\pcat\ar[d,"\tau_!"]\\
\Model_{\h\pcat}^\heartsuit\ar[r,tail]&\Model_{\h\pcat}
\end{tikzcd}.\end{center}
In other words, a model $X \in \Model_\pcat$ is a loop model if and only if $\tau_! X$ is discrete. Moreover, the truncation $\tau_n\colon \pcat\to\h_n\pcat$ satisfies $\tau_{n!}X\simeq X_{<n}$ for any loop model $X$.
\end{prop}
\begin{proof}
The same proofs as \cite[Corollary 3.2.2, Theorem 5.4.1]{balderrama2021deformations} apply.
\end{proof}

We also have the following.

\begin{prop}\label{prop:boundedlooptheory}
Let $\pcat$ be a loop theory.
\begin{enumerate}
\item $\Model_\pcat^\Omega\subset\Model_\pcat$ is closed under all small limits and geometric realizations of simplicial objects which are levelwise split.
\item Let $\kappa$ be a regular cardinal, and suppose that $\pcat$ is generated by a $\kappa$-ary theory $\pcat_\kappa\subset\pcat$.
\begin{enumerate}
\item $\pcat_\kappa\subset\pcat$ is closed under tensors by $S^1$.
\item The equivalence $\Model_\pcat\simeq\presheaves_\Sigma^\kappa(\pcat_\kappa)$ of \cref{thm:bounded} restricts to an equivalence between $\Model_\pcat^\Omega$ and the $\infty$-category of presheaves $X$ on $\pcat_\kappa$ that preserve $\kappa$-small products and $S^1$-cotensors.
\end{enumerate}
\item In particular, if $\pcat$ is a bounded loop theory then $\Model_\pcat^\Omega$ is an accessible localization of $\Model_\pcat$, and is therefore a presentable $\infty$-category.
\end{enumerate}
\end{prop}
\begin{proof}
The proofs are straightforward variants of the proofs for the analogous statements for plain theories, and so we omit the details.
\end{proof}

\begin{warning}
\label{warning:genericfibre}
Given a bounded loop theory $\pcat$, we do not have an explicit formula for the localization
\[
L\colon \Model_\pcat\to\Model_\pcat^\Omega,
\]
except in some particular special cases, see \cref{prop:loopcompletion} and \cref{prop:delooptheories}. As a consequence, if $\pcat$ is an arbitrary non-bounded loop theory, then we do not know whether $\Model_\pcat^\Omega$ is a localization of $\Model_\pcat$, or equivalently, whether $\Model_\pcat^\Omega$ admits all small colimits.
\end{warning}

Finally, we note that any Malcev theory can be completed to a loop theory in a natural way.

\begin{defn}\label{def:loopcompletion}
The \emph{loop completion} of a Malcev theory $\pcat$ is the full subcategory
\[
\pcat_\Sigma\subset\Model_\pcat
\]
generated by $\pcat$ under coproducts and tensors by $T \in \Sph$.
\end{defn}

\begin{prop}
\label{prop:loopcompletion}
Let $\pcat$ be a Malcev theory. Then the restricted Yoneda embedding
\[
\nu\colon \Model_\pcat\to\Model_{\pcat_\Sigma}
\]
is a fully faithful right adjoint, with left adjoint given by restriction along the homomorphism $\pcat\to\pcat_\Sigma$. The essential image of $\nu$ is spanned by the loop models of $\pcat_\Sigma$; that is,\ $\nu$ induces an equivalence
\[
\Model_\pcat\simeq\Model_{\pcat_\Sigma}^\Omega.
\]
\end{prop}

\begin{proof}
The same proof as \cite[Theorem 3.3.1]{balderrama2021deformations} applies. 
\end{proof}

\begin{example}
Let $\gcat$ denote the theory of groups, and identify the $\infty$-category of models of $\gcat$ with the $\infty$-category of pointed connected spaces. If $X \in \spaces_\ast^{\geq 1}$ and $T$ is any space, then
\[
T\otimes X \simeq T_+ \wedge X.
\]
It follows that the loop completion of $\gcat_\Sigma$ of $\gcat$ is equivalent to the full subcategory of $\spaces_\ast^{\geq 1}$ spanned by the wedges of positive-dimensional spheres. This is the classical example of a loop theory, going back to work of Blanc, Dwyer, Goerss, Kan, and Stover on $\Pi$-algebras \cite{stover1990vankampen, dwyerkan1989enveloping, realization_space_of_a_pi_algebra}, interpreted $\infty$-categorically in \cite{pstrkagowski2023moduli}.
\end{example}

\subsection{Variant: Deloop theories}
\label{subsection:deloop_theories}

The following dual variant of a loop theory is also convenient for modeling stable $\infty$-categories.

\begin{definition}
\label{def:stabledelooptheory}
A \emph{stable deloop theory} is a Malcev theory $\pcat$ which is pointed, admits loops, and for which $\Omega \colon \pcat \to \pcat$ preserves coproducts. The $\infty$-category of \emph{deloop models} for a stable deloop theory $\pcat$ is the full subcategory
\[
\Model_\pcat^\Sigma\subset\Model_\pcat
\]
consisting of those models $X$ for which the comparison map
\[
X(P) \to \Omega X(\Omega P)
\]
is an equivalence for all $P \in \pcat$.
\end{definition}

The following shows that stable deloop theories are a generalization of \emph{stable loop theories} considered in \cite{balderrama2021deformations}.

\begin{prop}
Let $\pcat$ be a stable deloop theory.
\begin{enumerate}
\item $\Model_\pcat^\Sigma$ is a stable $\infty$-category.
\item $\nu \colon \pcat\to\Model_\pcat$ lands in $\Model_\pcat^\Sigma$ if and only if $\Omega\colon \pcat\to\pcat$ is fully faithful. In this case, $\pcat$ is an additive $\infty$-category.
\item Suppose $\Omega\colon \pcat\to\pcat$ is an equivalence. Then $\pcat$ is a stable loop theory, that is, $\pcat$ admits suspensions and $\Sigma \colon \pcat \rightarrow \pcat$ is an equivalence. Moreover $\Model_\pcat^\Sigma\simeq\Model_\pcat^\Omega$.
\end{enumerate}
\end{prop}
\begin{proof}
(1)~~Clearly $\Model_\pcat^\Sigma\subset\Model_\pcat$ is closed under small limits. In particular $\Model_\pcat^\Sigma$ admits finite limits, and by \cite[Corollary 1.4.2.27]{higher_algebra} is therefore stable provided $\Omega\colon \Model_\pcat^\Sigma\to\Model_\pcat^\Sigma$ is an equivalence. For $X \in \Model_\pcat$, write $X_\Omega = X \circ \Omega = \Omega^\ast X$ for the precomposition of $X$ with $\Omega$. The condition that $\Omega$ preserves coproducts ensures that $X_\Omega$ is again a model of $\pcat$, and by definition $X \in \Model_\pcat^\Sigma$ if and only if the canonical map
\[
X \to \Omega X_\Omega
\]
is an equivalence. This shows that $(\bs)_\Omega$ defines an essential inverse to $\Omega$ on $\Model_\pcat^\Sigma$, proving that $\Omega$ defines an automorphism of $\Model_\pcat^\Sigma$ as claimed.

(2)~~By definition, $\nu(Q) \in \Model_\pcat^\Sigma$ if and only if, for all $P\in \pcat$, the map
\[
\map_\pcat(P,Q) \to \Omega \map_{\pcat}(\Omega P, Q) \simeq  \map_\pcat(\Omega P,\Omega Q)
\]
is an equivalence. This holds for all $Q \in \pcat$ if and only if $\Omega$ is fully faithful. By (1), it follows that $\nu\colon \pcat\to\Model_\pcat^\Sigma$ realizes $\pcat$ as a full subcategory of a stable $\infty$-category which is closed under coproducts, and is therefore additive.

(3)~~As $\Omega\colon \pcat\to\pcat$ is an equivalence, it admits an inverse which must be given by $\Sigma\colon \pcat\to\pcat$. By (2) we see that $\pcat$ is additive, and thus if $P \in \pcat$ then $S^1\otimes P \simeq P \oplus \Sigma P \in \pcat$, implying that $\pcat$ is a loop theory. Following \cref{prop:pointedlooptheory}, we see that $\Model_\pcat^\Omega\subset\Model_\pcat$ consists of those models $X$ for which
\[
X(\Sigma P) \to \Omega X(P)
\]
is an equivalence for all $P\in \pcat$, whereas $\Model_\pcat^\Sigma\subset\Model_\pcat$ consists of those models $X$ for which
\[
X(P) \to \Omega X(\Omega P)
\]
is an equivalence for all $P \in \pcat$. As $\Sigma \dashv \Omega$ are inverse automorphisms of $\pcat$, these conditions coincide.
\end{proof}

From here we have the following variant of \cite[Theorem 3.3.3]{balderrama2021deformations}.

\begin{prop}\label{prop:stableloopmodels}
Let $\dcat$ be a stable $\infty$-category which admits small colimits, and let $\pcat\subset\dcat$ be a full subcategory closed under coproducts and desuspensions, and generated under coproducts and retracts by a small subcategory $\pcat_0\subset\pcat$.
\begin{enumerate}
\item The restricted Yoneda embedding
\[
\nu\colon \dcat\to\Model_\pcat
\]
is fully faithful on restriction to the thick subcategory generated by $\pcat$.
\item Suppose that $\pcat_0$ is a set of generators for $\dcat$, in sense that if $X\in \dcat$ then $X\simeq 0$ if and only if $\map_\dcat(P,X) \simeq \ast$ for all $P \in \pcat_0$. Then $\nu$ restricts to an equivalence $\dcat\simeq\Model_\pcat^\Sigma$ in either of the following cases:
\begin{enumerate}
\item Every object of $\pcat_0$ is compact as an object of $\dcat$;
\item There is a fixed finite diagram $\jcat$ for which every object of $\dcat$ is equivalent to a retract of the colimit of a $\jcat$-indexed diagram of objects of $\pcat$.
\end{enumerate}
\end{enumerate}
\end{prop}
\begin{proof}
The same proof as \cite[Theorem 3.3.3]{balderrama2021deformations} applies.
\end{proof}

\begin{remark}
The criteria of \cref{prop:stableloopmodels}.(2) are not sharp: there exist equivalences $\dcat\simeq\Model_\pcat^\Omega$ with $\pcat$ a stable deloop theory not satisfying either of the given hypotheses.
\end{remark}

\begin{remark}
\label{remark:writing_something_as_loop_models_requires_choice}
We note that \cref{prop:stableloopmodels} shows that, unlike for ordinary theories and their $\infty$-categories of models, the presentation of a $\infty$-category $\dcat$ as of the form $\Model_\pcat^\Sigma$ or $\Model_\pcat^\Omega$ depends in an essential way on the choice of deloop or loop theory $\pcat$.

For example, let $\Mod_\integers^{\Cpl(p)}\subset\Mod_\integers$ be the full subcategory of $p$-complete $\integers$-modules, and let $\pcat,\qcat\subset\Mod_\integers^{\Cpl(p)}$ be the full subcategories generated under coproducts, suspensions, and desuspensions by $\integers_p^\wedge$ and $\integers/(p)$ respectively. Then $\pcat$ and $\qcat$ are (de)loop theories satisfying
\[
\Model_\pcat^\Omega\simeq\Model_\pcat^\Sigma\simeq\Mod_\integers^{\Cpl(p)}\simeq\Model_\qcat^\Sigma\simeq\Model_\qcat^\Omega,
\]
but the theories $\pcat$ and $\qcat$ are quite different: for example, $\Model_\qcat$ is compactly generated whereas $\Model_\pcat$ is not. Informally, whereas $\pcat$ encodes the homotopy groups of a $p$-complete $\integers$-module, $\qcat$ instead records homotopy groups mod $p$ together with residual Bockstein information.
\end{remark}

As with loop theories, a bit more can be said for bounded deloop theories.

\begin{prop}\label{prop:delooptheories}
Let $\pcat$ be a bounded deloop theory.
\begin{enumerate}
\item $\Model_\pcat^\Sigma$ is an accessible localization of $\Model_\pcat$. In particular, $\Model_\pcat^\Sigma$ is presentable (but of a possibly higher accessibility rank than $\Model_\pcat$).
\item Suppose that $\pcat$ is generated by a $\kappa$-ary theory $\pcat_\kappa\subset\pcat$ which is closed under $\Omega$.
\begin{enumerate}
\item $\Model_\pcat^\Sigma$ is equivalent to the $\infty$-category of functors $X\colon \pcat_\kappa\to\spaces$ which preserve $\kappa$-small products and for which $X(P) \to \Omega X(\Omega P)$ is an equivalence for $P \in \pcat$.
\item If $\kappa = \omega$, then the localization $L\colon \Model_\pcat\to\Model_\pcat^\Sigma$ is determined by
\[
(LX)(P) = \colim_n \Omega^n X(\Omega^n P)
\]
for $P \in \pcat_\omega$.
\end{enumerate}
\end{enumerate}
\end{prop}
\begin{proof}
(1)~~Choose a regular cardinal $\kappa$ and $\kappa$-ary theory $\pcat_\kappa\subset\pcat$ which generates $\pcat_\kappa$. Let $\pcat_0\subset\pcat$ be the full subcategory generated by $\pcat_\kappa$ under loops. As $\Model_\pcat$ is presentable and $\pcat_0\subset\Model_\pcat$ is small, there exists a regular cardinal $\lambda \geq \kappa$ for which all objects of $\pcat_0$ are $\lambda$-compact in $\Model_\pcat$. Let $\pcat_\lambda\subset\pcat$ be the full subcategory generated by $\pcat_0$ under $\lambda$-small coproducts. As $\Omega$ preserves coproducts, we see that $\pcat_\lambda$ is still closed under $\Omega$, and that $\pcat$ is generated by the $\lambda$-ary theory $\pcat_\lambda$.

By \cref{thm:bounded}, the restricted Yoneda embedding realizes $\Model_\pcat \simeq \presheaves_\Sigma^\lambda(\pcat_\lambda)$ as the full subcategory of presheaves on $\pcat_\lambda$ which preserve $\lambda$-small products. It is then easily seen that $\Model_\pcat^\Sigma$ is equivalent to the full subcategory of $\Model_\pcat$ for which $X(P) \to \Omega X(\Omega P)$ is an equivalence for all $P\in \pcat_\lambda$. This is the subcategory of objects of $\Model_\pcat$ which are local with respect to the set of morphisms $\{\Sigma \nu \Omega P \to \nu P : P \in \pcat_\lambda\}$, and therefore forms a presentable localization of $\Model_\pcat$.

(2a) is an easy consequence of \cref{thm:bounded}, as in the above argument.

(2b)~~Write $\presheaves_\Sigma^\Sigma(\pcat_\omega)\subset\presheaves_\Sigma(\pcat_\omega)$ for the full subcategory of deloop models in the sense of (2a). Given $X \in \presheaves_\Sigma(\pcat_\omega)$, define $LX \in \presheaves_\Sigma(\pcat_\omega)$ by
\[
(LX)(P) = \colim_n \Omega^n X(\Omega^n P).
\]
This indeed lives in $\presheaves_\Sigma(\pcat_\omega)$ as filtered colimits preserve finite products. We must show that $LX \in \presheaves_\Sigma^\Sigma(\pcat_\omega)$, and that if $X \in \presheaves_\Sigma^\Sigma(\pcat_\omega)$ then the evident map $X \to LX$ is an equivalence.

To see that $LX \in \presheaves_\Sigma^\Sigma(\pcat_\omega)$, we compute for $P \in \pcat_\omega$ that
\[
(LX)(P)\simeq \colim_n \Omega^n X(\Omega^n P)\simeq \colim_n \Omega^{n+1}X(\Omega^{n+1}P)\simeq \Omega \colim_n \Omega^n X(\Omega^{n+1}P)\simeq \Omega (LX)(\Omega P),
\]
as filtered colimits commute with $\Omega$. Conversely, if $X \in \presheaves_\Sigma^\Sigma(\pcat_\omega)$ then for all $P\in \pcat$ the tower $\{\Omega^n X(\Omega^n P) : n \geq 0\}$ is equivalent to the constant tower on $X(P)$, and therefore $X \to LX$ is an equivalence.
\end{proof}

\subsection{Functoriality and loop models}
\label{subsection:loop_theories_and_derived_functors}

Given a loop theory $\pcat$, the $\infty$-category $\Model_\pcat$ can be thought of as an $\infty$-category of ``formal resolutions'' of the loop models of $\pcat$ by a given choice of free objects $\pcat\subset\Model_\pcat^\Omega$. In this section, we describe the ways in which $\Model_{\pcat}$ is functorial in $\Model_\pcat^\Omega$ and vice versa. 

To relate the two, we make use of the inclusion functor $\nu_{\pcat} \colon \Model_{\pcat}^{\Omega} \rightarrow \Model_{\pcat}$ and its left adjoint $L_{\pcat} \colon \Model_{\pcat} \rightarrow \Model_{\pcat}^{\Omega}$, which exists when $\pcat$ is bounded by \cref{prop:boundedlooptheory}. The main, elementary construction can be set up at this level of generality, so we do so first: 

\begin{construction}
\label{construction:conjugation_by_an_adjunction}
Let 
\[
L_{1} \colon \ccat_1 \rightleftarrows \dcat_{1} \noloc R_{1},\qquad L_{2} \colon \ccat_2 \rightleftarrows \dcat_{2} \noloc R_2.
\]
be adjunctions of $\infty$-categories. 
\begin{enumerate}
\item The \emph{upper conjugate} of a functor $f\colon \ccat_1\to\ccat_2$ is the composite 
\[
f^{L \dashv R} \colonequals L_{2} \circ f \circ R_{1} \colon \dcat_{1} \rightarrow \dcat_{2};
\]
\item The \emph{lower conjugate} of a functor $s\colon \dcat_1\to\dcat_2$ is the composite
\[
s_{L \dashv R} \colonequals R_{2} \circ s \circ L_{1} \colon \ccat_{1} \rightarrow \ccat_{2}. 
\]
\end{enumerate}
The units of the adjunctions $L_i \dashv R_i$ induce a natural transformation
\[
\eta = \eta_2 f\eta_1 \colon f \to (f^{L \dashv R})_{L \dashv R} = R_{2} L_{2} f R_{1} L_{1} ,
\]
while the counits induce a natural transformation 
\[
\epsilon = \epsilon_2 s \epsilon_1 \colon (s_{L \dashv R})^{L \dashv R} = L_{2} R_{2} s L_{1} R_{1} \rightarrow s .
\]
\end{construction}

\begin{proposition}\label{prop:conjugationisadjunction}
The natural transformations $\eta$ and $\epsilon$ of \cref{construction:conjugation_by_an_adjunction} are the unit and counit of an adjunction
\[
(-)^{L \dashv R}  \colon \Fun(\ccat_{1}, \ccat_{2}) \rightleftarrows \Fun(\dcat_{1}, \dcat_{2}) \noloc (-)_{L \dashv R}
\] 
\end{proposition}

\begin{proof}
The triangle identities for the (co)units follow from those of $L_{1} \dashv R_{1}$ and $L_{2} \dashv R_{2}$. 
\end{proof}

We will primarily be concerned with the case where $L_1 = L_2$ and $R_1 = R_2$. The conjugation construction is compatible with composition in the following sense: 

\begin{construction}
\label{construction:monoidal_structure_on_conjugation_by_adjunction}
Let $L \dashv R \colon \ccat \rightleftarrows \dcat$ be an adjunction. Then the counit provides a natural transformation 
\[
\mathrm{id}_{\ccat}^{L \dashv R} = LR \rightarrow \mathrm{id}_{\dcat},
\]
while the unit provides for any pair of functors $f, g \colon \ccat \rightarrow \dcat$ a natural transformation 
\[
(f \circ g)^{L \dashv R} = LfgR \rightarrow LfRLgR = f^{L \dashv R} \circ g ^{L \dashv R} 
\]
In \cite[Lemma 3.10]{brantner2021pd}, Brantner--Campos--Nuiten show that these natural transformations extend to an oplax monoidal structure on the upper conjugation functor
\[
(-)^{L \dashv R} \colon \End(\ccat) \rightarrow \End(\dcat). 
\] 
Dually, the lower conjugation 
\[
(-)_{L \dashv R} \colon \End(\dcat) \rightarrow \End(\ccat) 
\] 
is canonically lax monoidal. 
\end{construction}

\begin{remark}
Suppose that $L \dashv R$ is a localization; that is, that the unit $LR \rightarrow \mathrm{id}_{\dcat}$ is an equivalence. Then the explicit description shows that the canonical maps
\[
\mathrm{id}_{\ccat}^{L \dashv R} \rightarrow \mathrm{id}_{\dcat}
\]
and
\[
(s \circ t)_{L \dashv R} \rightarrow s_{L \dashv R} \circ t_{L \dashv R}
\]
are natural equivalences. A dual statement holds when $L \dashv R$ is a colocalization; that is, when $\mathrm{id}_{\ccat} \rightarrow RL$ is an equivalence. 
\end{remark}

We now specialize to the case of bounded loop theories and the localization adjunction $L_{\pcat} \colon \Model_{\pcat} \rightleftarrows \Model_{\pcat}^{\Omega}\noloc \nu_\pcat$ of \cref{prop:boundedlooptheory}. We slightly adapt the construction in order to stay within the world of derived functors.

\begin{definition}
\label{definition:induced_derived_functor_from_a_functor_of_loop_models_and_vice_versa}
Let $\pcat$ and $\qcat$ be bounded loop theories. The \emph{induced derived functor} construction $s \mapsto D(s)$ is the composite 
\[
D:\begin{tikzcd}
	{\Fun(\Model_{\pcat}^\Omega, \Model_{\qcat}^\Omega)} & {\Fun(\Model_{\pcat} , \Model_{\qcat} )} & {\fun_!(\pcat,\qcat)}
	\arrow["{(-)_{L \dashv \nu} }", from=1-1, to=1-2]
	\arrow["{\mathrm{der}}", from=1-2, to=1-3]
\end{tikzcd}
\]
of lower conjugation of \cref{construction:conjugation_by_an_adjunction} and the right adjoint to the inclusion of derived functors of \cref{lemma:derivization_of_functor}. Dually, the \emph{induced functor between loop models} construction $f \mapsto E(f)$ is given by the composite
\[
E:\begin{tikzcd}
	{\fun_!(\pcat,\qcat)} & {\Fun(\Model_\pcat, \Model_\qcat)} & {\Fun(\Model_\pcat^{\Omega}, \Model_\qcat^{\Omega})}
	\arrow[hook, from=1-1, to=1-2]
	\arrow["{(-)^{L \dashv i} }", from=1-2, to=1-3]
\end{tikzcd}
\] 
As a composite of adjunctions, these two functors form an adjunction $E \dashv D$. 
\end{definition}

\begin{remark}
\label{remark:associated_derived_functor_construction_is_monoidal}
In the case of endomorphisms, $E \colon \Endo_!(\pcat) \rightarrow \End(\Model_{\pcat}^{\Omega})$ is canonically oplax monoidal by a combination of \cref{construction:monoidal_structure_on_conjugation_by_adjunction} and \cref{remark:lax_monoidal_structure_on_the_derivization_of_a_functor}. Dually, its right adjoint $D$ is canonically lax monoidal. 
\end{remark}

\begin{remark}
\label{remark:concrete_description_of_the_induced_functor_between_models_from_a_functor_of_loop_models}
If $\pcat, \qcat$ are bounded loop theories and $s \colon \Model_{\pcat}^{\Omega} \rightarrow \Model_{\qcat}^{\Omega}$ is a functor, then
\[
D(s) \colon \pcat\pto\qcat
\]
can be concretely described as the derived functor associated to the composite
\[
\begin{tikzcd}
	{\pcat} & {\Model_{\pcat}^{\Omega}} & {\Model_{\qcat}^{\Omega}} & {\Model_{\qcat}}
	\arrow["\nu", from=1-1, to=1-2]
	\arrow["s", from=1-2, to=1-3]
	\arrow["\nu_{\qcat}", tail, from=1-3, to=1-4]
\end{tikzcd}.
\]
This crucially uses the fact that the image of the Yoneda embedding is contained in the subcategory of loop models.  Conversely, if $f\colon \pcat\pto\qcat$ is a derived functor, then 
\[
E(f)\colon \Model_\pcat^\Omega\to\Model_\qcat^\Omega
\]
can be concretely described as the composite
\begin{center}\begin{tikzcd}
\Model_\pcat^\Omega\ar[r,"\nu_\pcat"]& \Model_\pcat\ar[r,"f_!"]&\Model_\qcat\ar[r,"L_\qcat"]&\Model_\qcat^\Omega
\end{tikzcd}.\end{center}
\end{remark}

\begin{remark}
Using \cref{remark:concrete_description_of_the_induced_functor_between_models_from_a_functor_of_loop_models} one can define $D$ without assuming that either $\pcat$ or $\qcat$ is bounded. Similarly, unwrapping the construction shows that for $E$ to be defined one only needs the inclusion $\Model_\qcat^\Omega\subset\Model_\qcat$ to admit a left adjoint, e.g.\ for $\qcat$ to be bounded.
\end{remark}

\begin{example}
Suppose that $\pcat$ and $\qcat$ are discrete Malcev theories. Then $\pcat$ and $\qcat$ are loop theories, and \cref{remark:concrete_description_of_the_induced_functor_between_models_from_a_functor_of_loop_models} shows that the induced derived functor construction $D$ of \cref{definition:induced_derived_functor_from_a_functor_of_loop_models_and_vice_versa} is equivalent to the construction
\[
\Fun(\Model_\pcat^\heartsuit,\Model_\qcat^\heartsuit)\to \fun_!(\pcat,\qcat)\subset \Fun(\Model_\pcat,\Model_\qcat)
\]
sending an ordinary functor $\Model_\pcat^\heartsuit\to\Model_\qcat^\heartsuit$ to its total derived functor $\Model_\pcat\to\Model_\qcat$ in the sense of classical homotopical algebra \cite[Section I.4]{quillen1967homotopical}.
\end{example}

As in the classical theory of left-derived functors, there is a canonical natural transformation to a functor from its associated derived functor. This will be important later in \S\ref{subsection:coalgebras_in_models_and_loop_models} and so we record the relevant notions here.

\begin{construction}
\label{construction:comparison_map_from_a_derived_functor_into_the_original_functor}
Let $f \colon \Model_{\pcat}^{\Omega} \rightarrow \Model_{\qcat}^{\Omega}$ be a functor and $X \in \Model_{\pcat}^{\Omega}$. Then the counit of the adjunction $E \dashv D$ provides a canonical map
\[
L_\qcat(D(f)(\nu_\pcat X)) \simeq E(D(f))(X) \to f(X),
\]
adjoint to a map $D(f)(\nu_\pcat X) \to \nu_\qcat(f(X))$.
\end{construction}

\begin{definition}
\label{definition:canonical_comparison_map_from_derived_functor_to_the_original_functor}
We call the map $D(f)(\nu_\pcat X) \rightarrow \nu_\qcat(f(X))$ of \cref{construction:comparison_map_from_a_derived_functor_into_the_original_functor} the \emph{canonical comparison map}. We say that $X$ is \emph{$f$-tame} if this comparison map is an equivalence. 
\end{definition}

\begin{rmk}
Let $f\colon \Model_\pcat^\Omega\to\Model_\qcat^\Omega$ be a functor for which the counit $E(D(f)) \to f$ is an equivalence. Then the canonical comparison map $D(f)(\nu_\pcat X) \to \nu_\qcat(f(X))$ is equivalent to the unit $D(f)(X) \to \nu_\qcat L_\qcat  D(f)(X)$, and so $X$ is $f$-tame if and only if $D(f)(X)$ is a loop model.
\end{rmk}

\begin{remark}
\label{remark:explicit_description_of_the_canonical_comparison_map_between_derived_functor_and_original_functor}
If $X$ is a loop model and $\nu_{\pcat} X \simeq | \nu_{\pcat} \Pss |$ is a simplicial resolution of its underlying model by representables, then the comparison map of \cref{definition:canonical_comparison_map_from_derived_functor_to_the_original_functor} can be concretely identified with the composite
\[
D(f)(\nu_{\pcat} X) \simeq D(f)(|\nu_{\pcat} \Pss| ) \simeq | D(f)(\nu_{\pcat} \Pss)| \simeq | \nu_{\qcat} (f \Pss) |\rightarrow \nu_{\qcat}(f(| \Pss|)) \simeq \nu_{\qcat}(f(X)),
\]
where the arrow in the middle is the colimit-comparison map. Our somewhat abstract construction ensures that this map is natural in $X$ and does not depend on the choice of resolution.
\end{remark}

\begin{example}
\label{example:representables_are_tame_for_all_functors}
Any $P \in \pcat\subset\Model_\pcat^\Omega$ is $f$-tame for any functor $f$. This follows from the explicit description given in 
\cref{remark:explicit_description_of_the_canonical_comparison_map_between_derived_functor_and_original_functor}. 
\end{example}

We record the following here for later use.

\begin{lemma}
\label{lemma:induced_derived_functor_of_a_cocontinuous_functor_preserves_l_equivalences}
Let $\pcat, \qcat$ be bounded loop theories and let $f \colon \Model_{\pcat}^{\Omega} \rightarrow \Model_{\qcat}^{\Omega}$ be a functor which preserves small colimits. Then $D(f)$ preserves takes $L_{\pcat}$-equivalences to $L_{\qcat}$-equivalences. 
\end{lemma}

\begin{proof}
We must show that if $X \to Y$ is a map in $\Model_\pcat$ for which the induced map $L_\pcat X \to L_\pcat Y$ is an equivalence, then $L_\qcat D(f) (X) \to L_\qcat D(f) (Y)$ is also an equivalence. By definition, $L_\pcat$ is the presentable localization of $\Model_\pcat$ onto the full subcategory of objects local with respect to the maps $S^1 \otimes \nu_\pcat P \to \nu_\pcat(S^1\otimes P)$ for $P \in \pcat$. The criteria of \cref{thm:freecocompletion} guarantee that the composite $L_\pcat \circ D(f)$ preserves colimits, and therefore it suffices to prove that if $P \in \pcat$ then the comparison map
\[
L_\qcat D(f)(S^1 \otimes \nu_\pcat P) \to L_\qcat D(f)(\nu_\pcat(S^1\otimes P))
\]
is an equivalence. Indeed, as $f$ and $L_\qcat D(f)$ both preserve $S^1$-tensors, we have
\[
L_\qcat D(f)(S^1\otimes \nu_\pcat P)\simeq S^1 \otimes L_\qcat D(f)(\nu_\pcat P)\simeq S^1 \otimes f(P) \simeq f(S^1 \otimes P) \simeq L_\qcat D(f)(\nu_\pcat(S^1\otimes P))
\]
as claimed.
\end{proof}

In general, the adjunction $E \dashv D$ is neither a localization nor a colocalization. We now identify useful criteria under which its (co)unit is an equivalence, i.e.\ under which a functor between loop models can be recovered from its associated derived functor, and conversely under which a derived functor is derived from its associated functor on loop models.

\begin{lemma} 
\label{lemma:a_criterion_for_a_functor_between_loop_models_is_determined_by_its_derived_functor}
Let $\pcat, \qcat$ be bounded loop theories. For a functor $f \colon \Model_{\pcat}^{\Omega} \rightarrow \Model_{\qcat}^{\Omega}$, the following are equivalent: 
\begin{enumerate}
    \item The counit map $E(D(f)) \rightarrow f$ of the adjunction of \cref{definition:induced_derived_functor_from_a_functor_of_loop_models_and_vice_versa} is an equivalence, 
    \item $f$ preserves geometric realization of those colimit diagrams which are preserved by the inclusion $\nu_{\pcat} \colon \Model_{\pcat}^{\Omega} \hookrightarrow \Model_{\pcat}$. 
\end{enumerate}
\end{lemma}

\begin{proof}
By unwrapping the construction, we see that $E(D(f))$ can be identified with the composite 
\[
\begin{tikzcd}
	{\Model_{\pcat}^{\Omega}} & {\Model_{\pcat}} & {\Model_{\qcat}} & {\Model_{\qcat}^{\Omega}}
	\arrow["\nu_{\pcat}", from=1-1, to=1-2]
	\arrow["{f_{!}}", from=1-2, to=1-3]
	\arrow["L_{\qcat}", from=1-3, to=1-4]
\end{tikzcd}, 
\]
where $f_{!} = D(f) \colon \Model_{\pcat} \rightarrow \Model_{\qcat}$ is the derived functor associated to the restriction of $f$ to $\pcat$. As $f_!$ and $L_\qcat$ preserve all geometric realizations, it follows that $E(D(f))$ always preserves geometric realizations of those colimit diagrams preserved by the inclusion $\nu_\pcat\colon \Model_\pcat^\Omega\subset\Model_\pcat$, implying (1)$\Rightarrow$(2).

Conversely, suppose that $f$ preserves geometric realizations of those colimit diagrams preserved by $\nu_\pcat$. By the above paragraph, the same is also true of $E(D(f))$. Given $X \in \Model_\pcat^\Omega$, we may choose a resolution $\nu_\pcat X \simeq |\nu_\pcat\Pss|$ of its underlying model, and as $L_\pcat$ is a localization it follows that $X\simeq |\Pss|$ realizes $X$ as a geometric realization of representables which is preserved by $\nu_\pcat$. Therefore to show that $E(D(f))(X) \to f(X)$ is an equivalence we reduce to the case where $X = \nu P$ for some $P \in \pcat$, where it is automatic, as we observed in \cref{example:representables_are_tame_for_all_functors}. 
\end{proof}

\begin{lemma}
\label{lemma:a_derived_functor_can_be_recovered_from_the_induced_functor_between_loop_models_iff_it_takes_representables_to_loop_models}
Let $\pcat, \qcat$ be bounded loop theories. For a geometric realization-preserving functor $f_{!} \colon \Model_{\pcat} \rightarrow \Model_{\qcat}$, the following are equivalent: 
\begin{enumerate}
    \item The unit map $f_{!} \rightarrow D(E(f_{!}))$ is a natural equivalence, 
    \item $f_{!}$ takes representables to loop models. 
\end{enumerate}
\end{lemma}

\begin{proof}
Since both sides of the unit transformation preserve geometric realizations, they agree if and only if they agree on representables. Since $D(E(f_{!}))(\nu P) \simeq L_{\qcat}(f_{!}\nu(P))$ for $P \in \pcat$, the result follows. 
\end{proof}

Another useful application of induced functors between loop models is an explicit description of functoriality of the assignment $\pcat \mapsto \Model_\pcat^\Omega$.

\begin{proposition}
\label{proposition:loopnearmonadicity}
Let $f\colon \pcat\to\qcat$ be a loop homomorphism of loop theories. Then restriction along $f$ defines a functor
\[
f^\ast\colon \Model_\qcat^\Omega\to\Model_\pcat^\Omega.
\]
which preserves all small limits and all geometric realizations of simplicial objects which are levelwise split. If $f$ is essentially surjective, then $f^\ast$ is conservative. If $\qcat$ is bounded, then $f^\ast$ is right adjoint to $E(f_!)\colon \Model_\pcat^\Omega\to\Model_\qcat^\Omega$.
\end{proposition}

\begin{proof}
If $X \colon \qcat^\op\to\spaces$ is a loop model, then as $f$ is a loop homomorphism, $f^\ast X = X \circ f \colon \pcat^\op\to\spaces$ is again a loop model. This construction preserves all limits and all levelwise split geometric realizations by \cref{prop:boundedlooptheory}.(1), and is clearly conservative if $f$ is essentially surjective. For the adjunction, observe that if $X \in \Model_\pcat^\Omega$ and $Y \in \Model_\qcat^\Omega$ then
\begin{align*}
\map_{\Model_\qcat^\Omega}(E(f_!)X,Y)&\simeq \map_{\Model_\qcat^\Omega}(L_\qcat f_!\nu_\pcat X,Y)\simeq \map_{\Model_\pcat}(f_! \nu_\pcat X,\nu_\qcat Y)\\
&\simeq \map_{\Model_\pcat}(\nu_\pcat X,f^\ast \nu_\qcat Y) \simeq \map_{\Model_\qcat^\Omega}(X,f^\ast Y),
\end{align*}
using the adjunction $f_! \dashv f^\ast$ of \cref{example:restriction_along_a_morphism_of_malcev_theories_is_a_derived_functor}.
\end{proof}

\subsection{Lax monoidality and actions} 
\label{subsection:lax_monoidality_and_actions}

In the previous section, we described how functors between $\infty$-categories of loop models induce derived functors between $\infty$-categories of models, and vice versa. For endomorphisms, these constructions assemble into an adjunction
\[
E\colon \Endo_!(\pcat) \rightleftarrows \End(\Model_{\pcat}^{\Omega}) \noloc D,
\]
where $\Endo_{!}(\pcat) \subseteq \End(\Model_{\pcat})$ is the full subcategory spanned by functors which preserve geometric realization. Our goal in this section is to establish that this is a monoidal adjunction, in the sense that the right adjoint $D$ is canonically lax monoidal, and that this monoidality is in a precise sense compatible with the tautological actions of these endomorphism $\infty$-categories (\cref{theorem:monoidal_functor_between_lmod_induced_by_induced_derived_functor_construction}). The results in this section will be applied in \S\ref{subsection:coalgebras_in_models_and_loop_models} to our discussion of the relation between of (co)monads and (co)algebras in the $\infty$-categories of models and loop models.

\begin{notation}
\label{notation:lmodoperad_monoidal_inftycategories}
Let $\lmodoperad^{\otimes}$ be the $\infty$-operad encoding a pair of an associative algebra and a left module of \cite[{Definition 4.2.1.7}]{higher_algebra}. An \emph{$\lmodoperad$-monoidal $\infty$-category} is a cocartesian fibration 
\[
(\acat, \mcat)^{\otimes} \rightarrow \lmodoperad^{\otimes}.
\]
satisfying the conditions of \cite[{Proposition 2.1.2.12}]{higher_algebra}. This determines and is determined by 
\begin{enumerate}
    \item A monoidal $\infty$-category $\acat$;
    \item An $\infty$-category $\mcat$;
    \item An action of $\acat$ on $\mcat$. 
\end{enumerate}
\end{notation}

\begin{recollection}[{Lax morphisms of $\infty$-operads}] 
If $(\acat_{1}, \mcat_{1}), (\acat_{2}, \mcat_{2})$ are $\lmodoperad$-monoidal $\infty$-categories, then a \emph{lax morphism} $f \colon (\acat_{1}, \mcat_{1}) \rightarrow (\acat_{2}, \mcat_{2})$ is a commutative diagram 
\[
\begin{tikzcd}
	{(\acat_{1}, \mcat_{1})^{\otimes}} && {(\acat_{2}, \mcat_{2})^{\otimes}} \\
	& {\lmodoperad^{\otimes}}
	\arrow["f", from=1-1, to=1-3]
	\arrow[from=1-1, to=2-2]
	\arrow[from=1-3, to=2-2]
\end{tikzcd}
\]
where the horizontal arrow $f$ is a morphism of $\infty$-operads. 
Concretely, $f$ determines 
\begin{enumerate}
    \item A lax monoidal functor $f_{\acat} \colon \acat_{1} \rightarrow \acat_{2}$ of monoidal $\infty$-categories;
    \item A functor $f_{\mcat} \colon \mcat_{1} \rightarrow \mcat_{2}$;
    \item A natural transformation 
    \[
    f_{\acat}(a) \cdot f_{\mcat}(m) \rightarrow f_{\mcat}(a \cdot m)
    \]
    of functors $\acat_{1} \times \mcat_{1} \rightarrow \mcat_{2}$. 
\end{enumerate}
This data is subject to a variety of coherences which informally say that the given natural transformation is suitably compatible with the lax monoidal structure of $f_{\acat}$ and actions of $\acat_1$ and $\acat_2$ on $\mcat_1$ and $\mcat_2$.
\end{recollection}

\begin{recollection}[{Algebras for an $\infty$-operad}] 
\label{recollection:algebras_for_the_lmod_operad}
If $p \colon (\acat, \mcat)^{\otimes} \rightarrow \lmodoperad^{\otimes}$ is an $\lmodoperad$-monoidal $\infty$-category, then an \emph{$\lmodoperad$-algebra} is a section 
\[
s \colon \lmodoperad^{\otimes} \rightarrow (\acat, \mcat)^{\otimes}
\]
which takes inert morphisms to inert morphisms. Such an algebra can be identified with a pair $(a, m)$ of an associative algebra $a \in \acat$ and an $a$-module $m \in \mcat$. Note that composition with a lax morphism of $\lmodoperad$-monoidal $\infty$-categories preserves algebras; that is, the association 
\[
(\acat, \mcat) \mapsto \Alg_{\lmodoperad}(\acat, \mcat) 
\]
is functorial in lax morphisms. 
\end{recollection}

We can now state the main theorem of this section.

\begin{theorem}
\label{theorem:monoidal_functor_between_lmod_induced_by_induced_derived_functor_construction}
The functors $D \colon \End(\Model_{\pcat}^{\Omega}) \rightarrow \End^{\sigma}(\Model_{\pcat})$ and $\nu \colon \Model_{\pcat}^{\Omega} \hookrightarrow \Model_{\pcat}$ can be promoted to a lax functor 
\[
(D, \nu) \colon (\End(\Model_{\pcat}^{\Omega}), \Model_{\pcat}^{\Omega}) \rightarrow (\Endo_!(\pcat),\Model_\pcat)
\]
of $\lmodoperad$-monoidal $\infty$-categories. 
\end{theorem}

The rest of this section is devoted to the proof of \cref{theorem:monoidal_functor_between_lmod_induced_by_induced_derived_functor_construction}. As $D$ is the composite of two construction, the conjugation construction of \cref{construction:conjugation_by_an_adjunction} and a left Kan extension, we shall deal with these separately.

\begin{construction}
\label{construction:conjugation_construction_and_algebras_for_a_monad}
Let $L \dashv R \colon \ccat \rightleftarrows \dcat$ be an adjunction. In \cref{construction:monoidal_structure_on_conjugation_by_adjunction}, we observed that the lower conjugation construction $(-)_{L \dashv R} = R \circ - \circ L$ can be made lax monoidal. We will extend this to a lax monoidal morphism 
\[
((-)_{L \dashv R}, R) \colon (\End(\dcat), \dcat) \rightarrow (\End(\ccat), \ccat) 
\] 
of $\lmodoperad$-monoidal $\infty$-categories. 

To ground the discussion, first recall that the lax monoidal structure on $(-)_{L \dashv R}$ comes from the span of restriction functors 
\[
\begin{tikzcd}
	{\End(\dcat)} & {\End_{\Delta^{1}}(\ecat)} & {\End(\ccat)}
	\arrow["{\res_{1}}"', from=1-2, to=1-1]
	\arrow["{\res_{0}}", from=1-2, to=1-3]
\end{tikzcd},
\]
where $\ecat \rightarrow \Delta^{1}$ is the bicartesian fibration encoding $L \dashv R$, so that $\ecat \times_{\Delta^{1}} \{ 0 \} \simeq \ccat$, $\ecat \times_{\Delta^{1}} \{ 1 \}$. Here $\End_{\Delta^{1}}$ is the $\infty$-category of endomorphisms over $\Delta^{1}$, and both functors are monoidal. The monoidal structure comes from the identification of \cite[{Proof of Lemma 3.10}]{brantner2021pd}, which describes $(-)_{L \dashv R}$ as the composite of the right adjoint of $\res_{1}$, and $\res_{0}$. 

Let $\mathrm{Sec}(\ecat)$ denote the $\infty$-category of sections $s \colon \Delta^{1} \rightarrow \ecat$. The functor $\ev_{1} \colon \mathrm{Sec}(\ecat) \rightarrow \dcat$ given by $\ev_{1}(s) = s(1)$ has a right adjoint, informally given by 
\[
d \mapsto (Rd \rightarrow d),
\]
where the arrow in $\ecat$ corresponds to the counit $LRd \rightarrow d$. We can identify $\mathrm{ev}_{1}$ with the pullback 
\[
\begin{tikzcd}
	{\mathrm{Sec}(\ecat) \simeq \Fun_{\Delta^{1}}(\Delta^{1}, \ecat)} & {} & {\Fun_{\Delta^{0}}(\Delta^{0}, \Delta^{0} \times_{\Delta^{1}} \ecat) \simeq \dcat }
	\arrow["{\Delta^{0} \times_{\Delta^{1}} -}", from=1-1, to=1-3]
\end{tikzcd}
\]
along $1 \colon \Delta^{0} \rightarrow \Delta^{1}$. The $\infty$-category $\End_{\Delta^{1}}(\ecat)$ acts on $\mathrm{Sec}(\ecat)$, and the above identification shows that $\mathrm{ev}_{1}$ is $\mathrm{res}_{1}$-linear. It follows that we have a morphism 
\[
\begin{tikzcd}
	{(\End(\dcat), \dcat)^{\otimes}} && {(\End_{\Delta^{1}}(\ecat), \mathrm{Sec}(\ecat))^{\otimes}} \\
	& {\lmodoperad^{\otimes}} 
	\arrow[from=1-1, to=2-2]
	\arrow["{a_{1}}"', from=1-3, to=1-1]
	\arrow[from=1-3, to=2-2]
\end{tikzcd}
\]
of $\lmodoperad$-monoidal $\infty$-categories as in \cref{notation:lmodoperad_monoidal_inftycategories}, where $a_{1}$ can be fibrewise identified with the cartesian product of $\res_{1} \colon \End_{\Delta^{1}}(\ecat) \rightarrow \End(\dcat)$ and $\ev_{1} \colon \mathrm{Sec}(\ecat) \rightarrow \dcat$. This is fibrewise a left adjoint and so by \cite[{Corollary 7.3.2.7}]{higher_algebra} admits a right adjoint 
\[
\begin{tikzcd}
	{(\End(\dcat), \dcat)^{\otimes}} && {(\End_{\Delta^{1}}(\ecat), \mathrm{Sec}(\ecat))^{\otimes}} \\
	& {\lmodoperad^{\otimes}} 
	\arrow[from=1-1, to=2-2]
	\arrow[from=1-1, to=1-3]
	\arrow[from=1-3, to=2-2]
\end{tikzcd}
\]
which is moreover a lax morphism of $\lmodoperad$-monoidal $\infty$-categories. Composing this with the analogous 
\[
a_{0} \colon (\End_{\Delta^{1}}(\ecat), \mathrm{Sec}(\ecat))^{\otimes} \rightarrow (\End(\ccat), \ccat)^{\otimes} 
\]
induced by pullback along $0 \colon \Delta^{0} \rightarrow \Delta^{1}$ yields a lax $\lmodoperad$-monoidal morphism of the required form. 
\end{construction}

\begin{construction}
\label{construction:algebras_over_a_derivation_of_a_functor}
We will extend the lax monoidal structure on the right adjoint $\mathrm{der}$ of \cref{lemma:derivization_of_functor} to a lax morphism 
\[
(\mathrm{der}, \mathrm{id}) \colon (\End(\Model_{\pcat}), \Model_{\pcat}) \rightarrow (\Endo_!(\pcat), \Model_{\pcat}). 
\] 
Observe that the inclusion 
\[
\begin{tikzcd}
	{(\Endo_!(\pcat), \Model_{\pcat})^{\otimes}} && {(\End(\Model_{\pcat}), \Model_{\pcat})^{\otimes}} \\
	& {\lmodoperad^{\otimes}} 
	\arrow[from=1-1, to=2-2]
	\arrow["j", hook, from=1-1, to=1-3]
	\arrow[from=1-3, to=2-2]
\end{tikzcd}
\]
is an inclusion of an $\lmodoperad$-monoidal subcategory. Fibrewise, it can be identified with products of the inclusion $\Endo_!(\pcat) \rightarrow \End(\Model_{\pcat})$ and the identity of $\Model_{\pcat}$. These are both right adjoints and hence $j$ admits a right adjoint by \cite[{Corollary 7.3.2.7}]{higher_algebra}. This is the required lax morphism. 
\end{construction}

\begin{proof}[{Proof of \cref{theorem:monoidal_functor_between_lmod_induced_by_induced_derived_functor_construction}:}] 
The needed lax morphism is the composite 
\[
\begin{tikzcd}[column sep=large]
	{(\End(\Model_{\pcat}^{\Omega}), \Model_{\pcat}^{\Omega})} & {(\End(\Model_{\pcat}), \Model_{\pcat})} & {(\Endo_!(\pcat), \Model_{\pcat})}
	\arrow["{((-)_{L \dashv \nu}, \nu)}", from=1-1, to=1-2]
	\arrow["{(\mathrm{der}, \mathrm{id})}", from=1-2, to=1-3]
\end{tikzcd}
\]
of the two lax morphisms of  \cref{construction:conjugation_construction_and_algebras_for_a_monad} and \cref{construction:algebras_over_a_derivation_of_a_functor}. 
\end{proof}

\begin{remark}
\label{remark:for_endofunctors_the_lmod_lax_monoidal_multiplication_map_is_the_canonical_comparison_map}
If $f$ is an endofunctor of loop models, then the lax monoidal structure on the pair $(D, \nu)$ provides in particular a canonical morphism
\[
D(f)(\nu X) \rightarrow \nu(f(X)). 
\]
By unwrapping the construction, one sees that this map coincides with the canonical comparison map of \cref{definition:canonical_comparison_map_from_derived_functor_to_the_original_functor}. Thus, \cref{theorem:monoidal_functor_between_lmod_induced_by_induced_derived_functor_construction} can be interpreted as saying that in the case of endofunctors, the canonical comparison map is coherently compatible with composition. 
\end{remark}

\subsection{(Co)algebras in models and loop models} 
\label{subsection:coalgebras_in_models_and_loop_models}

In this section, we apply the above results to describe the behavior of the adjunctions $E \dashv D$ and $L \dashv \nu$ with respect to $\infty$-categories of (co)algebras for (co)monads. The key technical point, \cref{theorem:monoidal_functor_between_lmod_induced_by_induced_derived_functor_construction}, was already established in the previous section, and our goal here is to concretely spell out the consequences. In more detail: 
\begin{enumerate}
    \item By \cref{proposition:induced_derived_functor_construction_passes_to_algebras_over_a_monad}, $\nu$ preserves algebras for a monad, and preserves coalgebras for a comonad under certain technical assumptions as we show in \cref{proposition:induced_derived_functor_construction_gives_a_functor_on_coalgebras}, 
    \item By \cref{proposition:e_and_l_take_coalgebras_to_coalgebras}, $L$ preserves coalgebras, and preserves algebras under certain technical assumptions as we show in \cref{proposition:l_gives_a_functor_on_algebras_over_l_equivalence_preserving_monad}.
\end{enumerate}

\begin{rmk}
Although we consider both monads and comonads in this section for completeness, much of the interaction between monads, theories, and loop theories can be accessed by significantly more direct means, bypassing the $\infty$-operadic technology of the preceding subsection. We will return to this in \S\ref{ssec:malcevmonads}.
\end{rmk}

We start with the simplest case. Fix throughout this section a bounded loop theory $\pcat$.

\begin{proposition}
\label{proposition:induced_derived_functor_construction_passes_to_algebras_over_a_monad}
Let $m$ be a monad on $\Model_\pcat^\Omega$. Then $D(m)$ can be made into a monad on $\Model_\pcat$ in such a way that if $X$ is an $m$-algebra then $\nu X$ is a $D(m)$-algebra, in the sense that there exists a dashed arrow filling in the diagram
\begin{center}\begin{tikzcd}
\Alg_m(\Model_\pcat^\Omega)\ar[r,dashed]\ar[d]&\Alg_{D(m)}(\Model_\pcat)\ar[d]\\
\Model_\pcat^\Omega\ar[r,"\nu"]&\Model_\pcat
\end{tikzcd}.\end{center}
\end{proposition}

\begin{proof}
By passing to algebras using the lax morphism of \cref{theorem:monoidal_functor_between_lmod_induced_by_induced_derived_functor_construction} as in \cref{recollection:algebras_for_the_lmod_operad}, we obtain a functor 
\[
\Alg_{\lmodoperad}(\End(\Model_{\pcat}^{\Omega}), \Model_{\pcat}^{\Omega}) \rightarrow \Alg_{\lmodoperad}(\End^{\sigma}(\Model_{\pcat}^{\Omega}), \Model_{\pcat}) 
\]
between the $\infty$-categories of pairs of a monad and an algebra over it. Passing to fibres over, respectively, $m$ and $D(m)$, we obtain a functor as in the statement. 
\end{proof}

\begin{remark}
Informally, if $X \in \Alg_m(\Model_\pcat^\Omega)$, then the $D(m)$-algebra structure on $\nu X$ is given by the composite
\[
D(m(\nu X) \to \nu(m X) \to \nu(X),
\]
where the first arrow is the canonical comparison map of \cref{definition:canonical_comparison_map_from_derived_functor_to_the_original_functor} and the second arrow is the $m$-algebra structure map of $X$. 
\end{remark}

The situation for comonads and coalgebras is more subtle. If $c$ is a comonad on $\Model_\pcat^\Omega$, then there is a span
\[
D(c)\circ D(c) \rightarrow D(c\circ c) \leftarrow D(c),
\]
where the right arrow is induced by the comultiplication on $c$. If the left arrow is an equivalence, then this can be turned into an honest arrow $D(c) \to D(c)\circ D(c)$, forming the first piece of the structure of a comonad on $D(c)$. We begin by identifying conditions that guarantee that this arrow is an equivalence.

\begin{lemma}
\label{lemma:criterion_for_the_associated_derived_functor_to_respect_composition}
Let $f$ and $g$ be endofunctors of $\Model_{\pcat}^{\Omega}$. Then the following are equivalent: 
\begin{enumerate}
    \item The canonical natural transformation $D(g) \circ D(f) \rightarrow D(g \circ f)$ coming from the lax monoidal structure of $D$ is an equivalence;
    \item For every representable $P \in \pcat$, the loop model $f(P) \in \Model^{\Omega}_{\pcat}$ is $g$-tame; that is, the canonical map $D(g)(\nu f(P)) \rightarrow \nu(g(f(P))$ is an equivalence. 
\end{enumerate}
\end{lemma}

\begin{proof}
Since both sides preserve geometric realizations, the given natural transformation is an equivalence if and only if it is so on $P \in \pcat$. In this case, we can rewrite the source as 
\[
D(g)(D(f)(\nu P)) \simeq D(g)(\nu f(P))
\]
and the explicit description given in \cref{remark:explicit_description_of_the_canonical_comparison_map_between_derived_functor_and_original_functor} shows that under this identification the natural transformation coincides with the canonical comparison map. 
\end{proof}

\begin{remark}
In the statement of \cref{lemma:criterion_for_the_associated_derived_functor_to_respect_composition}, we restrict to endofunctors as this is the context in which we made $D$ lax monoidal. However, more generally there is a canonical natural transformation $D(f \circ g) \rightarrow D(f) \circ D(g)$ for an arbitrary composable pair of functors between loop models, and \cref{lemma:criterion_for_the_associated_derived_functor_to_respect_composition} holds in this context as well, with the same proof. 
\end{remark}

\begin{ex}
\label{remark:d_is_monoidal_if_the_first_functor_preserves_representables}
By \cref{example:representables_are_tame_for_all_functors}, we see that if $f(\pcat)\subset\pcat$ then $D(g)\circ D(f) \to D(g\circ f)$ is an equivalence for any endofunctor $g$. 
\end{ex}

\begin{definition}
\label{definition:self_tame_endofunctor}
We say that an endofunctor $c$ of loop models is \emph{\derivable{}} if it satisfies either of the following two equivalent conditions:
\begin{enumerate}
    \item $D(c) \circ D(c) \rightarrow D(c^{2})$ is an equivalence,
    \item $c(P)$ is $c$-tame for all representables $P \in \pcat$. 
\end{enumerate}
\end{definition}

\begin{example}
\label{example:representable_preserving_endofunctors_of_models_are_self_tame}
Suppose that $c$ is an endofunctor of loop models which satisfies $c(\pcat)\subset\pcat$. Then $c$ is \derivable{}.
\end{example}

\begin{remark}
\label{remark:for_self_tame_endofunctors_d_is_strongly_monoidal}
Since $D$ preserves the identity functor, associativity shows that if $c$ is \derivable{}, then $D$ restricts to a strongly monoidal functor on the full monoidal subcategory of $\End(\Model_\pcat^\Omega)$ generated by $c$, i.e.\ $D(c^k)\circ D(c^l) \to D(c^{k+l})$ is an equivalence for all $k,l\geq 0$. In particular, $c^{k}(P)$ is $c^{l}$-tame for all $k, l \geq 0$. 
\end{remark}

As strong monoidal functors preserve coalgebras, we find that if $c$ is a comonad on loop models whose underlying endofunctor is \derivable{}, then $D(c)$ inherits the structure of a comonad. Similar discussions apply to coalgebras over comonads: if $X \in \Model_\pcat^\Omega$ is a $c$-coalgebra, then $\nu X$ will inherit the structure of a $D(c)$-coalgebra under appropriate tameness conditions on $X$. As $c$ need not preserve $c$-tame objects, the appropriate notion of tameness in this case is the following.

\begin{definition}
\label{definition:universally_tame_model}
Let $c \in \End(\Model_{\pcat}^{\Omega})$ be an endofunctor. We say that a loop model $X \in \Model_\pcat^\Omega$ is \emph{stably $c$-tame} if $c^{k}X$ is $c$-tame for all $k \geq 0$. We write
\[
\Model_\pcat^{\Omega,\stablyctame}\subset\Model_\pcat^\Omega
\]
for the full subcategory of stably $c$-tame loop models.
\end{definition}

By construction, if $X \in \Model_\pcat^\Omega$ is stably $c$-tame, then so is $c^k(X)$ for any $k\geq 0$; in other words, the endomorphism $c$ of $\Model_\pcat^\Omega$ preserves the full subcategory $\Model_\pcat^{\Omega,\stablyctame}$. We can now state the following.

\begin{proposition}
\label{proposition:induced_derived_functor_construction_gives_a_functor_on_coalgebras}
Let $c$ be a comonad on $\Model_\pcat^\Omega$ whose underlying endofunctor is \derivable{}. Then
\begin{enumerate}
    \item $D(c)\in \Endo_!(\pcat)$ inherits from $c$ the structure of a comonad on $\Model_\pcat$;
    \item The functor $\nu\colon \Model_\pcat^\Omega\to\Model_\pcat$, when restricted to stably $c$-tame loop models, lifts canonically to provide a cartesian diagram
    \begin{center}\begin{tikzcd}
    \CoAlg_c(\Model_\pcat^{\Omega,\stablyctame})\ar[r,tail,dashed]\ar[d]&\CoAlg_{D(c)}(\Model_\pcat)\ar[d]\\
    \Model_\pcat^{\Omega,\stablyctame}\ar[r,"\nu",tail]&\Model_\pcat
    \end{tikzcd}.\end{center}
\end{enumerate}
\end{proposition}

\begin{proof}
Let $(\ccat_0, \ccat_1) \subseteq (\End(\Model_{\pcat}^{\Omega}), \Model_{\pcat}^{\Omega})$ be the pair of full subcategories spanned by pairs of the form $(c^{k}, X)$ for $k \geq 0$ and $X$ a stably $c$-tame loop model. This is an $\lmodoperad$-monoidal subcategory, and \cref{remark:for_self_tame_endofunctors_d_is_strongly_monoidal} and the stable tameness assumption imply that the restriction of $(D,\nu)$ to
\[
(D, \nu) \colon (\ccat_0, \ccat_1)^{\otimes} \rightarrow (\Endo_!(\pcat),\Model_\pcat)^{\otimes} 
\]
is (strongly) $\lmodoperad$-monoidal; that is, it preserves all cocartesian morphisms. It thus induces a morphism between the dual cartesian fibrations over $(\lmodoperad^{\otimes})^{\mathrm{op}}$ which preserves all cartesian maps. 

The sections of the the dual fibration of $(\ccat_0, \ccat_1)^{\otimes}$ which preserve inert morphisms can be identified with the $\infty$-category of pairs of a comonad whose underlying endofunctor is equivalent to $c$ together with a $c$-coalgebra $X$ whose underlying loop model is $c$-tame. Passing to fibers over our fixed comonad $c$ then provides the promised fully faithful functor
\[
\CoAlg_c(\Model_\pcat^{\Omega,\stablyctame}) \to \CoAlg_{D(c)}(\Model_\pcat),
\]
realizing $\CoAlg_c(\Model_\pcat^{\Omega,\stablyctame})$ as the full subcategory of $\CoAlg_{D(c)}(\Model_\pcat)$ spanned by those $D(c)$-coalgebras whose underlying model is of the form $\nu X$ for a stably $c$-tame loop model $X$, as claimed.
\end{proof}

To summarize, $D$ preserves monads and their algebras, and comonads and coalgebras under some assumptions. The situation with its left adjoint $E$ is completely dual.

\begin{recollection}[{Dual cartesian fibration}]
\label{recollection:oplax_monoidal_structure_on_e_and_l}
The morphism of cocartesian fibrations 
\[
\begin{tikzcd}
	{(\End(\Model_{\pcat}^{\Omega}), \Model_{\pcat}^{\Omega})^{\otimes} } && {(\Endo_!(\pcat),\Model_\pcat)^{\otimes} } \\
	& {\lmodoperad^{\otimes}}
	\arrow[from=1-1, to=1-3]
	\arrow[from=1-1, to=2-2]
	\arrow[from=1-3, to=2-2]
\end{tikzcd}
\]
of \cref{theorem:monoidal_functor_between_lmod_induced_by_induced_derived_functor_construction} can be fibrewise identified with products of $D$ and $i$, both of which are right adjoints. By a result of Haugseng-Hebestreit-Linskens-Nuiten \cite[{Theorem B}]{haugsenghebestreitlinskensnuiten2023lax}, taking left adjoints levelwise assembles to a morphism between the dual cartesian fibrations over $(\lmodoperad^{\otimes})^{\op}$, which we can interpret as an \emph{oplax} $\lmodoperad$-monoidal structure on the pair 
\[
(E, L) \colon (\Endo_!(\pcat),\Model_\pcat) \rightarrow (\End(\Model_{\pcat}^{\Omega}), \Model_{\pcat}^{\Omega}),
\]
\end{recollection}

\begin{proposition}
\label{proposition:e_and_l_take_coalgebras_to_coalgebras}
Let $c_! \in \Endo_!(\Model_\pcat)$ be a comonad. Then $E(c_!)$ inherits from $c_!$ the structure of a comonad in such a way there exists a dashed arrow filling in the diagram
\begin{center}\begin{tikzcd}
\CoAlg_{c_!}(\Model_\pcat)\ar[r,dashed]\ar[d]&\CoAlg_{E(c_!)}(\Model_\pcat^\Omega)\ar[d]\\
\Model_\pcat\ar[r,"L"]&\Model_\pcat^\Omega
\end{tikzcd}.\end{center}
\end{proposition}

\begin{proof}
Passing to inert morphism-preserving sections from the morphism of the dual cartesian fibrations of \cref{recollection:oplax_monoidal_structure_on_e_and_l} yields a functor 
\[
\CoAlg_{\lmodoperad}(\Endo_!(\pcat),\Model_\pcat), \Model_{\pcat}) \rightarrow \CoAlg_{\lmodoperad}(\End(\Model_{\pcat}^{\Omega}), \Model_{\pcat}^{\Omega})
\]
between $\infty$-categories of pairs of a comonad and a coalgebra over it. Passing to fibres over $c_{!}$ and $E(c)$ yields the needed lift of $L$. 
\end{proof}

As before, the situation for monads is more subtle, but we may proceed by restricting to subcategories over which the oplax monoidal structure on $(E,L)$ of \cref{recollection:oplax_monoidal_structure_on_e_and_l} is strong monoidal. Two natural such subcategories are given the following.

\begin{lemma}
\label{lemma:for_l_equivalence_preserving_endofunctors_e_is_strongly_monoidal}
Let $f_{!}, g_{!} \in \Endo_{!}(\pcat)$ be derived endofunctors, and suppose that at least one of the following holds:
\begin{enumerate}
    \item $f_{!}$ preserves $L$-equivalences,
    \item $g_{!}$ preserves loop models. 
\end{enumerate}
Then the oplax structure map $E(f_{!} \circ g_{!}) \rightarrow E(f_{!}) \circ E(g_{!})$ is an equivalence.
\end{lemma}

\begin{proof}
Concretely, this oplax structure map is adjoint to the composite
\[
f_{!} \circ g_{!} \rightarrow D(E(f_{!})) \circ D(E(g_{!})) \rightarrow D(E(f_{!}) \circ E(g_{!})),
\]
where the first arrow is induced by the units of $E \dashv D$ and the second arrow comes from the lax monoidal structure on $D$. As $E$ is given by conjugation along the adjunction $L \dashv \nu$, we see that the oplax structure map can be identified with the map $Lf_{!}g_{!}\nu \rightarrow Lf_{!}\nu \circ L g_{!} \nu$ induced by the unit $\mathrm{id}_{\Model_{\pcat}} \rightarrow \nu \circ L$. If $(1)$ holds, then as the unit has components given by $L$-equivalences, which are preserved by $f_{!}$ and inverted by $L$, this is a natural equivalence. If $(2)$ holds, then $g_{!} \nu$ takes values in loop models, and on these the unit is already a natural equivalence. 
\end{proof}

For models we have the following.

\begin{lemma}
\label{lemma:applying_an_endofunctor_which_preserves_l_equivalences_commutes_with_l}
Let $f_{!} \in \Endo_{!}(\pcat)$ be an endofunctor which preserves $L$-equivalences. Then for any model $X \in \Model_{\pcat}$, the map $L(f_{!}X) \rightarrow E(f_{!})(LX)$, adjoint to the composite
\[
f_{!}X \rightarrow D(E(f_{!}))(\nu L X) \rightarrow \nu(E(f_{!})LX)
\]
with first arrow induced by the units of $E \circ D$ and $L \dashv \nu$ and second arrow coming from the lax monoidal structure on $(D, \nu)$, is an equivalence. 
\end{lemma}

\begin{proof}
This adjoint can be identified with the map $Lf_{!}X \rightarrow Lf_{!}\nu L X$ induced by the unit $\mathrm{id}_{\Model_{\pcat}} \rightarrow \nu \circ L$. 
The result then follows in the same way as in the proof of \cref{lemma:for_l_equivalence_preserving_endofunctors_e_is_strongly_monoidal}. 
\end{proof}

\begin{proposition}
\label{proposition:l_gives_a_functor_on_algebras_over_l_equivalence_preserving_monad}
Let $m_{!} \in \Endo_{!}(\pcat)$ be an monad such whose underlying functor preserves $L$-equivalences. Then $E(m_!)$ inherits the structure of a comonad in such a way that if $X$ is an $m_!$-algebra then $LX$ is an $E(m_!)$-algebra, in the sense that there exists a dashed arrow filling in the diagram
\begin{center}\begin{tikzcd}
\Alg_{m_!}(\Model_\pcat)\ar[r]\ar[d]&\Alg_{E(m!)}(\Model_\pcat^\Omega)\ar[d]\\
\Model_\pcat\ar[r,"L"]&\Model_\pcat^\Omega
\end{tikzcd}.\end{center}
\end{proposition}

\begin{proof}
The proof is exactly analogous to that of \cref{proposition:induced_derived_functor_construction_passes_to_algebras_over_a_monad}.
\end{proof}

With an eye towards future applications, we explicitly record some basic criteria for the adjunctions $E \dashv D$ to restrict to a monoidal equivalence between suitable subcategories.

\begin{lemma}
\label{lemma:passing_to_loop_models_fully_faithful_on_monoidal_on_derived_functors_which_preserve_loop_models}
Let $\Endo_{!}^{\Omega}(\pcat) \subseteq \Endo_{!}(\pcat) \subseteq \End(\Model_{\pcat})$ denote the subcategory of those derived functors which preserve loop models. Then the restriction 
\[
E \colon \Endo_{!}^{\Omega}(\pcat) \rightarrow \End(\Model_{\pcat}^{\Omega})
\]
is both fully faithful and strongly monoidal. 
\end{lemma}

\begin{proof}
These two properties follow from 
\cref{lemma:a_derived_functor_can_be_recovered_from_the_induced_functor_between_loop_models_iff_it_takes_representables_to_loop_models} and \cref{lemma:for_l_equivalence_preserving_endofunctors_e_is_strongly_monoidal} respectively.
\end{proof}

\begin{lemma}
\label{lemma:endofunctors_whose_derived_functors_preserve_loop_models_are_which_can_be_recovered_from_them_form_a_monoidal_subcategory_on_which_d_is_monoidal}
Let $\ccat \subseteq \End(\Model_{\pcat}^{\Omega})$ denote the full subcategory of those $f$ for which
\begin{enumerate}
    \item The unit $E(D(f)) \rightarrow f$ is an equivalence;
    \item $D(f)$ preserves loop models. 
\end{enumerate}
Then $\ccat$ is a monoidal subcategory of $\End(\Model_{\pcat}^{\Omega})$ on which $D$ restricts to a strongly monoidal functor. 
\end{lemma}

\begin{proof}
Suppose that $f, g \in \ccat$. We first claim that the canonical map $D(g) \circ D(f) \rightarrow D(gf)$ is an equivalence. By \cref{lemma:criterion_for_the_associated_derived_functor_to_respect_composition}, this amounts to verifying that $D(g)(\nu f(P)) \simeq \nu(g(f(P))$ for all $P\in \pcat$. As $D(g)$ preserves loop models, $E(D(g)) \simeq g$ can be verified with its restriction to the subcategory of loop models, which proves the claim. This shows, in particular, that $D(gf)$ preserves loop models. 

By \cref{lemma:passing_to_loop_models_fully_faithful_on_monoidal_on_derived_functors_which_preserve_loop_models}, we have $E(D(gf)) \simeq E(D(g)D(f)) \simeq E(D(g))E(D(f)) \simeq gf$, which shows that $gf \in \ccat$, showing that $\ccat$ is a monoidal subcategory. Monoidality of $D$ then follows from the first paragraph.
\end{proof}

\begin{corollary}
\label{corollary:endofunctors_whose_derived_functor_preserves_loop_models_and_which_are_its_restriction_are_self_tame}
Let $f$ be an endofunctor of $\Model_\pcat^\Omega$ for which $E(D(f)) \simeq f$ and $D(f)$ preserves loop models. Then $f$ is \derivable{} and every loop model is stably $f$-tame. \qed
\end{corollary}

\section{Theories and monads}
\label{sec:monads}

Theories and monads form the two main approaches to classical categorical algebra. In the classical context, they are essentially equivalent: more precisely, for any set $S$, there is a one-to-one correspondence between
\begin{enumerate}
\item Discrete $S$-sorted theories; that is, locally small categories generated under coproducts by an $S$-indexed set $\{P_s : s\in S\}$ of objects;
\item Categories which are monadic over $\sets^S$. 
\end{enumerate}
For details when $S = \ast$, see for example \cite[Section 6]{wraith1969algebraic}; the general case is not essentially different. As we outlined in \cref{question:properties_of_models_outside_of_the_malcev_case}, we do not know whether this correspondence extends to higher universal algebra in general. Still, some useful things can be said, and we record these here.

\subsection{General monadicity theorems}
\label{sec:genmonadic}

The \emph{monadicity theorem} \cite[Theorem 4.7.3.5]{higher_algebra} asserts that a functor $U\colon \dcat\to\ccat$ is the forgetful functor of a monadic adjunction if and only if it is a right adjoint, is conservative, and creates $U$-split geometric realizations. Such functors abound when working with theories as a consequence of the following.

\begin{prop}
\label{prop:nearmonadic}
Let $f\colon \pcat \to\qcat$ be an essentially surjective homomorphism of pretheories. Then the restriction
\[
f^\ast\colon \largemodels_\qcat\to\largemodels_\pcat
\]
is conservative and creates limits and $f^\ast$-split geometric realizations.
\end{prop}
\begin{proof}
As $f$ is essentially surjective, $f^\ast$-split geometric realizations are in particular levelwise split geometric realizations, so this is \cref{cor:nearmonadicity}.
\end{proof}

\begin{example}
Let $\pcat$ be a theory with set of generators $\{P_s : s \in S\}$. Then
\[
U\colon \Model_\pcat \to \spaces^S,\qquad UX = \{X(P_s) : s\in S\}
\]
is equivalent to restriction along the essentially surjective homomorphism
\[
f\colon \sets^S \to \pcat,\qquad \{A_s : s\in S \} \mapsto \coprod_{s\in S} \coprod_{a\in A_s} P_s
\]
of theories, and therefore is conservative and creates limits and $U$-split geometric realizations.
\end{example}

\cref{prop:nearmonadic} is very close to implying that $f^\ast$ is the forgetful functor of a monadic adjunction: all that is missing is the existence of a left adjoint, which one expects to exist in light of the fact that $f^\ast$ preserves limits. However in the absence of additional assumptions, such as presentability, it is generally a delicate matter to bridge the gap between limit-preserving functors and right adjoints. We record some cases that can be handled.

\begin{prop}\label{prop:sometimesmonadic}
Let $\pcat$ be an $S$-sorted theory, with sorts $\{P_s : s\in S\}$. Then
\[
U\colon \Model_\pcat\to\spaces^S,\qquad UX = \{X(P_s) : s\in S\}
\]
is monadic if and only if the following condition is satisfied:
\begin{enumerate}
\item[($\ast$)] For every $S$-indexed family $\{A_s : s\in S\}$ of spaces, $\coprod_{s\in S}A_s \otimes \nu P_s$ exists in $\Model_\pcat$.
\end{enumerate}
In particular, this holds if $\Model_\pcat$ is cocomplete, such as if $\pcat$ is bounded or Malcev.
\end{prop}
\begin{proof}
By the above discussion, $U$ is monadic if and only if it admits a left adjoint $L$. This in turn holds if and only if for every $S$-indexed family of spaces $\{A_s : s\in S\}$, the functor 
\[
\map_{\spaces^S}(\{A_s : s\in S\},U(\bs))\colon \Model_\pcat\to\spaces
\]
is corepresentable, in which case $L(\{A_s : s\in S\})$ is the corepresenting object. Given $X \in \Model_\pcat$, we may identify
\[
\map_{\spaces^S}(\{A_s : s\in S\},UX)\simeq \prod_{s\in S}\map_{\spaces}(A_s,X(P_s))\simeq\prod_{s\in S}\map_{\spaces}(A_s,\map_\pcat(\nu P_s,X)).
\]
Thus if $\map_{\spaces^S}(\{A_s : s\in S\},U(\bs))$ is corepresentable then it is corepresented by $\coprod_{s\in S} A_s \otimes \nu P_s$, proving the stated characterization.

If $\pcat$ is bounded or Malcev, then $\Model_\pcat$ admits all small colimits by \cref{cor:presentable} or \cref{lem:malcevcocomplete}, and so in particular $\coprod_{s\in S}A_s \otimes \nu P_s$ may always be formed in $\Model_\pcat$.
\end{proof}

\begin{prop}
\label{prop:restrictionmonadic}
Let $f\colon \pcat\to\qcat$ be an essentially surjective homomorphism of pretheories. If $\pcat$ and $\qcat$ are bounded, or $\pcat$ is a Malcev theory, then
\[
f^\ast\colon \Model_\qcat\to\Model_\pcat
\]
is the forgetful functor of a monadic adjunction.
\end{prop}
\begin{proof}
The condition that at least $\pcat$ is a theory is necessary to ensure that $f^\ast\colon \largemodels_\qcat\to\largemodels_\pcat$ restricts to a functor $\Model_\qcat\to\Model_\pcat$. By the above discussion, $f^\ast$ is monadic if and only if it admits a left adjoint. If $\pcat$ is Malcev, then the derived functor $f_!\colon \pcat\pto\qcat$ defines a left adjoint to $f^\ast$ 

So suppose that $\pcat$ and $\qcat$ are bounded. By routine arguments, we may find a regular cardinal $\kappa$ for which $\pcat$ and $\qcat$ are generated by $\kappa$-ary theories $\pcat_\kappa\subset\pcat$ and $\qcat_\kappa\subset\qcat$ satisfying $f(\pcat_\kappa)\subset\qcat_\kappa$. By \cref{thm:bounded}, we may then identify $f^\ast$ as the restriction
\[
f^\ast\colon \presheaves_\Sigma^\kappa(\qcat_\kappa)\to\presheaves_\Sigma^\kappa(\pcat_\kappa).
\]
In particular, $f^\ast$ is a limit-preserving $\kappa$-accessible functor between presentable categories, which therefore admits a left adjoint by the adjoint functor theorem \cite[Corollary 5.5.2.9]{lurie_higher_topos_theory}.
\end{proof}

\subsection{Malcev theories and monads}
\label{ssec:malcevmonads}

Malcev theories are extremely well behaved with respect to monads that preserve geometric realizations. In particular, in this context we have the following converse to \cref{prop:restrictionmonadic}: 

\begin{prop}
\label{prop:malcevmonads}
Let $\pcat$ be a Malcev theory, and let $T$ be a monad on $\Model_\pcat$. Let $T \pcat\subset\Alg_T$ denote the essential image of $\pcat$. Then $T\pcat$ is a Malcev theory, and the restricted Yoneda embedding
\[
\Alg_T\to \Model_{T\pcat}
\]
is an equivalence if and only if $T$ preserves geometric realizations.
\end{prop}

\begin{proof}
Since the composite $\pcat \rightarrow \Model_{\pcat} \rightarrow \Alg_{T}$ preserves coproducts, it is clear that the essential image of $\pcat$ is a Malcev theory. Write $t\colon \pcat\to T \pcat$ for the induced homomorphism. If $\Alg_T \to \Model_{T\pcat}$ is an equivalence, then $T$ can be identified with the monad $t^{*} t_{!}$ associated to the adjunction $t_!\dashv t^\ast$. As $t_!$ and $t^\ast$ preserve geometric realizations, it follows that so does $T$. We must prove conversely that if $T$ preserves geometric realizations then $\Alg_T \simeq \Model_{T\pcat}$.

As the forgetful functor $\Alg_T \to \Model_\pcat$ is conservative, so is the restricted Yoneda embedding $\Alg_T \to \Model_{T\pcat}$. Because $T$ preserves geometric realizations, the $\infty$-category $\Alg_T$ admits geometric realizations, and these are created by the forgetful functor $\Alg_T \to \Model_\pcat$. In particular, if $P \in \pcat$ then $T \nu P \in \Alg_T$ is projective, and hence strongly projective by the Malcev condition. Applying \cref{thm:freecocompletion}, we see that the inclusion $T\pcat\subset\Alg_T$ extends to a fully faithful and colimit-preserving functor $\Model_{T\pcat} \to \Alg_T$ which is left adjoint to a conservative functor, and hence an equivalence. 
\end{proof}

There is a corresponding statement for loop theories, first stated in \cite[Theorem 3.3.7]{balderrama2021deformations}. We give a self-contained treatment from our perspective here, starting with the following.

\begin{lemma}
\label{lemma:essentially_surjective_loop_morphsisms_have_monadic_right_adjoint}
Let $f\colon \pcat\to\qcat$ be an essentially surjective loop homomorphism of bounded loop theories. Then
\[
f^\ast\colon \Model_\qcat^\Omega\to\Model_\pcat^\Omega
\]
is the forgetful functor of a monadic adjunction.
\end{lemma}

\begin{proof}
By the monadicity theorem, it suffices to verify that $f^\ast$ admits a left adjoint and preserves geometric realizations of $f^\ast$-split simplicial objects. By \cref{proposition:loopnearmonadicity}, $f^\ast$ admits a left adjoint and preserves geometric realizations of levelwise split simplicial objects. As $f$ is essentially surjective, all $f^\ast$-split simplicial objects are levelwise split, proving the lemma.
\end{proof}

\begin{prop}
\label{proposition:recognizing_algebras_in_loop_models_as_loop_models_of_algebras}
Let $\pcat$ be a bounded loop theory and $T$ be an accessible monad on $\Model_\pcat^\Omega$. Let $i \colon T\pcat \hookrightarrow \Alg_T$ denote the essential image of $T\colon \pcat\to\Alg_T$, and write $t\colon \pcat\to T\pcat$ for the associated loop homomorphism. Then $T\pcat$ is a bounded loop theory, and the restricted Yoneda embedding defines a comparison map
\[
\begin{tikzcd}
\Alg_T\ar[rr,"i^{*}"]\ar[dr,"U"']&&\Model_{T\pcat}^\Omega\ar[dl,"t^\ast"]\\
&\Model_\pcat^\Omega
\end{tikzcd}
\]
which is an equivalence provided all composites $T^{k}$ preserve geometric realizations of those colimit diagrams which are preserved by the inclusion $\nu_\pcat\colon \Model_\pcat^\Omega\to\Model_\pcat$.
\end{prop}

\begin{proof}
We shall moderately abuse notation by writing $T$ for both the monad $T\colon \Model_\pcat^\Omega\to\Model_\pcat^\Omega$ and its associated free functor $T\colon \Model_\pcat^\Omega\to\Alg_T$. As $T\colon \Model_\pcat^\Omega\to\Alg_T$ preserves colimits, $T\pcat$ is a loop theory. As $T$ is accessible, $T\pcat$ is bounded. Both $U$ and $t^\ast$ are forgetful functors of monadic adjunctions, the first by assumption and the second by \cref{lemma:essentially_surjective_loop_morphsisms_have_monadic_right_adjoint}. By \cref{proposition:loopnearmonadicity}, the left adjoint to $t^\ast$ is given by $E(t_{!}) = L_{T \pcat}t_{!} \nu_{\pcat}$. By \cite[Corollary 4.7.3.16]{higher_algebra}, $i^{*}$ is an equivalence if and only if the induced natural transformation
\[
t^\ast E(t_{!}) \rightarrow T
\]
of endofunctors of $\Model_\pcat^\Omega$ is an equivalence. We claim that the given assumption on $T$ guarantees that $t^\ast E(t_!)\simeq E(D(T))$, from which it then also follows by \cref{lemma:a_criterion_for_a_functor_between_loop_models_is_determined_by_its_derived_functor} that $t^\ast E(t_!)\simeq T$. 

To see that $t^\ast E(t_!)\simeq E(D(T))$, first observe that the diagram
\begin{center}\begin{tikzcd}
\Model_{T\pcat}^\Omega\ar[r,"\nu_{T\pcat}"]\ar[d,"t^\ast"]&\Model_{T\pcat}\ar[d,"t^\ast"]\\
\Model_{\pcat}^\Omega\ar[r,"\nu_\pcat"]&\Model_\pcat
\end{tikzcd}\end{center}
of right adjoints clearly commutes. Unraveling the definitions, we see that it sufffices to show that its mate, filling in the bottom right portion of the diagram
\begin{center}\begin{tikzcd}
\Alg_T\ar[dd,"U"]&&\Model_{T\pcat}\ar[ll,"i_!"']\ar[dl,"L_{T\pcat}"]\ar[dd,"t^\ast"]\\
&\Model_{T\pcat}^\Omega\ar[ul,dashed,"\tilde{i}_!"]\ar[dl,"t^\ast"',dashed]\\
\Model_\pcat^\Omega&&\Model_{\pcat}\ar[ll,"L_\pcat"']
\end{tikzcd},\end{center}
again commutes. Here, the inclusion $i\colon T\pcat\to\Model_{T\pcat}$ preserves coproducts and $S^1$-tensors, and therefore extends uniquely to a colimit-preserving functor $i_!\colon \Model_{T\pcat}\to\Alg_T$ which factors uniquely through the localization $\Model_{T\pcat}^\Omega$, proving the upper triangle. The outer square commutes, as it commutes on restriction to $T\pcat\subset\Model_{T\pcat}$ by construction and all functors involved preserve geometric realizations, and so as $L_{T\pcat}$ is a localization it now suffices to prove that the left triangle commutes.

The proof is analogous to \cref{lemma:a_criterion_for_a_functor_between_loop_models_is_determined_by_its_derived_functor}. The left triangle clearly commutes on restriction to $T\pcat\subset\Model_{T\pcat}^\Omega$, and every $X \in \Model_{T\pcat}^\Omega$ admits a levelwise resolution $X\simeq |\nu_{T\pcat}\Qss|$ in $\Model_{T\pcat}$ with $\Qss \in T\pcat$. As $t^\ast$ preserves levelwise geometric realizations and $\tilde{i}_!$ preserves all colimits, it suffices to prove that the geometric realization of $\tilde{i}_!\nu_{T\pcat}\Qss = \Qss$ is preserved by the forgetful functor $U\colon \Alg_T\to\Model_\pcat$. By \cite[Corollary 4.2.3.5]{higher_algebra}, it suffices to verify that the comparison map $|T^k U\Qss|\to T^k |U\Qss|$ is an equivalence in $\Model_\pcat$ for $k\geq 0$. Here, the geometric realization of $U\Qss = t^\ast \nu_{T\pcat}\Qss$ is preserved by the inclusion $\Model_\pcat^\Omega\subset\Model_\pcat$, so this is exactly our assumption on $T$.
\end{proof}

\section{Examples}
\label{sec:examples}

In this section, we give examples of various $\infty$-categories that can be modeled using the framework developed in this paper. We discuss connective modules in \S\ref{sssec:connectivemodules}, filtered modules in \S\ref{subsection:filtered_modules}, synthetic $R$-algebras in \S\ref{subsection:synthetic_r_algebras}, synthetic spectra in \S\ref{subsection:synthetic_spectra}, synthetic spaces in \S\ref{subsection:examples_synthetic_spaces}, and synthetic $\mathbf{E}_{k}$-rings in \S\ref{ssec:examples:syntheticekrings}. 

\subsection{Connective modules}
\label{sssec:connectivemodules}

Given a connective $\bfE_1$-ring spectrum $R$, write
\[
\cfrees(R) = \langle \bigoplus R \rangle \subset \LMod_R
\]
for the full subcategory of $\LMod_R$ generated by $R$ under coproducts. This is an additive Malcev theory. Given a homomorphism $\phi\colon R \to S$ of $\bfE_1$-ring spectra, the extension of scalars functor $R\otimes_S(\bs)\colon \LMod_R\to\LMod_S$ sends $\cfrees(R)$ into $\cfrees(S)$, and therefore this defines a functor
\[
\cfrees\colon \Alg_{\bfE_1}(\spectra^{\geq 0}) \to \malcevtheories.
\]

\begin{prop}\label{cor:cnmodules}
Let $R$ be a connective $\bfE_1$-ring. Then the restricted Yoneda embedding
\[
\nu\colon \LMod_R^{\geq 0} \to \Model_{\cfrees(R)}
\]
defines an equivalence between the $\infty$-category of connective modules over $R$ and the $\infty$-category of models for $\cfrees(R)$. Moreover, given a map $\phi\colon R \to S$ of connective $\bfE_1$-ring spectra, the diagram
\begin{center}\begin{tikzcd}
\LMod_R^{\geq 0}\ar[r,"{S\otimes_R(\bs)}"]\ar[d,"\nu","\simeq"']&\LMod_S^{\geq 0}\ar[d,"\nu","\simeq"']\\
\Model_{\cfrees(R)}\ar[r,"{\cfrees(\phi)_!}"]&\Model_{\cfrees(S)}
\end{tikzcd}\end{center}
canonically commutes, i.e. $R \mapsto \LMod_R^{\geq 0}$ and $R \mapsto \Model_{\cfrees(R)}$ agree as functors.
\end{prop}
\begin{proof}
The equivalence $\LMod_R^{\geq 0}\simeq\Model_{\cfrees(R)}$ follows from \cref{prop:projprestable}, as $R$ is a projective generator for $\LMod_R^{\geq 0}$. To make the given diagram commute, as all functors preserve colimits it suffices suffices to make it commute on restriction to the subcategory $\cfrees(R)\subset\LMod_R^{\geq 0}$, where it commutes by definition.
\end{proof}

\begin{example}
The $\infty$-category $\cfrees(R)$ is easily seen to be an $n$-category if and only if $R$ is $(n-1)$-truncated, and in general if $R \in \Alg_{\bfE_1}(\Sp^{\geq 0})$ then 
\[
\h_n\cfrees(R) \to \cfrees(R_{<n})
\]
is an equivalence. In particular, the diagrams
\begin{center}\begin{tikzcd}[column sep=large]
\LMod_R^{{\geq 0}}\ar[r,"R_{<n} \otimes_R (\bs)"]\ar[d,"\nu","\simeq"']&\LMod_{R_{<n}}^{{\geq 0}}\ar[d,"\nu","\simeq"']\\
\Model_{\cfrees(R)}\ar[r,"\tau_{n!}"]&\Model_{\h_n\cfrees(R)}
\end{tikzcd}\end{center}
commute.
\end{example}

\begin{example}
\label{ex:connectivealgebras}
Fix $ k \leq \infty$ and let $R$ be a connective $\bfE_{k+1}$-ring, so that $\LMod_R$ is a $\bfE_k$-monoidal $\infty$-category. Let
\[
\lcat_0^{\bfE_k}(R)\subset\Alg_{\bfE_k}(\LMod_R^{\geq 0})
\]
be the full subcategory spanned by those $\bfE_k$-$R$-algebras which are free on a set. By \cref{prop:malcevmonads}, the restricted Yoneda embedding defines an equivalence
\[
\Model_{\lcat_0^{\bfE_k}}(R)\simeq\Alg_{\bfE_k}(\LMod_R^{{\geq 0}}).
\]
In the special case where $k=1$, we can moreover identify $h_n\lcat_0^{\bfE_1}(R)\simeq \lcat_0^{\bfE_1}(R_{<n})$, and for any map $R\to S$ of $\bfE_2$-rings the diagram
\begin{center}\begin{tikzcd}
\Model_{\lcat_0^{\bfE_1}(R)}\ar[d,"f_!"]\ar[r,"\nu","\simeq"']&\Alg_{\bfE_1}(\LMod_R^{\geq 0})\ar[d,"S\otimes_R(\bs)"]\\
\Model_{\h_n\lcat_0^{\bfE_1}(S)}\ar[r,"\nu","\simeq"']&\Alg_{\bfE_1}(\LMod_{S}^{\geq 0})
\end{tikzcd}\end{center}
commutes. 
\end{example}

\begin{remark}
The construction $\lfrees_{0}$ does not lose any information about the ring if we remember the canonical map from the sphere; that is, it induces a fully faithful embedding 
\[
\Alg_{\bfE_1}^{\geq 0}(\spectra) \to \malcevtheories_{\cfrees(\thesphere)/}.
\]
To see this, note that passing to $\infty$-categories of models provides an equivalence
\[
\Map_{\malcevtheories_{\cfrees(\thesphere)/}}(\lfrees_0(R), \lfrees_0(S)) \simeq \Hom_{\PrL_{\spectra^{\geq 0} / }}(\LMod_{R}^{\geq 0}, \LMod_{S}^{\geq 0}), 
\]
as any colimit-preserving functor $\LMod_R^{\geq 0} \to \LMod_S^{\geq 0}$ of $\infty$-categories under $\spectra^{\geq 0}$ automatically sends $\cfrees(R)$ into $\cfrees(S)$. The latter mapping space can be identified with $\map_{\Alg_{\bfE_1}(\Sp)}(R,S)$ by \cite[{Proposition 7.1.2.6}]{lurie2017higheralgebra}. 
\end{remark}

\begin{rmk}\label{rmk:manygeneralizations}
The construction of $\lfrees_0(R)$ and equivalence $\Model_{\lfrees_0(R)}\simeq\LMod_R^{\geq 0}$ can be generalized to various other settings, such as to equivariant ring spectra, completed ring spectra, and more. The proof in each case is the same, and proceeds by identifying a family of projective generators then applying \cref{prop:projprestable}.
\end{rmk}

\subsection{Filtered modules}
\label{subsection:filtered_modules}

In the previous example, we saw that the $\infty$-category of connective modules over a connective ring spectrum can be described as the $\infty$-category of models for a Malcev theory. In this section, we show that the extra flexibility present in $\infty$-categorical algebraic theories allows one to also use Malcev theories to naturally encode certain deformations based on \emph{filtered modules}. It is this identification that ultimately underlies the various descriptions of synthetic deformations in terms of filrations.

\begin{notation}
\label{notation:various_variants_on_free_r_module_theories}
Given a (possibly nonconnective) $\bfE_1$-ring spectrum $R$, we write 
\[
\lfrees_0(R),\lfrees(R),\lfrees_+(R),\lfrees_-(R)\subset\LMod_\pcat
\]
for the full subcategories generated under coproducts by $\Sigma^n R$ for $n = 0$, $n \in \integers$, $n \geq 0$, and $n \leq 0$ respectively. These are additive Malcev theories, sitting in a commutative diagram
\begin{center}\begin{tikzcd}
\lfrees_0(R)\ar[r]\ar[d]&\lfrees_+(R)\ar[d]\\
\lfrees_-(R)\ar[r]&\lfrees(R)
\end{tikzcd}\end{center}
of theories and fully faithful homomorphisms. 
\end{notation}

Note that $\lfrees(R), \lfrees_+(R)$ are loop theories, while $\lfrees_-(R)$ is a stable deloop theory in the sense of \cref{def:stabledelooptheory}. The associated $\infty$-categories of (de)loop models are easy to identify:

\begin{prop}\label{example:loop_models_and_deloop_models_for_variants_of_free_modules}
The restricted Yoneda embeddings from $\LMod_R$ provides equivalences
\[
\Model_{\lfrees(R)}\simeq \LMod_{R}\simeq\Model_{\lfrees_-(R)}^\Sigma.
\]
If $R$ is connective, then moreover
\[
\Model_{\lfrees_+(R)}^\Omega\simeq\LMod_R^{\geq 0}\simeq \Model_{\lfrees_0(R)}.
\]
\end{prop}
\begin{proof}
The first equivalences follow from \cref{prop:stableloopmodels}, whereas the second equivalences follow from \cref{prop:loopcompletion} and \cref{cor:cnmodules}.
\end{proof}

In general, if $\pcat$ is a loop theory whose $\infty$-category $\Model_\pcat^\Omega$ of loop models is some well-understood $\infty$-category, then the $\infty$-category $\Model_\pcat$ of all models may be more exotic, and admit no classical description. Instead, $\Model_\pcat$ can be thought of as an $\infty$-categorical deformation, whose generic fibre is given by the subcategory of loop models, but whose objects are ``formal resolutions'' of loop models by the objects of $\pcat\subset\Model_\pcat^\Omega$. In the special case where $\pcat = \lfrees(R)$, this deformation admits a more familiar explicit description in terms of filtered spectra. We refer the reader to  \cite[{\S 3}]{van2025introduction} for a pleasant introduction to the filtered perspective on deformations.

\begin{notation}
We write $\Sp^{\fil} \colonequals \Fun(\mathbb{Z}^{\op}, \Sp)$ for the $\infty$-category of filtered spectra, whose objects we identify with diagrams 
\[
\cdots\to X_{q+1} \to X_q \to X_{q-1}\to \cdots.
\]
This forms a symmetric monoidal stable $\infty$-category under Day convolution with respect to $+$. Any filtered spectrum has a canonical degree-shifting endomorphism $\tau$ corresponding to the map of diagrams 
\[
\begin{tikzcd}
	\ldots & {X_{q+1}} & {X_{q}} & {X_{q-1}} & \ldots \\
	\ldots & {X_{q}} & {X_{q-1}} & {X_{q-2}} & \ldots
	\arrow[from=1-1, to=1-2]
	\arrow[from=1-2, to=1-3]
	\arrow[from=1-2, to=2-2]
	\arrow[from=1-3, to=1-4]
	\arrow[from=1-3, to=2-3]
	\arrow[from=1-4, to=1-5]
	\arrow[from=1-4, to=2-4]
	\arrow[from=2-1, to=2-2]
	\arrow[from=2-2, to=2-3]
	\arrow[from=2-3, to=2-4]
	\arrow[from=2-4, to=2-5]
\end{tikzcd}
\]
\end{notation}

\begin{recollection}
We write $C(\tau^{n})$ for the cofibre of the endomorphism $\tau^{n}$ of the monoidal unit of filtered spectra. Explicitly, this is the filtered spectrum 
\[
\cdots \rightarrow 0 \rightarrow \thesphere \rightarrow \cdots  \rightarrow \thesphere \rightarrow 0 \rightarrow \cdots
\]
with $n$ copies of the sphere spectrum, ending in $\thesphere$ in degree zero. By \cite[{Proposition 3.2.5}]{lurie2015rotation}, $C(\tau^n)$ has a unique commutative algebra structure compatible with the canonical map from the unit. 
\end{recollection}

\begin{recollection}[{The diagonal $t$-structure}]
We will consider $\Sp^\fil$ as equipped with its \emph{diagonal $t$-structure}, whose full subcategory $\Sp^\fil_{\geq 0}\subset\Sp^\fil$ of connective objects consists of those filtered spectra $X$ for which $X_{q}$ is $q$-connective. This $t$-structure is compatible with the symmetric monoidal structure, and taking Whitehead towers defines a lax symmetric monoidal \cite[Proposition A.1.2]{balderrama2021deformations} and fully faithful \cite[{Theorem A.1}]{annala2025note} functor
\[
\tau_{\geq \star} \colon \Sp\to \Sp^\fil_{\geq 0}\subset\Sp^\fil,\qquad \tau_{\geq q}(X) \colonequals X_{\geq q}.
\]
\end{recollection}

We note that since $\tau_{\geq \star}$ is lax monoidal, if $R$ is a $\bfE_1$-ring spectrum then $\tau_{\geq \star} R \in \Alg_{\bfE_1}(\Sp^\fil_{\geq 0})$.

\begin{prop}
\label{prop:models_of_free_r_modules_as_modules_over_the_whitehead_tower}
Let $R$ be a $\bfE_1$-ring spectrum. 
\begin{enumerate}
\item There is an equivalence 
\[
w_! \colon \Model_{\lfrees(R)} \simeq \LMod_{\tau_{\geq \star} R}(\Sp^\fil_{\geq 0})
\]
sending $X \in \Model_{\lfrees(R)}$ to a $\tau_{\geq \star} R$-module with underlying filtered spectrum
\[
w_!(X) = \{\cdots \to \Sigma^{q+1}X(\Sigma^{q+1}R) \to \Sigma^q X(\Sigma^q R) \to \Sigma^{q-1}X(\Sigma^{q-1}R) \to \cdots\},
\]
where the arrows are adjoint by the canonical colimit-to-limit comparison map 
\[
X(\Sigma (\Sigma^{q} R)) \rightarrow \Omega X(\Sigma^{q} R). 
\]
\item For $1 \leq n < \infty$, there are equivalences
\[
w_{n!}\colon \Model_{\h_n\lfrees(R)}\simeq\LMod_{C(\tau^n)\otimes \tau_{\geq \star} R}(\Sp^\fil_{\geq 0})
\]
for which the following diagrams commute
\[
\begin{tikzcd}
\Model_{\lfrees(R)}\ar[r,"w_!","\simeq"']\ar[d,"\tau_{n!}"]&\LMod_{\tau_{\geq \star} R}^{\geq 0}(\Sp^\fil)\ar[d,"C(\tau^n)\otimes (\bs)"]\\
\Model_{\h_{n}\lfrees(R)}\ar[r,"w_{n!}","\simeq"']&\LMod_{C(\tau^{n}) \otimes \tau_{\geq \star}R}(\Sp^\fil_{\geq 0})
\end{tikzcd}. 
\]
\item The equivalence $w_{!}$ identifies the subcategory 
\[
 \Model_{\lfrees(R)}^\Omega \subseteq \Model_{\lfrees(R)},
\]
of loop models with the subcategory of those objects of $\LMod_{\tau_{\geq \star} R}^{\geq 0}(\Sp^\fil)$ which are $\tau$-periodic in the sense of \cite[{Definition 2.17}]{abstract_gh_theory}; that is, such that $\tau$ is a $1$-connective cover. 
\end{enumerate}
\end{prop}

\begin{proof}
(1)~~By \cref{prop:projprestable}, to give an equivalence $\Model_{\lfrees(R)}\simeq\LMod_{\tau_{\geq \star}R}(\Sp^\fil_{\geq 0})$ is to give a fully faithful functor
\[
\lfrees(R) \to \LMod_{\tau_{\geq \star}R}(\Sp^\fil_{\geq 0})
\]
whose essential image consists of a family of projective generators for $\LMod_{\tau_{\geq \star}R}(\Sp^\fil_{\geq 0})$. As $\tau_{\geq \star}$ is lax monoidal, the Whitehead tower admits a lift to a fully faithful functor
\[
\tau_{\geq \star}\colon \LMod_R\to\LMod_{\tau_{\geq \star}R}(\Sp^\fil_{\geq 0}).
\]
We claim that the desired functor is given by the restriction of this functor by $\lfrees(R)\subset\LMod_R$. As the Whitehead filtration commutes with direct sums, it suffices to show that the family $\{\tau_{\geq \star}(\Sigma^q R) : q \in \integers\}$ is a set of projective generators for $\LMod_{\tau_{\geq \star}R}(\Sp^\fil_{\geq 0})$. The filtered module $\tau_{\geq \star}(\Sigma^q R)$ is a shift and suspension of the monoidal unit, and satisfies
\[
\map_{\tau_{\geq\bullet R}}(\tau_{\geq \star}(\Sigma^q R),X)\simeq\Omega^{\infty+q}X_q.
\]
As $X$ is connective in the diagonal $t$-structure, $X_q$ is $q$-connective, and so this functor preserves geometric realization by \cite[{Proposition 1.4.3.9}]{higher_algebra}. These functors are clearly conservative as $q$ varies, establishing (1).

(2)~~By the proof of (1), we find that $\{C(\tau^n)\otimes\tau_{\geq \star}(\Sigma^q R) : q\in\integers\}$ form a set of projective generators for $\LMod_{C(\tau^n)\otimes\tau_{\geq \star}R}(\Sp^\fil_{\geq 0})$. These objects corepresent the functors
\[
X \mapsto \mathrm{fib}(\Omega^{\infty+q} X_{q} \rightarrow \Omega^{\infty+q} X_{q-n}). 
\]
Using this formula we see that the composite
\[
\begin{tikzcd}[column sep=large]
	{\lfrees(R)} & {\LMod_{\tau_{\geq \ast}R}(\Sp^\fil_{\geq 0})} & {\LMod_{C(\tau^{n}) \otimes \tau_{\geq \ast}R}(\Sp^\fil_{\geq 0})}
	\arrow["{\tau_{\geq \ast}}", from=1-1, to=1-2]
	\arrow["{C(\tau^{n})\otimes-}", from=1-2, to=1-3]
\end{tikzcd}
\]
factors uniquely through an equivalence from $\h_n\lfrees(R)$ to the full subcategory generated under coproducts by these projective generators. This yields the stated equivalence by \cref{prop:projprestable}, which fits in the given diagram by construction.

(3)~~Recall from \cref{prop:loopmodelifspecialfiberdiscrete} that $X \in \Model_{\lfrees(R)}$ is a loop model if and only if $\tau_! X$ is discrete, and from \cite[{Proposition 2.16}]{abstract_gh_theory} that $X \in \Model_{\tau_{\geq \star}R}(\Sp^\fil_{\geq 0})$ is $\tau$-periodic if and only if $C(\tau)\otimes X$ is discrete. These conditions agree under the equivalences of (1,2), yielding the stated identification.
\end{proof}

\begin{remark}
Using \cref{prop:models_of_free_r_modules_as_modules_over_the_whitehead_tower}, the $\infty$-categories $\Model_{\lfrees_\varepsilon(R)}$ with $\varepsilon \in \{-,0,+\}$ can be identified as the full subcategories of $\LMod_{\tau_{\geq \star}R}(\Sp^\fil_{\geq 0})$ generated under colimits by $\tau_{\geq \star}(\Sigma^n R)$ for, respectively, $n \leq 0$, $n=0$, and $n\geq 0$.
\end{remark}

\begin{remark}
As with \cref{rmk:manygeneralizations}, many variations of \cref{prop:models_of_free_r_modules_as_modules_over_the_whitehead_tower} are possible, providing deformations based on the Whitehead tower for other stable $\infty$-categories.
\end{remark}

When $R$ is Adams-type in the sense of \cite[{Definition 3.13}]{pstrkagowski2018synthetic}, we can also give a description of the $\infty$-category of models of $\lfrees(R)$ in terms of the $\infty$-category $\Syn_{R}$ of $R$-based synthetic spectra. This will be used in \S\ref{subsection:synthetic_spectra}, where we describe a completion of $\Syn_R$ in terms of $\infty$-categories of models.

\begin{proposition}
\label{proposition:synthetic_nur_modules_as_models}
Let $R$ be an Adams-type $\bfE_1$-ring spectrum. Then there is a unique equivalence
\[
\nu_!\colon \Model_{\lfrees(R)} \to \LMod_{\nu(R)}(\Syn^{\geq 0}_R)
\]
satisfying $\nu_!(P) = \nu(P)$ for $P \in \lfrees(R)\subset\Model_{\lfrees(R)}$.
\end{proposition}

\begin{proof}
By \cref{thm:freecocompletion}, the given description of $\nu_!$ on representables extends uniquely to a colimit-preserving functor $\nu_!\colon \Model_{\lfrees(R)}\to\LMod_{\nu(R)}(\Syn^{\geq 0}_R)$, which we must show is an equivalence. By construction, the $\infty$-category $\Syn_R^{\geq 0}$ of connective $R$-synthetic spectra is generated under colimits by objects of the form $\nu(F)$ where $F$ is a finite $R_\ast$-projective spectrum. It follows that connective $\nu(R)$-modules are generated under colimits by objects of the form
\[
\nu(R) \otimes \nu(F) \simeq \nu(R \otimes F),
\]
where we use the equivalence of \cite[{Lemma 4.24}]{pstrkagowski2018synthetic}. As $F$ is a finite $R_\ast$-projective spectrum, $\nu(R\otimes F)$ is a retract of a direct sum of copies of $\nu(\Sigma^q R)$ for $q \in \integers$. As $\nu$ is fully faithful \cite[{Corollary 4.38}]{pstrkagowski2018synthetic}, it follows that this set of generators is equivalent to the idempotent completion of the full subcategory $\lfrees^\omega(R)\subset\lfrees(R)$ generated by $\Sigma^q R$ for $q\in \integers$ under \emph{finite} direct sums.

By \cite[{Theorem 2.8, Lemma 4.23}]{pstrkagowski2018synthetic}, $\Syn_{R}^{{\geq 0}}$ is equivalent to a full subcategory of the $\infty$-category of presheaves of connective spectra on its full subcategory of objects of the form $\nu(F)$ for $F$ finite $R_\ast$-projective, spanned by those presheaves that send cofibre sequences to fibre sequences. Any such cofibre sequence is split after tensoring with $\nu(R)$, implying that $\LMod_{\nu(R)}(\Syn^{\geq 0}_R)$ is equivalent to the $\infty$-category of additive presheaves on its full subcategory generated under finite coproducts by $\nu(\Sigma^q R)$ for $q\in \integers$.

By the above, this full subcategory is equivalent to $\lfrees^\omega(R)$. Applying \cref{thm:bounded}, we find that
\[
\LMod_{\nu(R)}(\Syn^{{\geq 0}}_{R}) \simeq \presheaves_{\Sigma}(\lfrees(R)^{\omega}) \simeq \Model_{\lfrees(R)}. 
\]
Tracing through the definitions, we see that this chain of equivalences agrees with $\nu$ on representables, and must therefore coincide with the claimed equivalence $\nu_!$.
\end{proof}

\begin{remark}
\label{remark:synthetic_analogue_functor_corresponds_to_inclusion_of_loop_models_under_equivalence_between_nur_modules_and_models_of_lfreesr}
Both the synthetic analogue $\nu \colon \LMod_{R} \rightarrow \LMod_{\nu(R)}(\Syn_{R}^{{\geq 0}})$ and the functor $\LMod_{R} \rightarrow \Model_{\lfrees(R)}$ inducing the equivalence $\LMod_{R} \simeq \Model_{\lfrees(R)}$ of \cref{prop:stableloopmodels} are defined using the restricted Yoneda embedding. It follows that under \cref{proposition:synthetic_nur_modules_as_models}, the two uses of $\nu$, either as the synthetic analogue functor or the inclusion of the subcategory of loop models, are compatible with each other. 
\end{remark}

\begin{remark}
Combining \cref{prop:models_of_free_r_modules_as_modules_over_the_whitehead_tower} and \cref{proposition:synthetic_nur_modules_as_models} gives a description of $\LMod_{\nu(R)}(\Syn_{R})$ in terms of filtered spectra, recovering a variation on \cite[{Remark 6.13}]{gheorghe2022c} and \cite[{\S 3.5}]{pstrkagowski2025perfect}.
\end{remark}

\subsection{Synthetic \texorpdfstring{$R$}{R}-algebras}
\label{subsection:synthetic_r_algebras}

In \cref{prop:models_of_free_r_modules_as_modules_over_the_whitehead_tower}, we saw that if $R$ is an $\bfE_1$-ring spectrum, then the $\infty$-category of models for the theory $\lfrees(R) \subseteq \Mod_{R}(\spectra)$ generated under direct sums by $\Sigma^{q} R$ for $q \in \integers$ is equivalent to the $\infty$-category of connective filtered modules over the filtered spectrum $\tau_{\geq \star}R$. The advantage of working with models is that it provides a clean universal property for this $\infty$-category, and that it naturally generalizes to nonadditive settings, such as for algebras over ring spectrum.

\begin{example}\label{ex:freeoalgebras}
Let $R$ be (for simplicity) an $\bfE_\infty$-ring spectrum and $\calO$ be an $\infty$-operad. We write 
\[
\lfrees^\calO(R) \subset\Alg_\calO(\Mod_R)
\]
for the full subcategory spanned by the essential image of the free $\calO$-algebra functor
\[
\Free_\calO\colon \lfrees(R) \subset\Mod_R \to \Alg_\calO(\Mod_R).
\]
Then $\lfrees^\calO(R)$ is a loop theory whose $\infty$-category of modules acts a deformation with generic fiber the $\infty$-category of $\calO$-algebras over $R$, deformed along the homotopy groups functor. Indeed:

\begin{enumerate}
    \item By \cref{proposition:recognizing_algebras_in_loop_models_as_loop_models_of_algebras} and \cref{example:loop_models_and_deloop_models_for_variants_of_free_modules}, the restricted Yoneda embedding determines an equivalence
    \[
    \Alg_\calO(\Mod_R)\simeq\Model_{\lfrees^\calO(R)}^\Omega
    \]
between $\Alg_\calO(\Mod_R)$ and the $\infty$-category of loop models of $\lfrees^\calO(R)$, thought of as the generic fiber of this deformation.

\item The subcategory $\Model_{\lfrees^\calO(R)}^\heartsuit$ of discrete models is the ``best possible'' target for the homotopy groups of an $R$-$\calO$-algebra; its objects can be identified with graded sets equipped with an action by the natural transformations between products of the homotopy group functors
\[
\pi_q\colon \Alg_\calO(\Mod_R)\to\sets.
\]
This follows from \cref{remark:for_a_theory_discrete_models_are_the_same_as_for_its_homotopy_category} and the fact that, by the Yoneda lemma, such natural transformations are in bijective correspondence with sets of maps in the homotopy category $\h\lfrees^\calO(R)$. \cref{thm:animation} now provides an equivalence
\[
\Model_{\h \lfrees^\calO(R)} \simeq \dcat_{\geq 0}(\acat) \simeq \Fun(\Delta^{\op}, \Model_{\lfrees^\calO(R)}^\heartsuit)[\mathrm{W}^{-1}],
\]
realizing $\Model_{\h\lfrees^\calO(R)}$, thought of as the special fiber of the deformation, as the classical nonabelian derived $\infty$-category, or animation, of $\Model_{\lfrees^\calO(R)}^\heartsuit$.
\end{enumerate} 
\end{example}

We refer the reader to \cite{balderrama2023algebraic} for further exposition and applications of these ideas.

\begin{remark}
By \cref{prop:restrictionmonadic}, the adjunction 
\[
\Model_{\lfrees(R)} \rightleftarrows \Model_{\lfrees^{\calO}(R)} 
\]
induced by the homomorphism $\Free_{\calO} \colon \lfrees(R) \rightarrow \lfrees^{\calO}(R)$ of Malcev theories is monadic. By unwrapping the definitions, one sees that under the equivalence 
\[
\Model_{\lfrees(R)} \simeq \Mod_{\tau_{\geq \star} R}(\Sp^{\fil}_{\geq 0}) 
\]
of \cref{prop:models_of_free_r_modules_as_modules_over_the_whitehead_tower}, the corresponding monad on $\Model_{\tau_{\geq \star}R}(\Sp^\fil_{\geq 0})$ is the unique geometric realization-preserving monad satisfying
\[
\tau_{\geq \star} (M) \mapsto \tau_{\geq \star} (\Free_\calO(M))
\]
for $M \in \lfrees(R) \subseteq \Mod_{R}$. This is one advantege of working with models of Malcev theories, even if one is ultimately interested in invariants valued in filtered spectra: the universal property of the $\infty$-category of models allows for a very efficient construction of monads and more general functors.
\end{remark}

\subsection{Synthetic spectra} 
\label{subsection:synthetic_spectra}

Let $R$ be an Adams-type $\mathbf{E}_{1}$-ring spectrum. In 
\cref{proposition:synthetic_nur_modules_as_models}, we gave a description of $\nu(R)$-modules in (connective) synthetic spectra as models of a Malcev theory $\lfrees(R)$. One might ask whether such a description is possible for the $\infty$-category of synthetic spectra itself, rather than for $\nu(R)$-modules. Except in very simple cases, such as $R = \thesphere$, such a direct description is not possible\footnote{To see this, observe that if $\Model_\pcat$ is additve, then its full subcategory of discrete objects must be an abelian category with enough projectives. On the other hand, by \cite[{Proposition 4.16}]{pstrkagowski2018synthetic} the heart of $\Syn_R$ can be identified with the abelian category of $R_{*}R$-comodules, which in general does not have enough projectives.}. In this section, we instead describe how: 
\begin{enumerate}
\item Up to Postnikov completion, synthetic spectra can be identified with the $\infty$-category of \emph{coalgebras} for a suitable comonad on the $\infty$-category $\Model_{\lfrees(R)}$, and
\item This comonad can be produced independently of the usual description of $\Syn_{R}$, using the machinery of \S\ref{subsection:coalgebras_in_models_and_loop_models}. 
\end{enumerate}
Taken together, this gives a new construction of the Postnikov completion of the $\infty$-category of synthetic spectra, see \cref{theorem:nilpotent_complete_synthetic_spectra_as_coalgebras_for_a_derived_functor}. The advantage of the current approach is that it immediately generalizes to other contexts, such as to construct new $\infty$-categories of synthetic spaces and synthetic $\calO$-algebras, as we explain in \S\ref{subsection:examples_synthetic_spaces} and \S\ref{ssec:examples:syntheticekrings}. 

\begin{recollection}
Let $\acat$ be a prestable $\infty$-category with finite limits; that is, an $\infty$-category which can be written in the form $\ccat_{\geq 0}$ where $(\ccat_{\geq 0}, \ccat_{\leq 0})$ is a $t$-structure on a stable $\infty$-category \cite[{Proposition C.1.2.9}]{lurie_spectral_algebraic_geometry}. The \emph{Postnikov completion} of $\acat$ is
\[
\acat_{\cpl} \colonequals \lim_{n\to\infty} \acat_{\leq n}, 
\]
the limit of its subcategories of $n$-truncated objects along the truncation functors. This is again a prestable $\infty$-category, and it can be identified with the subcategory $\widehat{\ccat}_{\geq 0}$ of connective objects of the left completion $\widehat{\ccat}$ of \cite[{Proposition 1.2.1.17}]{higher_algebra}. 

The truncation functors $\acat \rightarrow \acat_{\leq n}$ induce a canonical functor $\acat \rightarrow \acat_{\cpl}$ which is exact and restricts to an equivalence on the subcategories of bounded objects. This functor is initial with respect to exact functors into Postnikov complete prestable $\infty$-categories\footnote{This follows formally from the fact that $(\acat_{\cpl})_{\cpl} \simeq \acat_{\cpl}$. Beware that this does not mean that $\acat \rightarrow \acat_{\cpl}$ is a localization of $\infty$-categories, and indeed this is usually not the case.}. 
\end{recollection}

Our main result in this section, \cref{theorem:nilpotent_complete_synthetic_spectra_as_coalgebras_for_a_derived_functor}, gives a coalgebraic description of $(\Syn_{R}^{\geq 0})_{\cpl}$, the Postnikov completion of the $\infty$-category of connective $R$-based synthetic spectra. 

\begin{remark}
Any Postnikov complete prestable $\infty$-category is in particular \emph{separated}; that is, the $t$-structure homotopy groups detect equivalences. In particular, the canonical functor from $\Syn_R^{\geq 0}$ to its Postnikov completion factors uniquely through a conservative functor
\[
\widehat{\Syn}{}_{R}^{{\geq 0}} \rightarrow (\Syn_{R}^{{\geq 0}})_{\cpl}
\]
from the $\infty$-category of hypercomplete synthetic spectra of \cite[{\S 5.1}]{pstrkagowski2018synthetic}. Beware that even this induced functor is not in general an equivalence; this is analogous to the difference between $E$-local and $E$-nilpotent complete spectra. 

A notable case where these $\infty$-categories agree is when $R$ is a Morava $E$-theory spectrum, as proven in \cite[{Theorem 7.4}]{abstract_gh_theory}. Thus, in this case the results of this section give a description of hypercomplete synthetic spectra. 
\end{remark}

\begin{construction}
Consider the free and forgetful adjunction
\[
\begin{tikzcd}
	{\Syn^{{\geq 0}}_R} && {\Mod_{\nu(R)}(\Syn^{{\geq 0}}_{R})}
	\arrow["{\nu(R) \otimes_{\nu(\thesphere)} \bs}", from=1-1, to=1-3, shift left=1]
	\arrow[shift left=1, from=1-3, to=1-1]
\end{tikzcd}.
\]
This gives $\nu(R) \otimes_{\nu(S)} \bs$, considered as an endofunctor of $\nu(R)$-modules, the structure of a comonad. Moreover, the adjunction yields a canonical comparison functor 
\begin{equation}
\label{equation:comparison_functor_from_synthetic_spectra_into_coalgebras_for_nur_tensor_comonad}
\phi \colon \Syn_{R}^{{\geq 0}} \rightarrow \CoAlg_{\nu(R) \otimes_{\nu(\thesphere)} \bs}(\Mod_{\nu(R)}(\Syn_{R}^{{\geq 0}}))
\end{equation}
into the associated $\infty$-category of coalgebras.
\end{construction}

\begin{lemma}
\label{lemma:functor_out_of_synthetic_spectra_into_coalgebras_is_postnikov_completion}
The target of (\ref{equation:comparison_functor_from_synthetic_spectra_into_coalgebras_for_nur_tensor_comonad}) is a Postnikov complete prestable $\infty$-category, $\phi$ is exact, and the induced functor 
\[
\widehat{\phi} \colon (\Syn_{R}^{{\geq 0}})_{\cpl} \rightarrow \CoAlg_{\nu(R) \otimes_{\nu(\thesphere)} \bs}(\Mod_{\nu(R)}(\Syn_{R}^{{\geq 0}}))
\]
from the Postnikov completion of $\Syn_R^{\geq 0}$ is an equivalence. 
\end{lemma}

\begin{proof}
We first claim that the right exact functor $\nu(R) \otimes_{\nu(\thesphere)} -$ is also left exact; note that there is something to check here, as we are working in the prestable context of \emph{connective} synthetic spectra. Indeed, if $F$ is a finite $E$-projective spectrum as in \cite[{Definition 3.13}]{pstrkagowski2018synthetic}, then $\nu(F) \otimes_{\nu(\thesphere)} \bs$ has both a left and right adjoint, given by tensoring with the Spanier-Whitehead dual of $F$ \cite[{Lemma 3.21}]{pstrkagowski2018synthetic}, and so is exact. Since $R$ is Adams-type, we can write $R \simeq \colim R_{\alpha}$ as a filtered colimit of finite $R$-projective spectra, which shows that 
\[
\nu(R) \otimes_{\nu(\thesphere)} (\bs) \simeq \colim \nu(R_{\alpha}) \otimes_{\nu(\thesphere)} (\bs)
\]
is exact, as filtered colimits are exact in connective synthetic spectra \cite[{Proposition 4.16}]{pstrkagowski2018synthetic}. 

As $\nu(R)\otimes_{\nu(\thesphere)}\bs$ is exact, it follows that the forgetful functor
\[
 \CoAlg_{\nu(R) \otimes_{\nu(\thesphere)} \bs}(\Mod_{\nu(R)}(\Syn_{R}^{{\geq 0}})) \rightarrow \Mod_{\nu(R)}(\Syn_{R}^{{\geq 0}})
\]
is also exact. As this functor is conservative and the target $\Model_{\nu(R)}(\Syn_R^{\geq 0})\simeq\Model_{\lfrees(R)}$ is Postnikov complete, it follows that $\CoAlg_{\nu(R) \otimes_{\nu(\thesphere)} \bs}(\Mod_{\nu(R)}(\Syn_{R}^{{\geq 0}}))$ is Postnikov complete.

The comparison functor $\widehat{\phi}$ is an exact functor between Postnikov complete prestable $\infty$-categories, and so to show that it is an equivalence it suffices to show that it is an equivalence between the hearts. Indeed, it then follows using exactness by induction that it induces equivalences between subcategories of $n$-truncated objects for all $n \geq 0$, and completeness allows us to pass to the limit. We have 
\[
\Mod_{\nu(R)}(\Syn_{R}^{\heartsuit}) \simeq \Model_{\lfrees(R)}^{\heartsuit} \simeq \Mod_{R_{*}}^\heartsuit,
\]
and the comonad $\nu(R)\otimes_{\nu(\thesphere)}\bs$ restricts to $R_\ast R \otimes_{R_\ast}\bs$ on discrete objects. Thus the result follows from the equivalence $\Syn_{R}^{\heartsuit} \simeq \CoAlg_{R_{*}R\otimes_{R_\ast}\bs}(\Mod_{R_{*}}^\heartsuit)$ of \cite[{Proposition 4.16}]{pstrkagowski2018synthetic}. 
\end{proof}

\begin{notation}
\label{notation:comonad_defining_synthetic_spectra}
Under the equivalence $\Mod_{\nu(R)}(\Syn_{R}^{{\geq 0}}) \simeq \Model_{\lfrees(R)}$ of \cref{proposition:synthetic_nur_modules_as_models}, $\nu(R) \otimes_{\nu(\thesphere)} -$ corresponds to a comonad on $\Model_{\lfrees(R)}$, which we will temporarily denote by
\[
c_{!} \colon \Model_{\lfrees(R)} \rightarrow \Model_{\lfrees(R)}.
\]
Note that we have equivalences 
\[
\CoAlg_{c_{!}}(\Model_{\lfrees(R)}) \simeq \CoAlg_{\nu(R) \otimes_{\nu(\thesphere)} -}(\Mod_{\nu(R)}(\Syn_{R}^{{\geq 0}})) \simeq (\Syn_{R}^{{\geq 0}})_{\cpl}, 
\]
the first being the definition of $c_{!}$ and the second being \cref{lemma:functor_out_of_synthetic_spectra_into_coalgebras_is_postnikov_completion}. 
\end{notation}

The rest of this section is devoted to showing that $c_{!}$ can be constructed independently of synthetic spectra, derived from the corresponding non-synthetic comonad on the $\infty$-category of $R$-modules. This yields the desired coalgebraic description of Postnikov-complete synthetic spectra, which we prove in \cref{theorem:nilpotent_complete_synthetic_spectra_as_coalgebras_for_a_derived_functor}. 

\begin{notation}
By \cref{prop:stableloopmodels}, the restricted Yoneda embedding $\LMod_{R} \rightarrow \Model_{\lfrees(R)}$ induces an equivalence $\LMod_{R} \simeq \Model_{\lfrees(R)}^{\Omega}$. In what follows, we will implicitly identify these $\infty$-categories, and write $\nu$ for the inclusion of either into $\Model_{\lfrees(R)}$.

Note that under the equivalence $\Mod_{\nu(R)}(\Syn_{R}^{{\geq 0}}) \simeq \Model_{\lfrees(R)}$, the synthetic analogue functor corresponds to the inclusion of loop models, see \cref{remark:synthetic_analogue_functor_corresponds_to_inclusion_of_loop_models_under_equivalence_between_nur_modules_and_models_of_lfreesr}, so our use of $\nu$ in both contexts is consistent. 
\end{notation}

\begin{lemma}
The underlying functor of the comonad $c_{!}$:
\begin{enumerate}
    \item Preserves geometric realizations;
    \item Is given by $\nu(P) \mapsto \nu(R \otimes_{\thesphere} P)$ for $P \in \lfrees(R)$.
\end{enumerate}
Moreover, it is uniquely determined by these two properties. 
\end{lemma}

\begin{proof}
These properties are a restatement of the corresponding properties of the endofunctor $\nu(R)\otimes_{\nu(\thesphere)}\bs$ of $\Model_{\nu(R)}(\Syn_R^{\geq 0})$. They uniquely determine it by the universal property of $\Model_{\lfrees(R)}$ as a cocompletion given in \cref{thm:freecocompletion}.
\end{proof}

As $R \otimes_{\thesphere}M\simeq (R\otimes_\thesphere R)\otimes_R M$ for any $M \in \LMod_R$, the underlying functor of the comonad $c_!$ can be thought of the derived functor of tensoring with the $R$-bimodule $R\otimes_\thesphere R$. The following describes some basic properties of such functors in general.

\begin{lemma}
\label{lemma:properties_of_the_functor_between_models_of_freesr_induced_by_tensoring_with_bimodule}
Let $R$ and $S$ be $\mathbf{E}_{1}$-rings and $M$ be an $S$-$R$-bimodule. Consider the derived functor
\[
D(M\otimes_R \bs)\colon \Model_{\lfrees(R)}\to\Model_{\lfrees(S)}
\]
of tensoring with $M$, as in \cref{definition:induced_derived_functor_from_a_functor_of_loop_models_and_vice_versa}. Then 
\begin{enumerate}
    \item $D(M\otimes_R\bs)$ takes $L_{\lfrees(R)}$-equivalences to $L_{\frees(S)}$-equivalences;
    \item The counit $E(D(M \otimes_{R} -)) \rightarrow M \otimes_{R} -$ is an equivalence;
    \item $D(M\otimes_R\bs)$ preserves loop models provided $\pi_{*}(M)$ is flat as a right $\pi_{*}(R)$-module.;
    \item If $R = S$, then the endofunctor $M\otimes_R\bs$ of $\LMod_R$ is \derivable{} in the sense of \cref{definition:self_tame_endofunctor} if and only if $\Tor_{R_\ast}^k(M_\ast,M_\ast) = 0$ for $k > 0$.
\end{enumerate}
\end{lemma}

\begin{proof}
(1)~~This follows from \cref{lemma:induced_derived_functor_of_a_cocontinuous_functor_preserves_l_equivalences}  as $M\otimes_R\bs$ preserves colimits.

(2)~~This follows from \cref{lemma:a_criterion_for_a_functor_between_loop_models_is_determined_by_its_derived_functor} as $M\otimes_R\bs$ preserves colimits.

(3)~~By \cref{prop:loopmodelifspecialfiberdiscrete}, $D(M\otimes_R\bs)$ preserves loop models if and only if $\tau_! D(M\otimes R\bs)(\nu N)$ is discrete for any left $R$-module $N$. Under the identification of $\Model_{\h\lfrees(R)}$ and $\Model_{\h\lfrees(S)}$ with the connective derived $\infty$-categories of graded $R_\ast$ and $S_\ast$-modules, we can identify
\[
\tau_! D(M\otimes_R\bs)(\nu N)\simeq M_\ast \otimes_{R_\ast}N_\ast
\]
as the \emph{derived} tensor product with $M_\ast$. This is discrete provided $\Tor_{R_\ast}^f(M_\ast,N_\ast) = 0$ for $f\geq 1$, which is automatic if $M_\ast$ is flat as a right $R_\ast$-module.

(4)~~By definition, $M\otimes_R\bs$ is \derivable{} if and only if the natural transformation
\[
D(M\otimes_R\bs)\circ D(M\otimes_R\bs) \to D(M\otimes_R M \otimes_R\bs)
\]
of derived functors is an equivalence. As both sides are compatible with sums and suspensions, this natural transformation is an equivalence if and only if it evaluates to an equivalence on $\nu R$. As with (3), applying $\tau_!$ to this natural transformation and evaluating on $\nu R$ provides the map
\[
M_\ast\otimes_{R_\ast}M_\ast \to \Tor_{R_\ast}^0(M_\ast,M_\ast)
\]
in $\LMod_{R_\ast}^{\geq 0}$ which projects the derived tensor product $M_\ast\otimes_{R_\ast}M_\ast$ onto its $0$-truncation. Clearly this is an equivalence if $\Tor_{R_\ast}^k(M_\ast,M_\ast) = 0$ for $k > 0$. Conversely, if this is an equivalence then \cref{prop:loopmodelifspecialfiberdiscrete} implies that
\[
D(M\otimes_R\bs)\circ D(M\otimes_R\bs)(\nu R) \to D(M\otimes_R M \otimes_R\bs)(\nu R)
\]
is a map between loop models which induces an isomorphism on $\pi_0$, and which is therefore itself an equivalence.
\end{proof}

\begin{proposition}
\label{proposition:comonad_defining_synthetic_spectra_comes_is_a_derived_functor_as_a_comonad}
Suppose that $R$ is Adams-type. Then 
\begin{enumerate}
    \item The endofunctor of $R$-modules given by $R \otimes_{\thesphere} \bs$ is \derivable{} in the sense of \cref{definition:self_tame_endofunctor}, so that $D(R\otimes_\thesphere\bs)$ inherits the structure of a comonad;
    \item There is a canonical equivalence of comonads $D(R\otimes_\thesphere\bs)\simeq c_!$.
\end{enumerate} 
\end{proposition}

\begin{proof}
The first claim follows from \cref{lemma:properties_of_the_functor_between_models_of_freesr_induced_by_tensoring_with_bimodule}.(4), as if $R$ is Adams type then $R_\ast R$ is flat over $R_\ast$. As $D(R\otimes_\thesphere\bs)\simeq c_!\simeq D(E(c_!))$ as functors, to give an equivalence as comonads it therefore suffices by 
\cref{lemma:endofunctors_whose_derived_functors_preserve_loop_models_are_which_can_be_recovered_from_them_form_a_monoidal_subcategory_on_which_d_is_monoidal}
to prove that $R\otimes_\thesphere\bs\simeq E(c_!)$ as comonads on $\LMod_R$. Under the equivalence $\Mod_{\nu(R)}(\Syn_R^{\geq 0})\simeq\Model_{\lfrees(R)}$, the localization $L\colon \Model_{\lfrees(R)}\to\Model_{\lfrees(R)}^\Omega$ can be identified with inverting $\tau$ in the sense of \cite[{\S 4.4}]{pstrkagowski2018synthetic}. Thus $E(c_!)$ is the comonad associated to the adjunction obtained from
\[
\Syn_{R}^{{\geq 0}} \rightleftarrows \Mod_{\nu(R)}(\Syn_{R}^{{\geq 0}})
\]
by inverting $\tau$ on both sides, which is exactly the comonad $R \otimes_\thesphere\bs$ as needed.
\end{proof}

Putting everything together yields the main theorem of this section.

\begin{theorem}
\label{theorem:nilpotent_complete_synthetic_spectra_as_coalgebras_for_a_derived_functor}
There is an equivalence of $\infty$-categories
\[
(\Syn_{R}^{{\geq 0}})_{\cpl} \simeq \CoAlg_{D(R \otimes_{\thesphere} \bs)}(\Model_{\lfrees(R)})
\]
between the Postnikov completion of $\Syn_{R}^{{\geq 0}}$ and coalgebras for the derived comonad $D(R \otimes_{\thesphere} -)$. 
\end{theorem}
\begin{proof}
By \cref{lemma:functor_out_of_synthetic_spectra_into_coalgebras_is_postnikov_completion}, the left hand side can be identified with coalgebras for the comonad $\nu(R) \otimes_{\nu(\thesphere)} \bs$, which in turn can be identified with coalgebras for $c_{!}$ of \cref{notation:comonad_defining_synthetic_spectra}. The result then follows from \cref{proposition:comonad_defining_synthetic_spectra_comes_is_a_derived_functor_as_a_comonad}. 
\end{proof}

\begin{rmk}\label{rmk:twoconstructionsofsyntheticspectra}
Let $R$ be a homotopy ring spectrum for which $\Tor_{R_\ast}^{>0}(R_\ast R,R_\ast R) = 0$. In this case, there is a good $\infty$-category $\Syn_{R,\cpl}^{\geq 0}$ of Postnikov-complete $R$-synthetic spectra in either of the following two cases:
\begin{enumerate}
\item $R$ is Adams-type, in which case we may define $\Syn_{R,\cpl}^{\geq 0} \colonequals (\Syn_R^{\geq 0})_\cpl$ as the Postnikov completion of the connective $R$-synthetic spectra of \cite{pstrkagowski2018synthetic};
\item $R$ is equipped with an $\bfE_1$-ring structure, in which case we may define 
\[
\Syn_{R,\cpl}^{\geq 0} \colonequals \CoAlg_{D(R \otimes_{\thesphere} \bs)}(\Model_{\lfrees(R)})
\]
\end{enumerate}
By \cref{theorem:nilpotent_complete_synthetic_spectra_as_coalgebras_for_a_derived_functor}, these two $\infty$-categories are equivalent when both assumptions are satisfied. It is interesting to observe that these are the same assumptions that enable one to construct a K\"unneth spectral sequence for $R$-homology \cite{adams1969lectures,elmendorf2007rings}. It is an interesting question whether there is some weaker common assumption on $R$ that allows one to construct a good $\infty$-category of $R$-synthetic spectra.
\end{rmk}

\subsection{Synthetic spaces} 
\label{subsection:examples_synthetic_spaces}

Given an $\bfE_1$-ring spectrum $R$, the $R$-Adams spectral sequence is equivalent to the descent spectral sequence associated to the adjunction $\Sp \rightleftarrows \LMod_{R}$ between spectra and $R$-modules. \cref{theorem:nilpotent_complete_synthetic_spectra_as_coalgebras_for_a_derived_functor} categorifies this description, showing that the Postnikov completion of $R$-synthetic spectra is equvalent to the $\infty$-category of coalgebras for the derived comonad associated to this adjunction.

The \emph{unstable} $R$-based Adams spectral sequence of Bendersky--Curtis--Miller \cite{bendersky1978unstable} and Bendersky--Thompson \cite{benderskythompson2000bousfield} is again the descent spectral sequence for an adjunction, but now comes from the adjunction
\begin{equation}
\label{equation:adjunction_giving_rise_to_synthetic_spaces}
R \otimes \Sigma^{\infty}_{+} \dashv \Omega^{\infty}: \spaces \rightleftarrows \Mod_{R}(\Sp)
\end{equation}
between \emph{spaces} and $R$-modules. Combined with our machinery, this naturally leads to a categorical deformation encoding the unstable $R$-Adams spectral sequence, as we explain this section.

\begin{definition}
\label{definition:ring_spectrum_with_r_free_loop_spaces}
Let $R$ be an $\mathbf{E}_{1}$-ring spectrum. We say  that $R$ has \emph{projective loop spaces} if $R_\ast (\Omega^\infty \Sigma^n R)$ is projective as a left $R_\ast$-module for all $n \in \integers$.
\end{definition}

\begin{example}
If $R_{*}$ is a field, then every spectrum has projective (in fact, free) $R$-homology, and thus $R$ has free loop spaces.
\end{example}

\begin{example}
Both $\mathrm{MU}$ and $\mathrm{BP}$ have projective (in fact, free) loop spaces, by a fundamental result of Wilson \cite{wilson1973omega, wilson1975omega}. 
\end{example}

\begin{rmk}
Any spectrum $M$ satisfies
\[
M\simeq\colim_{n\to\infty}\Sigma^{-n}\Sigma^\infty\Omega^\infty\Sigma^n M.
\]
It follows that if $R$ has projective loop spaces, then $R_\ast R$ is a filtered colimit of projective $R_\ast$-modules, and is therefore flat as a left $R_\ast$-module.
\end{rmk}

Recall that $\calL(R)\subset\LMod_R$ is the full subcategory generated under coproducts by $\Sigma^n R$ for $n \in\integers$. Its idempotent completion $\calL^\sharp(R)\subset\LMod_R$ may be identified as the full subcategory of projective modules. We will implicitly identify the loop models of $\calL^\sharp(R)$ with $\LMod_R$ via the restricted Yoneda embedding.

\begin{prop}\label{prop:projloopspaces}
Let $R$ be an $\bfE_1$-ring spectrum with projective loop spaces. Then the full subcategory $\calL^\sharp(R)\subset\LMod_R$ is closed under the endofunctor $R \otimes \Sigma^\infty_+\Omega^\infty(\bs)$.
\end{prop}
\begin{proof}
It suffices to show that the class of spectra $M$ for which $R \otimes\Sigma^\infty\Omega^\infty M$ is projective as a left $R$-module is closed under retracts and direct sums. It is clearly closed under retracts, so we must only contend with direct sums.

In general, the functor $\Sigma^\infty_+\Omega^\infty$ preserves filtered colimits and satisfies
\begin{align*}
\Sigma^\infty_+\Omega^\infty(M\times N)&\simeq \Sigma^\infty_+\Omega^\infty M \otimes \Sigma^\infty_+\Omega^\infty N\\
&\simeq \thesphere \oplus \Sigma^\infty \Omega^\infty M\oplus \Sigma^\infty\Omega^\infty N \oplus \Sigma^\infty\Omega^\infty M \otimes \Sigma^\infty\Omega^\infty N
\end{align*}
for any two spectra $M$ and $N$. By iterating this, we find that if $\{M_i : i \in I\}$ is any set of spectra then
\[
\Sigma^\infty_+\Omega^\infty(\bigoplus_{i\in I}M_i) \simeq \bigoplus_{\substack{F\subset I\\\text{finite}}} \bigotimes_{i\in F}\Sigma^\infty\Omega^\infty M_i.
\]
As the class of spectra with $R$-projective homology is closed under direct sums and smash products, it follows that the class of spectra $M$ for which $\Sigma^\infty_+\Omega^\infty M$ has $R$-projective homology is closed under direct sums.
\end{proof}

We assume for the rest of this section that $R$ is an $\bfE_1$-ring spectrum with projective loop spaces.

\begin{notation}
We write $c \colonequals R \otimes \Sigma^{\infty}_{+} \Omega^{\infty}$ for the comonad on $\LMod_{R}\simeq\Model_{\lfrees^\sharp(R)}^\Omega$ induced by the adjunction (\ref{equation:adjunction_giving_rise_to_synthetic_spaces}).
\end{notation}

\begin{rmk}\label{rmk:genericfibersynspaces}
By construction an object of $\CoAlg_c(\LMod_R)$ consists of an $R$-module equipped with descent data for the adjunction $\spaces\rightleftarrows\LMod_R$. In particular, there is a lift in 
\begin{center}\begin{tikzcd}
&\CoAlg_c(\LMod_R)\ar[d]\\
\spaces\ar[r,"R\otimes\Sigma^\infty_+(\bs)"']\ar[ur,dashed,"{R[\bs]}",dashed]&\LMod_R
\end{tikzcd}\end{center}
with the property that if $X$ and $Y$ are spaces then
\begin{align*}
\map_{c}(R[Y],R[X])&\simeq \Tot \map_c(R[Y],(R\otimes\Sigma^\infty_+\Omega^\infty\bs)^{\bullet+1}R[X])\\
&\simeq\Tot\map_{\spaces}(Y,(\Omega^\infty(R\otimes\Sigma^\infty_+\bs))^{\bullet+1}X)\simeq
\map_{\spaces}(Y,X_R^\wedge),
\end{align*}
where $X_R^\wedge$ is the unstable $R$-nilpotent completion of $X$.
\end{rmk}

\begin{proposition}
\label{proposition:comonad_defining_synthetic_spaces_is_self_tame}
The underlying endofunctor $c$ has the following properties: 
\begin{enumerate}
\item $c$ is \derivable{} in the sense of \cref{definition:self_tame_endofunctor}; that is, $D(c) \circ D(c) \simeq D(c^{2})$,
\item The counit natural transformation $E(D(c)) \rightarrow c$ is an equivalence. 
\end{enumerate}
\end{proposition}

\begin{proof}
The first claim follows from \cref{prop:projloopspaces} and  \cref{remark:d_is_monoidal_if_the_first_functor_preserves_representables}. For the second, we apply the criterion of \cref{lemma:a_criterion_for_a_functor_between_loop_models_is_determined_by_its_derived_functor}. 

Suppose that $M \simeq | M_{\bullet}|$ is a geometric realization in the $\infty$-category of $R$-modules which is preserved by the restricted Yoneda embedding $\LMod_{R} \hookrightarrow \Model_{\lfrees(R)}$. As $\Omega^{\infty}$ is corepresentable in $R$-modules by $R \in \lfrees(R)$, this in particular means that $\Omega^{\infty} M \simeq | \Omega^{\infty} M_{\bullet}|$ as spaces. It follows that 
\[
R \otimes \Sigma_{+}^{\infty} \Omega^{\infty} M \simeq | R \otimes \Sigma_{+}^{\infty} \Omega^{\infty} M_{\bullet} |, 
\]
which verifies the criterion. 
\end{proof}

\begin{definition}
\label{def:rsyntheticspaes}
An \emph{$R$-synthetic space} is a $D(c)$-coalgebra in $\Model_{\lfrees^\sharp(R)}$. We write 
\[
\Syn\spaces_{R} \colonequals \CoAlg_{D(c)}(\Model_{\lfrees^\sharp(R)})
\]
for the $\infty$-category of $R$-synthetic spaces. 
\end{definition}

We now explain how $\Syn\spaces_{R}$ can be thought of as a deformation. 

\begin{ex}[The generic fiber]
By \cref{proposition:e_and_l_take_coalgebras_to_coalgebras}, there is a generic fibre functor
\[
L\colon \Syn\spaces_R \to \CoAlg_c(\LMod_R),
\]
where $\CoAlg_c(\LMod_R)$ is the $\infty$-category discussed in \cref{rmk:genericfibersynspaces}. Postcomposing with $\map_c(R[\ast],\bs)$ then provides a realization functor
\[
\Syn\spaces_R \to \spaces.
\]
In the other direction, if $\spaces^{R{\hbox{{-}}}\text{proj}}\subset\spaces$ is the full subcategory of spaces with projective $R$-homology, then \cref{proposition:induced_derived_functor_construction_gives_a_functor_on_coalgebras}
provides a synthetic analogue functor
\[
\nu\colon \spaces^{R\hbox{{-}}\text{proj}} \to \Syn\spaces_R.
\]
\end{ex}

\begin{ex}[The special fiber]
The comonad $c$ restricted to $\lfrees^\sharp(R)$ induces a comonad $c_1$ on its homotopy category $\h\lfrees^\sharp(R)$, equivalent to the category $\LMod_{R_\ast}^{\proj}$ of projective graded $R_\ast$-modules. By construction, $\CoAlg_{c_1}(\LMod_{R_\ast}^{\proj})$ is the category of \emph{unstable $R_\ast R$-coalgebras} whose underlying $R_\ast$-module is projective; see \cite[\S6--8]{bendersky1978unstable} for further discussion. This comonad may itself be extended to define a comonad $c_{1!}$ on the nonnegative derived $\infty$-category of graded $R_\ast$-modules
\[
\Model_{\h\lfrees^\sharp(R)}\simeq\LMod_{R_\ast}^{\geq 0}.
\]
The $\infty$-category of $c_{1!}$-coalgebras is therefore a \emph{derived $\infty$-category of unstable $R_\ast R$-coalgebras}. In the sequel \cite{usd3}, we will explain how there is a special fiber functor
\[
\tau_!\colon \Syn\spaces_R\to\CoAlg_{c_{1!}}(\LMod_{R_\ast}^{\geq 0}),
\]
factoring as a tower of square-zero extensions, categorifying the unstable Adams spectral sequence in a manner analogous to the corresponding stable tower $\Syn_R\to\cdots\to\Mod_{C(\tau)}(\Syn_R)$.

Note that as $\CoAlg_{c_{1!}}(\LMod_{R_\ast}^{\geq 0})$ is comonadic over $\LMod_{R_\ast}^{\geq 0}$ by construction, it might be thought of as a homotopy theory of \emph{simplicial} unstable coalgebras, although we do not know whether it is (for example) the underlying $\infty$-category of some model structure on the category of simplicial unstable coalgebras. This is by contrast with the homotopy theories of \emph{cosimplicial} unstable coalgebras that have appeared in prior work related to the unstable Adams spectral sequence \cite{bousfield1989homotopy,blanc2001realizing,bousfield2003cosimplicial, biedermannraptisstelzer2017realization}. It would be desirable to have a comparison of these algebraic homotopy theories.
\end{ex}

\begin{rmk}
Suppose that $R$ is connective, and write $\calL_+^\sharp(R)\subset\calL^\sharp(R)$ for the full subcategory of connective objects, itself a loop theory satisfying $\Model_{\smash{\calL_+^\sharp(R)}}^\Omega\simeq \LMod_R^{\geq 0}$ by \cref{prop:loopcompletion}. In this case, the comonad $c$ on $\lfrees^\sharp(R)$ satisfies $c(\calL^\sharp(R))\subset\calL_+^\sharp(R)$, and so the inclusion $\calL_+^\sharp(R)\subset\calL^\sharp(R)$ extends to an equivalence
\[
\Syn\spaces_R\simeq\CoAlg_{D(c)}(\Model_{\calL_+^\sharp(R)}).
\]
This alternate definition has some technical advantages: for example, the comonad $c$ already preserves all geometric realizations on $\LMod_R^{\geq 0}$, and by \cref{prop:loopcompletion} we see that the realization functor is given simply by
\[
\map_{\Syn\spaces_R}(\nu(\ast),\bs)\colon \Syn\spaces_R \to \spaces.
\]
Moreover, this $\infty$-category $\CoAlg_{D(c)}(\Model_{\calL_+^\sharp(R)})$ of $R$-synthetic spaces may be defined under just the assumption that $R_\ast\Omega^\infty\Sigma^n R$ is projective as a left $R_\ast$-module for $n\geq 0$.
\end{rmk}

\subsection{Synthetic \texorpdfstring{$\mathbf{E}_{k}$}{E\_k}-rings} 
\label{ssec:examples:syntheticekrings}

In the previous sections, we saw how one may construct $\infty$-categories of $R$-synthetic spectra and spaces by deriving the adjunctions
\[
\spectra\rightleftarrows \LMod_R,\qquad \spaces\rightleftarrows \LMod_R.
\]
This robust technique can be applied in a variety of situations. In this section, we highlight the example of \emph{$R$-synthetic $\calO$-algebras} for an $\infty$-operad $\calO$. 

For simplicity, let $R$ be a $\bfE_\infty$-ring spectrum for which $R_\ast R$ is flat over $R_\ast$, and let $\calO$ be an $\infty$-operad. In \cref{ex:freeoalgebras}, we described a loop theory $\lfrees^\calO(R)$ satisfying
\[
\Model_{\lfrees^\calO(R)}^\Omega\simeq\Alg_\calO(\Mod_R).
\]
We will construct an $\infty$-category of $R$-synthetic $\calO$-algebras by using this loop theory to derive the free-forgetful adjunction
\[
R\otimes_\thesphere\bs\colon \Alg_\calO\rightleftarrows\Alg_\calO(\Mod_R)\noloc U.
\]
To do so, we must verify some of the technical conditions discussed in \S\ref{subsection:coalgebras_in_models_and_loop_models}. These come from the following.

\begin{lemma}
\label{lem:flatadjointable}
Under the assumption that $R_\ast R$ is flat over $R_\ast$, the diagram
\begin{center}\begin{tikzcd}[column sep=large]
\Model_{\lfrees^\calO(R)}\ar[d,"\Free_\calO^\ast"]\ar[r,"D(R\otimes_\thesphere\bs)"]&\Model_{\lfrees^\calO(R)}\ar[d,"\Free_\calO^\ast"]\\
\Model_{\lfrees(R)}\ar[r,"D(R\otimes_\thesphere\bs)"]&\Model_{\lfrees(R)}
\end{tikzcd}\end{center}
commutes.
\end{lemma}
\begin{proof}
As all functors in this diagram preserve geometric realizations, to make it commute we must only make it commute on restriction to $\lfrees^\calO(R)\subset\Model_{\lfrees^\calO(R)}$. Indeed, if $A \in \lfrees^\calO(R)$ then by construction
\[
\Free_\calO^\ast \simeq D(R\otimes_\thesphere\bs)(R\otimes_\thesphere\bs)(\nu_{\lfrees^\calO(R)} A) = \Free_\calO^\ast \nu_{\lfrees^\calO(R)}(R\otimes_\thesphere A) = \nu_{\lfrees(R)}(R\otimes_\thesphere A),
\]
whereas by \cref{lemma:properties_of_the_functor_between_models_of_freesr_induced_by_tensoring_with_bimodule}
we have
\[
D(R\otimes_\thesphere\bs) \Free_\calO^\ast(\nu_{\lfrees^\calO(R)}(A)) = D(R\otimes_\thesphere\bs)(\nu_{\lfrees(R)}(A)) \simeq \nu_{\lfrees(R)}(R\otimes_\thesphere A),
\]
and the comparison map identifies these.
\end{proof}

\begin{prop}
The endofunctor $D(R\otimes_\thesphere\bs)$ of $\Model_{\calL^\calO(R)}$ has the following properties.
\begin{enumerate}
\item It preserves the full subcategory of loop models;
\item The counit $E(D(R\otimes_\thesphere\bs)) \to R\otimes_\thesphere\bs$ is an equivalence.
\item It is \derivable{}, and therefore inherits the structure of a comonad.
\item Every object of $\Alg_\calO(\Mod_R)$ is stably tame with respect to $D(R\otimes_\thesphere\bs)$.
\end{enumerate}
\end{prop}
\begin{proof}
The first claim follows from \cref{lem:flatadjointable}, as $\Free_\calO^\ast$ reflects the property of being a loop model, and the second claim follows from \cref{lemma:a_criterion_for_a_functor_between_loop_models_is_determined_by_its_derived_functor} as $R\otimes_\thesphere\bs\colon \Alg_\calO(\Mod_R)\to\Alg_\calO(\Mod_R)$ preserves geometric realizations.  The third and fourth now follow from \cref{corollary:endofunctors_whose_derived_functor_preserves_loop_models_and_which_are_its_restriction_are_self_tame}
\end{proof}

\begin{defn}
\label{def:syntheticoalgebras}
Let $R$ be an $\bfE_\infty$-ring for which $R_\ast R$ is flat as an $R_\ast$-module, and let $\calO$ be an $\infty$-operad. The $\infty$-category of \emph{$R$-synthetic $\calO$-algebras} is given by
\[
\Syn_{R,\cpl}^{\geq 0}(\Alg_\calO) = \CoAlg_{D(R\otimes_\thesphere\bs)}(\Model_{\lfrees^\calO(R)}).
\]
\end{defn}

We note that when $\calO$ is the trivial $\infty$-operad, the above recovers the $\infty$-category $\Syn_{R,\cpl}^{\geq 0}$ of $R$-synthetic spectra of \cref{rmk:twoconstructionsofsyntheticspectra}. We now discuss how our formalism recovers the expected features of the deformation.

\begin{ex}[The generic fibre]
Postcomposing the generic fibre functor constructed in \cref{proposition:e_and_l_take_coalgebras_to_coalgebras} with the right adjoint $\CoAlg_{R\otimes_\thesphere\bs}(\Alg_\calO(\Mod_R))\to\Alg_\calO$ provides a realization functor
\[
\Syn_R(\Alg_\calO) \to \Alg_\calO.
\]
\end{ex}

\begin{ex}[The special fibre]
Abbreviate $c_! = D(R\otimes_\thesphere\bs)$. As we will explain in \cite{usd3}, this comonad induces in a canonical way a comonad $c_{1!}$ on the special fiber $\Model_{\h\lfrees^\calO(R)}$. By construction, $\CoAlg_{c_{1!}}(\Model_{\h\lfrees^\calO(R)})$ is a suitable derived $\infty$-category of $R_\ast R$-comodules equipped with compatible $\calO$-$R$-power operations. 

For example, if $R = \bfF_p$ and $\calO = \bfE_\infty$, then its full subcategory of discrete objects is exactly equivalent to the category of graded commutative $\bbF_p$-algebras equipped with a compatible coaction by the dual Steenrod algebra and action by the Dyer--Lashof operations, satisfying the usual Cartan formulae, instability conditions, and Nishida relations.
\end{ex}

\begin{ex}[{The forgetful functor}]
\label{example:forgetful_functor_from_synr_algebras_to_synr_spectra}
The forgetful functor $\Alg_\calO\to\Sp$ commutes with the comonads $R\otimes_\thesphere\bs$, and \cref{lem:flatadjointable} shows that its induced derived functor $\Free_\calO^\ast\colon \Model_{\lfrees^\calO(R)}\to\Model_{\lfrees(R)}$ similarly commutes with the derived comonad $D(R\otimes_\thesphere\bs)$. Passing to $\infty$-categories of coalgebras then provides a forgetful functor
\[
U\colon \Syn_{R,\cpl}^{\geq 0}(\Alg_\calO) \to \Syn_{R,\cpl}^{\geq 0}.
\]
To write down $U$ formally, it is convenient to use the language of $(\infty,2)$-categories, and so we defer a formal description to \cite{usd3}. By construction, $U$ fits into a commutative diagram 
\begin{center}
\begin{tikzcd}
\Syn_{R,\cpl}^{\geq 0}(\Alg_\calO)\ar[d,"U"]\ar[r]&\Model_{\calL^\calO(R)}\ar[d,"\Free_\calO^\ast"]\\
\Syn_{R,\cpl}^{\geq 0}\ar[r]&\Model_{\calL(R)}
\end{tikzcd},
\end{center}
where the horizontal arrows forget the coalgeba structure. In particular, since $\Free_\calO^\ast$ preserves geometric realizations and is conservative, the same is true for $U$. 
\end{ex}

\begin{ex}[Synthetic analogues]
We observe that \cref{proposition:induced_derived_functor_construction_gives_a_functor_on_coalgebras} provides a synthetic analogue functor
\[
\nu\colon \Alg_\calO \to \Syn_{R,\cpl}^{\geq 0}(\Alg_\calO)
\]
lifting the synthetic analogue functor for spectra. Together with the forgetful functor $U$ of \cref{example:forgetful_functor_from_synr_algebras_to_synr_spectra}, this fits into a commutative diagram 
\begin{center}
\begin{tikzcd}
\Alg_\calO\ar[r,"\nu"]\ar[d]&\Syn_{R,\cpl}^{\geq 0}(\Alg_\calO)\ar[d,"U"]\\
\Sp\ar[r,"\nu"]\ar[r]&\Syn_{R,\cpl}^{\geq 0}
\end{tikzcd}. 
\end{center}
\end{ex}

\begin{rmk}
There are other ways one might try to deform the $\infty$-category of $\calO$-algebras: 
\begin{enumerate}
\item If $R$ is Adams-type, one can consider the $\infty$-category $\Alg_{\calO}(\Syn_{R})$ of algebras in synthetic spectra. Here, the special fibre is given by
\[
\Alg_{\calO}(\Mod_{C\tau}(\Syn_{R})) \simeq \Alg_{\calO}(\mathrm{Stable}_{R_{*}R}), 
\]
the $\infty$-category of $\mathcal{O}$-algebras in Hovey's stable $\infty$-category of $E_{*}E$-comodules. Informally, this means that the special fibre sees the product on homology, but not the $R$-$\calO$-power operations, and so this deformation is in a sense less structured than deformation introduced in \cref{def:syntheticoalgebras}.
\item In upcoming work, Devalapurkar-Hahn-Raksit-Senger introduce a variant of even $\MU$-synthetic $\bfE_\infty$-rings which is closely related to the $\bfE_\infty$-even filtration of Hahn--Raksit--Wilson \cite{hahnmotivic} and should be the purely homotopy-theoretic equivalent of the normed rings of motivic homotopy theory \cite{bachmann2017norms}. This deformation is more refined than the $\MU$-synthetic $\bfE_\infty$-rings presented here.
\end{enumerate}
\end{rmk}

Other variations are also possible, such as by taking $R$ to instead be a highly structured equivariant or motivic ring spectrum. We just give one additional class of examples that makes essential use of the fact that all of our machinery has been developed for \emph{infinitary} theories.

\begin{ex}\label{ex:kn}
Let $\Gamma$ be a formal group of height $n$ over a perfect field of characteristic $p$, and let $E_n$ denote the corresponding spectrum of Morava $E$-theory, with maximal ideal $\frakm\subset\pi_0 E_n$. If $\calL_\frakm^\wedge(E_n)\subset\Mod_{E_n}^{\Cpl(\frakm)}$ is the full subcategory generated under direct sums of $\Sigma^n E_n$ for $n \in \integers$, then \cref{prop:stableloopmodels} implies that
\[
\Model_{\calL_\frakm^\wedge(E_n)}\simeq\Mod_{E_n}^{\Cpl(\frakm)}.
\]
As a varation of \cref{def:syntheticoalgebras}, deforming the comonad associated to the adjunction
\[
\CAlg(\Sp_{K(n)}) \rightleftarrows \CAlg(\Mod_{E_n}^{\Cpl(\frakm)})
\]
produces an $\infty$-category $\Syn_{E_n}(\CAlg(\Sp_{K(n)}))$ of synthetic $K(n)$-local $\bfE_\infty$-rings, with generic fibre exactly equivalent to $\CAlg(\Sp_{K(n)})$.

Our work in the sequel \cite{usd3} will show that $\Syn_{E_n}(\CAlg(\Sp_{K(n)}))$ admits a well-behaved special fiber which houses the obstruction groups for various spectral sequences and obstruction theories internal to $\Syn_{E_n}(\CAlg(\Sp_{K(n)}))$. In general, $\Syn_{E_n}(\CAlg(\Sp_{K(n)}))$ acts as a categorified extension of the spectral sequence for maps between $K(1)$-local $\bfE_\infty$-rings originally constructed by Goerss--Hopkins \cite[Theorem 2.4.14]{moduli_problems_for_structured_ring_spectra}; when $n=1$, this extension removes the technical $p$-completeness assumption needed there.

Informally, the special fiber of $\Syn_{E_n}(\CAlg(\Sp_{K(n)}))$ is a derived $\infty$-category of $\frakm$-complete $\bbT$-algebras in the sense of Rezk \cite{rezk2009congruence} equipped with a suitably compatible and continuous action by the Morava stabilizer group. Heuristically, such objects consist of $\frakm$-complete alternating $E_\ast$-algebras equipped with an action by isogenies of formal groups deforming $\Gamma$ (subject to the congruence criterion of \cite{rezk2009congruence}). Thus the deformation $\Syn_{E_n}(\CAlg(\Sp_{K(n)}))$ provides one answer to a conjecture of Lawson \cite[Conjecture 1.6.6]{lawson2020en} on an obstruction theory for maps between $K(n)$-local $\bfE_\infty$-rings.

By instead deforming the comonad associated to the adjunction
\[
\Lie(\Sp_{K(n)}) \rightleftarrows\Lie(\Mod_{E_n}^{\Cpl(\frakm)}),
\]
one may obtain an $\infty$-category of synthetic $K(n)$-local spectral Lie algebras, which by work of Heuts \cite{heuts2021lie} can be interpreted as an $\infty$-category of synthetic (non-telescopic) $v_n$-periodic spaces. Its special fiber is a derived $\infty$-category of Brantner's Hecke Lie algebras \cite{brantner2017lubin} equipped with suitably compatible and continuous action by the Morava stabilizer group.
\end{ex}

\begingroup
\raggedright
\bibliographystyle{amsalpha}
\bibliography{bibliography}

\providecommand{\bysame}{\leavevmode\hbox to3em{\hrulefill}\thinspace}
\providecommand{\MR}{\relax\ifhmode\unskip\space\fi MR }
\providecommand{\MRhref}[2]{%
  \href{http://www.ams.org/mathscinet-getitem?mr=#1}{#2}
}
\providecommand{\href}[2]{#2}
\begin{thebibliography}{EKMM07}

\bibitem[Ada69]{adams1969lectures}
J.~F. Adams, \emph{Lectures on generalised cohomology}, Category {T}heory,
  {H}omology {T}heory and their {A}pplications, {III} ({B}attelle {I}nstitute
  {C}onference, {S}eattle, {W}ash., 1968, {V}ol. {T}hree), Lecture Notes in
  Math., vol. No. 99, Springer, Berlin-New York, 1969, pp.~1--138. \MR{251716}

\bibitem[AP25]{annala2025note}
Toni {Annala} and Piotr {Pstr{\k{a}}gowski}, \emph{{A note on weight
  filtrations at the characteristic}}, arXiv e-prints (2025), arXiv:2502.19626.

\bibitem[Bad02]{badzioch2002algebraic}
Bernard Badzioch, \emph{Algebraic theories in homotopy theory}, Ann. of Math.
  (2) \textbf{155} (2002), no.~3, 895--913. \MR{1923968}

\bibitem[Bal23]{balderrama2023algebraic}
William Balderrama, \emph{Algebraic theories of power operations}, J. Topol.
  \textbf{16} (2023), no.~4, 1543--1640 (English).

\bibitem[Bal25]{balderrama2021deformations}
\bysame, \emph{Deformations of homotopy theories via algebraic theories}, 2025,
  pp.~Paper No. 110496, 76. \MR{4950976}

\bibitem[BB04]{borceuxbourn2004malcev}
Francis Borceux and Dominique Bourn, \emph{Mal'cev, protomodular, homological
  and semi-abelian categories}, Mathematics and its Applications, vol. 566,
  Kluwer Academic Publishers, Dordrecht, 2004. \MR{2044291}

\bibitem[BCM78]{bendersky1978unstable}
Martin Bendersky, Edward~B Curtis, and Haynes~R Miller, \emph{The unstable
  {A}dams spectral sequence for generalized homology}, Topology \textbf{17}
  (1978), no.~3, 229--248.

\bibitem[BCN25]{brantner2021pd}
D.~Lukas~B. Brantner, Ricardo Campos, and Joost Nuiten, \emph{P{D} {O}perads
  and {E}xplicit {P}artition {L}ie {A}lgebras}, Mem. Amer. Math. Soc.
  \textbf{315} (2025), no.~1597, v+125. \MR{5003477}

\bibitem[BDG04]{realization_space_of_a_pi_algebra}
D.~Blanc, W.~G. Dwyer, and P.~G. Goerss, \emph{The realization space of a
  {$\Pi$}-algebra: a moduli problem in algebraic topology}, Topology
  \textbf{43} (2004), no.~4, 857--892.

\bibitem[Ber06]{bergner_rigidification_of_algebras}
Julia~E. Bergner, \emph{Rigidification of algebras over multi-sorted theories},
  Algebr. Geom. Topol. \textbf{6} (2006), 1925--1955. \MR{2263055
  (2007f:18005)}

\bibitem[BF78]{bousfieldfriedlander1978homotopy}
A.~K. Bousfield and E.~M. Friedlander, \emph{Homotopy theory of
  {{\(\Gamma\)}}-spaces, spectra, and bisimplicial sets}, Geom. {Appl}.
  {Homotopy} {Theory}, {II}, {Proc}. {Conf}., {Evanston} 1977, {Lect}. {Notes}
  {Math}. 658, 80-130 (1978)., 1978.

\bibitem[BGvO71]{barrgrilletosdol1971exact}
Michael Barr, Pierre~A. Grillet, and Donovan~H. van Osdol, \emph{Exact
  categories and categories of sheaves}, Lecture Notes in Mathematics, vol.
  236, Springer-Verlag, Berlin, 1971. \MR{3727441}

\bibitem[BH21]{bachmann2017norms}
Tom Bachmann and Marc Hoyois, \emph{Norms in motivic homotopy theory},
  Ast\'erisque (2021), no.~425, ix+207. \MR{4288071}

\bibitem[Bla01]{blanc2001realizing}
David Blanc, \emph{Realizing coalgebras over the {Steenrod} algebra}, Topology
  \textbf{40} (2001), no.~5, 993--1016 (English).

\bibitem[Bou89]{bousfield1989homotopy}
A.~K. Bousfield, \emph{Homotopy spectral sequences and obstructions}, Israel J.
  Math. \textbf{66} (1989), no.~1-3, 54--104. \MR{1017155}

\bibitem[Bou03]{bousfield2003cosimplicial}
AK~Bousfield, \emph{Cosimplicial resolutions and homotopy spectral sequences in
  model categories}, Geometry \& Topology \textbf{7} (2003), no.~2, 1001--1053.

\bibitem[Bou17]{bourn2017groups}
Dominique Bourn, \emph{From groups to categorial algebra}, Compact Textbooks in
  Mathematics, Birkh\"auser/Springer, Cham, 2017, Introduction to protomodular
  and Mal'tsev categories. \MR{3674493}

\bibitem[BPa]{usd2}
William Balderrama and Piotr Pstr\k{a}gowski, \emph{Unstable synthetic
  deformations {II}: Infinitesimal extensions}.

\bibitem[BPb]{usd3}
\bysame, \emph{Unstable synthetic deformations {III}: Naturality of the spiral
  tower}.

\bibitem[Bra17]{brantner2017lubin}
Lukas Brantner, \emph{{The Lubin-Tate Theory of Spectral Lie Algebras}},
  Ph.{D}. thesis, Harvard University, 2017, urn-3:HUL.InstRepos:41140243.

\bibitem[BRS17]{biedermannraptisstelzer2017realization}
Georg Biedermann, Georgios Raptis, and Manfred Stelzer, \emph{The realization
  space of an unstable coalgebra}, Ast{\'e}risque, vol. 393, Paris:
  Soci{\'e}t{\'e} Math{\'e}matique de France (SMF), 2017 (English).

\bibitem[BT00]{benderskythompson2000bousfield}
Martin Bendersky and Robert~D. Thompson, \emph{The {Bousfield}-{Kan} spectral
  sequence for periodic homology theories}, Am. J. Math. \textbf{122} (2000),
  no.~3, 599--635 (English).

\bibitem[CKP93]{carbonikellypedicchio}
A.~Carboni, G.~M. Kelly, and M.~C. Pedicchio, \emph{Some remarks on {M}al' tsev
  and {G}oursat categories}, Appl. Categ. Structures \textbf{1} (1993), no.~4,
  385--421. \MR{1268510}

\bibitem[DK89]{dwyerkan1989enveloping}
W.~G. Dwyer and D.~M. Kan, \emph{The enveloping ring of a {{\(\Pi\)}}-algebra},
  Advances in homotopy theory, {Proc}. {Conf}. in {Honour} of {I}.{M}. {James},
  {Cortona}/{Italy} 1988, {Lond}. {Math}. {Soc}. {Lect}. {Note} {Ser}. 139,
  49-60 (1989)., 1989.

\bibitem[DKS93]{dwyerkanstover1993e2}
W.~G. Dwyer, D.~M. Kan, and C.~R. Stover, \emph{An {{\({E}^ 2\)}} model
  category structure for pointed simplicial spaces}, J. Pure Appl. Algebra
  \textbf{90} (1993), no.~2, 137--152 (English).

\bibitem[DP61]{doldpuppe1961homologie}
Albrecht Dold and Dieter Puppe, \emph{Homologie nicht-additiver {F}unktoren.
  {A}nwendungen}, Ann. Inst. Fourier (Grenoble) \textbf{11} (1961), 201--312.
  \MR{150183}

\bibitem[EKMM07]{elmendorf2007rings}
Anthony~D Elmendorf, Igor Kriz, Michael~A Mandell, and J~Peter May,
  \emph{Rings, modules, and algebras in stable homotopy theory}, no.~47,
  American Mathematical Soc., 2007.

\bibitem[Fin60]{findlay1960reflexive}
G.~D. Findlay, \emph{Reflexive homomorphic relations}, Canad. Math. Bull.
  \textbf{3} (1960), 131--132. \MR{124251}

\bibitem[GH04]{moduli_spaces_of_commutative_ring_spectra}
P.~G. Goerss and M.~J. Hopkins, \emph{Moduli spaces of commutative ring
  spectra}, Structured ring spectra, London Math. Soc. Lecture Note Ser., vol.
  315, Cambridge Univ. Press, Cambridge, 2004, pp.~151--200.

\bibitem[GH05]{moduli_problems_for_structured_ring_spectra}
P.~G. Goerss and M.~J. Hopkins, \emph{Moduli problems for structured ring
  spectra}, 2005.

\bibitem[GIKR22]{gheorghe2022c}
Bogdan Gheorghe, Daniel~C Isaksen, Achim Krause, and Nicolas Ricka,
  \emph{{$\mathbb{C}$}-motivic modular forms}, Journal of the European
  Mathematical Society \textbf{24} (2022), no.~10, 3597--3628.

\bibitem[Hen08]{henriques2008integrating}
Andr{\'e} Henriques, \emph{Integrating {{\(L_\infty\)}}-algebras}, Compos.
  Math. \textbf{144} (2008), no.~4, 1017--1045 (English).

\bibitem[Heu21]{heuts2021lie}
Gijs Heuts, \emph{Lie algebras and {{\(v_n\)}}-periodic spaces}, Ann. Math. (2)
  \textbf{193} (2021), no.~1, 223--301 (English).

\bibitem[HHLN23]{haugsenghebestreitlinskensnuiten2023lax}
Rune Haugseng, Fabian Hebestreit, Sil Linskens, and Joost Nuiten, \emph{Lax
  monoidal adjunctions, two-variable fibrations and the calculus of mates},
  Proc. Lond. Math. Soc. (3) \textbf{127} (2023), no.~4, 889--957. \MR{4655344}

\bibitem[HL17]{hollingslawson2017wagner}
Christopher~D. Hollings and Mark~V. Lawson, \emph{Wagner's theory of
  generalised heaps}, Springer, Cham, 2017. \MR{3729305}

\bibitem[HRW]{hahnmotivic}
Jeremy Hahn, Arpon Raksit, and Dylan Wilson, \emph{A motivic filtration on the
  topological cyclic homology of commutative ring spectra}, Annals of
  Mathematics.

\bibitem[JP02]{jibladzepirashvili2002kan}
M.~Jibladze and T.~Pirashvili, \emph{On {K}an fibrations for {M}altsev
  algebras}, Georgian Math. J. \textbf{9} (2002), no.~1, 71--74. \MR{1916490}

\bibitem[Kla01]{klaus2001simplicial}
Stephan Klaus, \emph{On simplicial loops and {$H$}-spaces}, Topology Appl.
  \textbf{112} (2001), no.~3, 337--348. \MR{1824167}

\bibitem[Lam55]{lambek1955groups}
J.~Lambek, \emph{Groups and herds}, Bull. Amer. Math. Soc. \textbf{61} (1955),
  78.

\bibitem[Lam92]{lambek1992ubiquity}
\bysame, \emph{On the ubiquity of {M}al{'}cev operations}, Proceedings of the
  {I}nternational {C}onference on {A}lgebra, {P}art 3 ({N}ovosibirsk, 1989),
  Contemp. Math., vol. 131, Amer. Math. Soc., Providence, RI, 1992,
  pp.~135--146. \MR{1175879}

\bibitem[Law63]{lawvere1963functorial}
F~William Lawvere, \emph{Functorial semantics of algebraic theories},
  Proceedings of the National Academy of Sciences \textbf{50} (1963), no.~5,
  869--872.

\bibitem[Law19]{lawson2019calculating}
Tyler Lawson, \emph{Calculating obstruction groups for {$E_\infty$} ring
  spectra}, Homotopy theory: tools and applications, Contemp. Math., vol. 729,
  Amer. Math. Soc., [Providence], RI, [2019] \copyright 2019, pp.~179--203.
  \MR{3959600}

\bibitem[Law20]{lawson2020en}
\bysame, \emph{{{\(E_n\)}}-spectra and {Dyer}-{Lashof} operations}, Handbook of
  homotopy theory, Boca Raton, FL: CRC Press, 2020, pp.~793--849 (English).

\bibitem[Lin69]{linton1969coequalizers}
F.~E.~J. Linton, \emph{Coequalizers in categories of algebras}, Sem. on
  {T}riples and {C}ategorical {H}omology {T}heory ({ETH}, {Z}\"urich, 1966/67),
  Lecture Notes in Math., vol. No. 80, Springer, Berlin-New York, 1969,
  pp.~75--90. \MR{244341}

\bibitem[LS10]{luptonsmith2010whitehead}
Gregory Lupton and Samuel~Bruce Smith, \emph{Whitehead products in function
  spaces: {Q}uillen model formulae}, J. Math. Soc. Japan \textbf{62} (2010),
  no.~1, 49--81. \MR{2648216}

\bibitem[Lura]{higher_algebra}
Jacob Lurie, \emph{Higher algebra},
  \url{http://www.math.harvard.edu/~lurie/papers/HA.pdf}.

\bibitem[Lurb]{lurie_spectral_algebraic_geometry}
\bysame, \emph{Spectral algebraic geometry},
  \url{http://www.math.harvard.edu/~lurie/papers/SAG-rootfile.pdf}.

\bibitem[Lur09]{lurie_higher_topos_theory}
\bysame, \emph{Higher topos theory}, Annals of Mathematics Studies, vol. 170,
  Princeton University Press, Princeton, NJ, 2009.

\bibitem[Lur11]{lurie2011dag8}
Jacob Lurie, \emph{Derived algebraic geometry viii: Quasi-coherent sheaves and
  tannaka duality theorems}, 2011.

\bibitem[Lur15]{lurie2015rotation}
\bysame, \emph{Rotation invariance in algebraic {K}-theory},
  \url{https://www.math.ias.edu/~lurie/papers/Waldhaus.pdf}, 2015.

\bibitem[Lur17]{lurie2017higheralgebra}
\bysame, \emph{Higher algebra},
  \url{http://www.math.ias.edu/~lurie/papers/HA.pdf}, 2017.

\bibitem[Mal54]{malcev1954general}
A.~I. Mal'cev, \emph{On the general theory of algebraic systems}, Mat. Sb.
  (N.S.) \textbf{35(77)} (1954), 3--20. \MR{65533}

\bibitem[Moo55]{moore1954homotopie}
J.~C. Moore, \emph{Homotopie des complexes monoïdaux, i}, Séminaire Henri
  Cartan \textbf{7} (1954-1955), no.~2, 1--8 (fre).

\bibitem[MP89]{makkaipare1989accessible}
Michael Makkai and Robert Par{\'e}, \emph{Accessible categories: {The}
  foundations of categorical model theory}, Contemp. Math., vol. 104,
  Providence, RI: American Mathematical Society, 1989 (English).

\bibitem[Por65]{porter1965spaces}
G.~J. Porter, \emph{Spaces with vanishing {W}hitehead products}, Quart. J.
  Math. Oxford Ser. (2) \textbf{16} (1965), 77--84. \MR{172292}

\bibitem[PP21]{patchkoria2021adams}
Irakli {Patchkoria} and Piotr {Pstr{\k{a}}gowski}, \emph{{Adams spectral
  sequences and Franke's algebraicity conjecture}}, arXiv e-prints (2021),
  arXiv:2110.03669.

\bibitem[Pr{\"u}24]{prufer1924theorie}
H.~Pr{\"u}fer, \emph{Theorie der {Abelschen} {Gruppen}. {I}.
  {Grundeigenschaften}.}, Math. Z. \textbf{20} (1924), 165--187 (German).

\bibitem[Pst23a]{pstrkagowski2018synthetic}
Piotr Pstr{\k{a}}gowski, \emph{Synthetic spectra and the cellular motivic
  category}, Invent. Math. \textbf{232} (2023), no.~2, 553--681 (English).

\bibitem[Pst23b]{pstrkagowski2023moduli}
Piotr Pstrągowski, \emph{Moduli of spaces with prescribed homotopy groups},
  Journal of Pure and Applied Algebra \textbf{227} (2023), no.~10, 107409.

\bibitem[Pst25]{pstrkagowski2025perfect}
Piotr Pstr{\k{a}}gowski, \emph{Perfect even modules and the even filtration},
  Journal of the European Mathematical Society (2025).

\bibitem[PV22]{abstract_gh_theory}
Piotr Pstr{\k{a}}gowski and Paul VanKoughnett, \emph{Abstract
  {G}oerss-{H}opkins theory}, 2022, pp.~Paper No. 108098, 51. \MR{4363589}

\bibitem[Qui67]{quillen1967homotopical}
Daniel~G. Quillen, \emph{Homotopical algebra}, Lecture Notes in Mathematics,
  vol. No. 43, Springer-Verlag, Berlin-New York, 1967. \MR{223432}

\bibitem[Rak]{arpon_spectral_algebraic_theories}
Arpon Raksit, \emph{Spectral algebraic theories {I}: Derived categories}, To
  appear.

\bibitem[Rez]{rezk_realization_fibration}
Charles Rezk, \emph{When are homotopy colimits compatible with homotopy base
  change?},
  \url{http://www.math.uiuc.edu/~rezk/i-hate-the-pi-star-kan-condition.pdf}.

\bibitem[Rez09]{rezk2009congruence}
Charles Rezk, \emph{The congruence criterion for power operations in {M}orava
  {$E$}-theory}, Homology Homotopy Appl. \textbf{11} (2009), no.~2, 327--379.
  \MR{2591924}

\bibitem[Sey80]{seymour1980kan}
R.~M. Seymour, \emph{Kan fibrations in the category of simplicial spaces},
  Fundam. Math. \textbf{106} (1980), 141--152 (English).

\bibitem[Sha74]{shafaat1974note}
Ahmad Shafaat, \emph{A note on {M}al'cevian varieties}, Canad. Math. Bull.
  \textbf{17} (1974), no.~4, 609. \MR{382125}

\bibitem[Smi76]{smith1976malcev}
Jonathan D.~H. Smith, \emph{Mal'cev varieties}, Lecture Notes in Mathematics,
  vol. Vol. 554, Springer-Verlag, Berlin-New York, 1976. \MR{432511}

\bibitem[Sto90]{stover1990vankampen}
Christopher~R. Stover, \emph{A van {Kampen} spectral sequence for higher
  homotopy groups}, Topology \textbf{29} (1990), no.~1, 9--26 (English).

\bibitem[Vag53]{wagner1953theory}
V.~V. Vagner, \emph{The theory of generalized heaps and generalized groups},
  Mat. Sbornik N.S. \textbf{32/74} (1953), 545--632. \MR{59267}

\bibitem[{van}25]{van2025introduction}
Sven {van Nigtevecht}, \emph{{An introduction to filtered and synthetic
  spectra}}, arXiv e-prints (2025), arXiv:2509.21127.

\bibitem[Vr13]{vokrinek2013computing}
Luk\'a\v~s Vok\v~r\'inek, \emph{Computing the abelian heap of unpointed stable
  homotopy classes of maps}, Arch. Math. (Brno) \textbf{49} (2013), no.~5,
  359--368. \MR{3159334}

\bibitem[Vr14]{vokrinek2014heaps}
\bysame, \emph{Heaps and unpointed stable homotopy theory}, Arch. Math. (Brno)
  \textbf{50} (2014), no.~5, 323--332. \MR{3303781}

\bibitem[Wil73]{wilson1973omega}
W~Stephen Wilson, \emph{The {$\Omega$}-spectrum for {B}rown-{P}eterson
  cohomology. {P}art {I}}, Commentarii Mathematici Helvetici \textbf{48}
  (1973), no.~1, 45--55.

\bibitem[Wil75]{wilson1975omega}
\bysame, \emph{The Ω-spectrum for {B}rown-{P}eterson cohomology part ii},
  American Journal of Mathematics \textbf{97} (1975), no.~1, 101--123.

\bibitem[Wra70]{wraith1969algebraic}
G.~C. Wraith, \emph{Algebraic theories}, Lectures Autumn 1969. Lecture Notes
  Series, No. 22, Matematisk Institut, Aarhus Universitet, Aarhus, 1970.
  \MR{0262334}

\end{thebibliography}
\endgroup

\end{document}